    \documentclass[graybox]{SNmult}

   \usepackage{type1cm}        
    \usepackage{multicol}        
    \usepackage[bottom]{footmisc}

   \usepackage{newtxtext}       %
   \usepackage[varvw]{newtxmath}       
   \usepackage{makeidx}    
   \makeindex             

\usepackage[   paper=a4paper,   left=80pt,  right=80pt,   top=86pt,  bottom=56pt,  bindingoffset=0mm,  includefoot,  includehead,  foot=16.5pt,   head=16.5pt ]{geometry}

\usepackage{graphicx}                       
\usepackage{amsmath} 
\usepackage{amsthm}
\usepackage{nicefrac}
\definecolor{linkcolor}{HTML}{0D6A9E}
\definecolor{3blue}{HTML}{0072B2}
\definecolor{3green}{HTML}{009E73}
\definecolor{3ochre}{HTML}{E69F00}
\definecolor{3yellow}{HTML}{F0E442}
\definecolor{3cyan}{HTML}{56B4E9}
\definecolor{3red}{HTML}{D55E00}
\definecolor{3pink}{HTML}{CC79A7}
\definecolor{2blue}{HTML}{1A85FF}
\definecolor{2red}{HTML}{D41159}
\usepackage{graphicx}
\graphicspath{ {./gfx/} }
\usepackage{caption}
\usepackage{tensor}
\usepackage{tikz-cd} 
\tikzcdset{arrow style=tikz} 

\usepackage{hyperref}
\hypersetup{
	colorlinks=true,
	linktoc=all,
	linkcolor=linkcolor,
	citecolor=linkcolor,
	filecolor=magenta,      
	urlcolor=linkcolor,
	pdftitle={Chewing gums, snakes and candle cakes.},
	pdfpagemode=FullScreen,
}
\usepackage{cleveref}
\newtheoremstyle{komait}
{\topsep}   
{\topsep}   
{\itshape}  
{0pt}       
{\bfseries\sffamily} 
{.}         
{5pt plus 1pt minus 1pt} 
{}          
\newtheoremstyle{komanormal}
{\topsep}   
{\topsep}   
{\rmfamily}  
{0pt}       
{\bfseries\sffamily} 
{.}         
{5pt plus 1pt minus 1pt} 
{}          
\theoremstyle{komait}

\newcommand{\tcr}{\textcolor{red}}
\newcommand{\tcb}{\textcolor{blue}}

\newcommand{\PGL}{\mathsf{PGL}}

\newcommand{\integer}{\mathbb{Z}}

\newcommand{\real}{\mathbb{R}}

\newcommand{\im}{\text{Im}}
\newcommand{\opn}{\operatorname}

\begin{document}


\makeatletter
 \renewcommand\normalsize{\@setfontsize\normalsize{11pt}{13.6pt}}
 \makeatother
 \normalsize


\title*{Chewing gums, snakes and candle cakes.}
\author{Benedetta Facciotti\orcidID{0009-0006-4111-3193}\\ Marta Mazzocco \orcidID{0000-0001-9917-2547}\\ Nikita Nikolaev\orcidID{0000-0002-7042-0479}}
\institute{Benedetta Facciotti \at UPC, Av. Dr. Marañón 44-50,  Barcelona, \email{benedetta.facciotti@upc.edu}
\and Marta Mazzocco \at ICREA and UPC, Av. Dr. Marañón 44-50, Barcelona, \email{marta.mazzocco@upc.edu} 
\and Nikita Nikolaev \at University of Birmingham, Watson Building,  Edgbaston, \email{n.nikolaev@bham.ac.uk}}
%
%
\maketitle

\abstract*{The aim of these lecture notes, based on lectures given by the second author at the CIME school, is to illustrate a range of ideas surrounding higher Teichm\"uller spaces of Riemann surfaces with marked boundaries through explicit and computationally tractable examples. After reviewing the classical Teichm\"uller space of hyperbolic Riemann surfaces with boundary and its combinatorial description in terms of Thurston shear coordinates on a ribbon graph, we explain how the bordered cusped Teichm\"uller space arises as a confluent limit when two boundary components in the Riemann surface collide via the so-called chewing-gum move giving rise to a candle cake. We then revisit these constructions from the Fock--Goncharov perspective, explaining snake calculus for transport matrices in $\mathbb P\textnormal{SL}_n(\mathbb R)$ and explain how the chewing-gum move is the inverse of amalgamation. Rather than focusing on formal proofs, our goal is to illustrate the underlying theorems and constructions in a concrete and intuitive way.
}

\abstract{The aim of these lecture notes, based on lectures given by Marta Mazzocco at the CIME school in Cetraro, is to illustrate a range of ideas surrounding higher Teichm\"uller spaces of Riemann surfaces with marked boundaries through explicit and computationally tractable examples. After reviewing the classical Teichm\"uller space of hyperbolic Riemann surfaces with boundary and its combinatorial description in terms of Thurston shear coordinates on a ribbon graph, we explain how the bordered cusped Teichm\"uller space arises as a confluent limit when two boundary components collide via a chewing-gum move. We then revisit these constructions from the Fock–Goncharov perspective, explaining snake calculus for transport matrices in $\mathbb P\textnormal{SL}_n(\mathbb R)$ and recover the bordered cusped theory as the inverse of amalgamation. Rather than focusing on formal proofs, our goal is to illustrate the underlying theorems and constructions in a concrete and intuitive way.
}

\section{Introduction}

In these lecture notes a \textbf{surface with boundary} is a compact, connected, oriented, smooth real two-dimensional
manifold with finitely many
boundary components each of which is diffeomorphic to a circle. 

Given a surface with boundary 
$S$, we may ask ourselves in how many inequivalent ways we can turn it into a Riemann surface with boundary (see beginning of Section~\ref{se:RS}). More precisely, in how many ways we can choose a pair $(X,f)$ where $X$ is a Riemann surface with boundary and $f: S\to X$ is an orientation preserving homeomorphism. The set of all these choices is called the moduli space of surfaces with boundary. 
The \textbf{Teichm\"uller space} $\mathcal T(S)$ is the
moduli space of complex structures on $S$ modulo diffeomorphisms isotopic to the
identity. Namely, the points in the Teichm\"uller space are equivalence classes of pairs $(X,f)$, where two pairs $(X_1,f_1)$, $(X_2,f_2)$ are equivalent iff there exists a bi-holomorphic map
$$
h: X_1\to X_2, \quad \text{such that}\quad h\circ f_1= f_2.
$$
When $S$ has negative Euler characteristic, 
the Teichm\"uller space is equivalently defined as
$$
\mathcal T(S)=\opn{Hom}'(\pi_1(S),\mathbb P\textnormal{SL}_2(\mathbb R))/\mathbb P\textnormal{SL}_2(\mathbb R),
$$
where $\opn{Hom}'$ denotes 
the space of discrete faithful representations.

This manifestation of the Teichm\"uller space is particularly useful because it allows on one side a combinatorial coordinatization in terms of Thurston shear coordinates \cite{Penn1,Fock1,Kashaev,ChF1}, and, on the other side,
provides a link with a very rich generalization in which $\mathbb P\textnormal{SL}_2(\mathbb R)$ is replaced by a simple real Lie group $G$ of
higher rank. Based on the work of many  mathematicians \cite{hit,lab,FockG,Gold,BIW},
in \cite{Wien}, \textbf{higher Teichm\"uller spaces} $\mathcal T(S,G)$ are defined as connected components of the
representation variety 
$$
\opn{Hom}'(\pi_1(S), G)/G
$$
consisting entirely of discrete and faithful representations. 

In these lecture notes, we discuss the combinatorial description of the Teichm\"uller space in the special case of the Riemann sphere with three boundaries $S_{0,3}$ and show how such combinatorial description links to the so-called Fock--Goncharov variables. We will then generalize this example by replacing the group $\mathbb P\textnormal{SL}_2(\mathbb R)$ with $\mathbb P\textnormal{SL}_3(\mathbb R)$, hence moving to the higher Teichm\"uller space. 

Another generalization of the notion of Teichm\"uller space is to allow marked points on the boundary of $S$. We call \textbf{surface with marked boundary} a pair $(S,P)$ where $S$ is a surface with boundary and
$P$ is a finite subset $P \subset \partial S$ of marked points, defined up to isotopy on the boundary $\partial S$ - this means
marked points can slide continuously along the boundary without crossing each other. In  \cite{ChM}, we introduced the notion of \textbf{bordered cusped Teichm\"uller space}
$$
\opn{Hom}'(\pi_1(S,P), \mathbb P\textnormal{SL}_2(\mathbb R))/U_P,
$$
where $U_P:=\sqcup_{p\in P} U_p$, $U_p$ being unipotent radicals of Borel subgroups in $\mathbb P\textnormal{SL}_2(\mathbb R)$. We showed that the bordered cusped Teichm\"uller space arises
as a \textit{confluent} version of the standard Teichm\"uller space when two boundaries collide giving rise to a merged boundary with marked points on it.
In these notes we will show how this confluence works when colliding two holes in a Riemann sphere with three boundaries $S_{0,3}$ to produce a cylinder with two marked points on one boundary. 

The bordered cusped Teichm\"uller space admits a generalization to $\mathbb P\textnormal{SL}_n(\mathbb C)$, which was introduced in \cite{CMR,CMR1} and called \textbf{decorated character variety}. This is defined as a connected component consisting entirely of discrete and faithful representations of
$$
\opn{Hom}'(\pi_1(S,P), \mathbb P\textnormal{SL}_n(\mathbb C))/U_P,
$$
where $U_P:=\sqcup_{p\in P} U_p$, $U_p$ being unipotent  radicals of Borel subgroups in $\mathbb P\textnormal{SL}_n(\mathbb C)$.
In \cite{FMN}, we showed that the coordinate ring of the decorated character variety
can be identified with the \textbf{moduli space
	of pinnings} introduced by Goncharov and Shen \cite{GS}. 

In these lecture notes, we will explain this result in the case of the cylinder with two marked points on one boundary for the groups $\mathbb P \textnormal{SL}_2(\mathbb C)$ and $\mathbb P \textnormal{SL}_3(\mathbb C)$.

These lecture notes use an intentionally informal style to make the ideas easier to follow. Likewise, in all of the next sections, references are chosen for their clarity of presentation rather than for being the original sources of the results. We have also designed examples to make the text more accessible, and we encourage readers who find the theoretical sections challenging to look at the examples first and then return to the theory.

\medskip

\noindent \textbf{Acknowledgments.} The authors thank Leonya Chekhov and Davide Dal Martello for many helpful discussions. Sections~\ref{se:teich} and \ref{se:confl} are based on joint work by M.M. with L. Chekhov \cite{ChM} and with V. Rubtsov \cite{CMR}. Section~\ref{se:FG} was developed by all three authors jointly and is mostly based on the papers of Fock--Goncharov and Shen \cite{FockG}, \cite{GS} and on their interpretation as presented in \cite{DMM} (in particular, Sections~\ref{suse:proj-bases} and \ref{suse:calc} in these notes are  taken from \cite[Section 3.1.1]{DMM}).
This work was funded by 
the Leverhulme Trust Research Project Grant RPG-2021-047 and by the Proyecto de Generaci\'on de Conocimiento, PID2024-155686NB-I00 of the Spanish Ministry of Science, Innovation and Universities.

\section{Riemann surfaces with boundary}\label{se:RS}

In what follows, we will denote by $\mathbb H$ and $\partial \mathbb H$ the hyperbolic plane and its boundary in $\mathbb P^1$ respectively, i.e.
\begin{align*}
	\mathbb H &:= \{z \in \mathbb C| \operatorname{Im}(z) > 0\},\\
	\partial \mathbb H &:= \{z \in\mathbb C|\operatorname{Im}(z) = 0\} \cup \{ \infty \}.
\end{align*}
The upper half plane $\mathbb H$  is endowed with the hyperbolic metric (see Appendix \ref{app:hyp-geo}).

A Riemann surface is a one dimensional holomorphic manifold and, for the purpose of these lecture notes, a Riemann surface with boundary is a Riemann surface from which we have removed a finite disjoint union of open discs. More formally,
a Riemann surface with boundary, $\Sigma$, is a pair $(S, (U_i,\phi_i)_{i\in I})$ where $S$ is a connected, oriented, smooth 2-manifold with boundary $\partial S,$ and $(U_i,\phi_i)_{i\in I}$ is a collection of holomorphically compatible charts on $S$ of two types
\begin{itemize}
	\item an interior chart is a bijection $\phi_i:U_i\to V_i$, where 
	$U_i$ is a  set in the interior of $S$, and
	$V_i$ is an open subset of $\mathbb C$;
	\item a boundary chart is a bijection
	$\phi_i:U_i\to V_i$, where $U_i$ is a  set in $S$ and $V_i$ is an open set in $\overline{\mathbb H}\subset\mathbb C$ in the subset topology.
\end{itemize}
The complex structure on $\Sigma$ is by definition the one given by the charts $(U_i,\phi_i)_{i\in I}$.

Let us denote by $\Sigma_{g,s}$ a Riemann surface of genus $g$ with $s$ boundaries. 
Given a point $x_0\in\Sigma$, the universal cover $\widetilde\Sigma_{g,s}$ of $\Sigma_{g,s}$ based at $x_0$ is 
the set of homotopy classes of paths with starting point $x_0$. This set is automatically a simply connected 
space with a covering map 
$p:\widetilde\Sigma_{g,s}\to \Sigma_{g,s}$ defined as 
$p([\gamma])=\gamma(1)$. Thanks to the covering map, one can identify homotopy classes in 
$\widetilde\Sigma_{g,s}$ with their end points, hence $\widetilde\Sigma_{g,s}$ is endowed with a Riemann surface structure.

The covering map $p$ coincides with the quotient map $\widetilde\Sigma_{g,s}\to \widetilde\Sigma_{g,s}\slash_{\textrm{Deck}(\Sigma_{g,s})}$, 
where ${\textrm{Deck}}(\Sigma_{g,s})$ is the group of deck transformations. 
The group ${\textrm{Deck}}(\Sigma_{g,s})$  acts freely and properly discontinuously, and is isomorphic to the fundamental group of $\Sigma_{g,s}$  
\begin{equation}
	\label{eq:fund-g}
	\pi_1(\Sigma_{g,s},x_0)=\langle a_j,b_j,c_k | a_1b_1a_1^{-1} b_1^{-1} \dots  a_g b_g a_g^{-1} b_g^{-1}c_1 \dots c_s\simeq x_0\rangle_{1\leq j\leq g, \, 1\leq k\leq s},
\end{equation}
with $\simeq x_0$ denoting homotopy equivalence to a loop that is contractible to the point $x_0$.

The uniformization theorem states that every simply connected Riemann surface is conformally equivalent\footnote{Two Riemann surfaces
	$\Sigma_{g,s}$ and $\Sigma'_{g',s'}$, are conformally equivalent if there exists a bi-holomorphic map
	$
	f:\Sigma_{g,s}\to \Sigma'_{g',s'}$,
	that locally preserves angles.} to one of three Riemann surfaces: the upper half plane $\mathbb H$, the complex plane $\mathbb C$, or the Riemann sphere $\mathbb P^1$. Hyperbolic Riemann surfaces are those whose universal cover is conformally equivalent to $\mathbb H$. As a consequence, a hyperbolic Riemann surface is conformally equivalent to
the quotient of the free and holomorphic action of a discrete group $\Delta_{g,s}\subset\mathbb P SL(2,\mathbb R)$ isomorphic to  $\pi_1(\Sigma_{g,s},x_0)$ on $\mathbb H$:
\begin{equation}\label{eq:unif}
	\Sigma_{g,s}\sim \mathbb H\slash{\Delta_{g,s}}.
\end{equation}

Thanks to the conformal equivalence, we can pull back the metric of $\mathbb H$ to $\Sigma_{g,s}$, therefore a hyperbolic Riemann surface admits a (unique!) Riemannian metric of constant negative curvature $-1$. Note that due to this, the boundaries of the Riemann surfaces will be infinitely distant from it, because the absolute $\partial\mathbb H$ is infinitely distant from $\mathbb H$. 

In conclusion, we may think of the Teichm\"uller space as the set of all possible choices for the group ${\Delta_{g,s}}$ that give rise to the same topological surface. Further reading on the Teichm\"uller space as the space of metrics of constant curvature modulo diffeomorphisms isotopic to the identity can be found in chapter V, section 4 of \cite{Nash}.

Let us illustrate these ideas in the example of a Riemann sphere with three boundaries $S_{0,3}$. We will need to use a few facts about the upper half plane $\mathbb H$, its metric and the action of $\mathbb PSL_2(\mathbb R)$ on it. We have collected these useful facts in Appendix \ref{app:hyp-geo}. We will denote by  $\gamma$ a generic matrix in $SL_2(\mathbb R)$ and by $\gamma(z)$ the corresponding fractional linear transformation:
$$
\gamma=\left(\begin{array}{cc}
	a & b \\
	c& d
\end{array}\right),\quad \gamma(z):=\frac{a z+b}{c z+ d}, \quad a d - b c=1.
$$

\begin{example}\label{ex:3holedsphere}
	Consider the following three hyperbolic elements in $\mathbb PSL_2(\mathbb R)$
	$$
	\gamma_1(z):= \frac{z}{3}-2,\quad 
	\gamma_2(z):=\frac{6 z}{1+3 z},\quad
	\gamma_3(z):=-\frac{ z+2}{3 z+4},
	$$
	then $\gamma_1\circ\gamma_2\circ \gamma_3 (z)=z,$ so that the group $\Delta_{0,3}:=\langle\gamma_1,\gamma_2,\gamma_3| \gamma_1\circ\gamma_2\circ \gamma_3 =\operatorname{id}\rangle$ is isomorphic to the fundamental group of $\Sigma_{0,3}$, the Riemann sphere with three open discs removed.
	
	Let us study the action of
	$\Delta_{0,3}$ on $\mathbb H$ in order to understand the quotient $ \mathbb H\slash_{\Delta_{0,3}}$. For each $\gamma_i$, $i=1,2,3$ we denote by $q^{(i)}_1,q^{(i)}_2$ its fixed points:
	$$
	q^{(1)}_1=\infty, \, q^{(1)}_2=-3,\qquad
	q^{(2)}_1=0,\, q^{(2)}_2=\frac{5}{3},\qquad
	q^{(3)}_1=-1,\, q^{(3)}_2=-\frac{2}{3}
	$$
	In \Cref{fig:inv-axis} we show the invariant axes of the three generators of $\Delta_{0,3}$.
	\begin{figure}[!htb]
		\centering
		\includegraphics[width=0.7\textwidth]{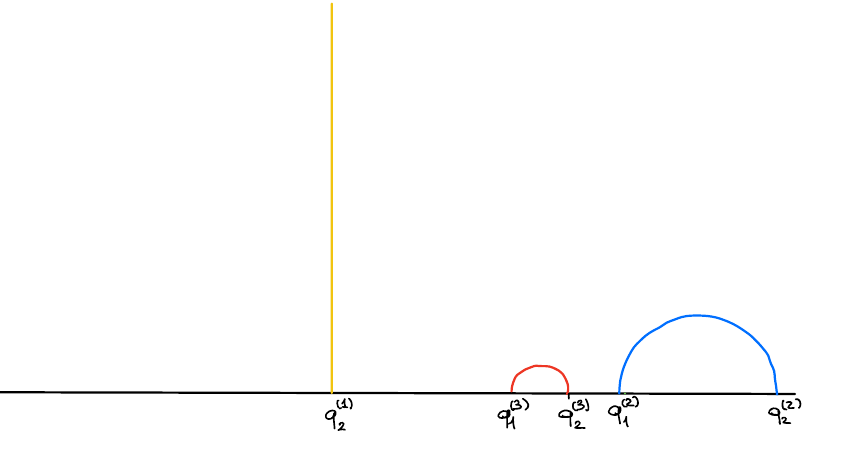}
		\caption{The invariant axis of $\gamma_1$ in yellow, of $\gamma_2$ in blue and of $\gamma_3$ in red.}\label{fig:inv-axis}
	\end{figure}
	
	We now want to draw the fundamental domain of the group $\Delta_{0,3}$. To this aim, we select the unique geodesic $g_{12}$ orthogonal to the $\gamma_1$ and the $\gamma_2$ invariant axes and its image under $\gamma_1$, and the unique geodesic $g_{23}$ orthogonal to the $\gamma_2$ and the $\gamma_3$ invariant axes and its image under $\gamma_3^{-1}$. These geodesics are displayed in \Cref{fig:all-geo}.
	
	\begin{figure}[!htb]
		\centering
		\includegraphics[width=0.8\textwidth]{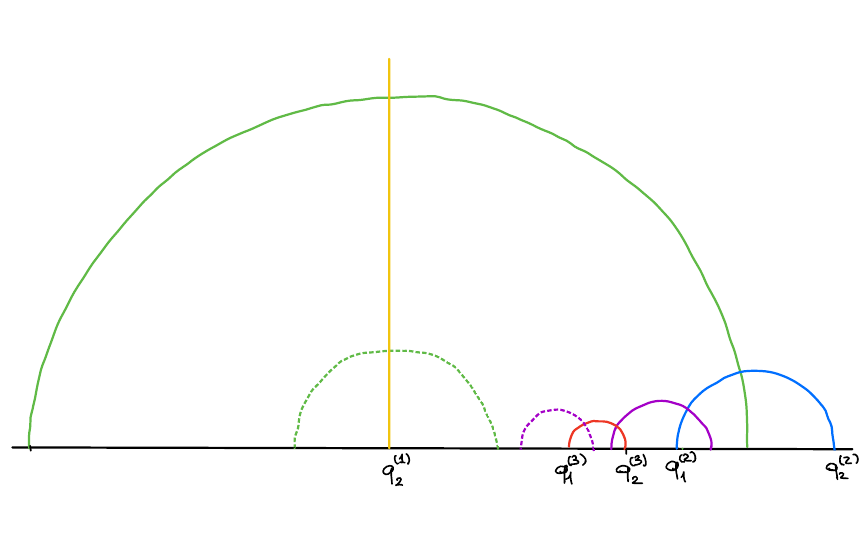}
		\caption{The geodesic $g_{12}$ is drawn in solid green, its image under $\gamma_1$ in dashed green. The geodesic  $g_{23}$ in solid purple, its image under $\gamma_3^{-1}$ in dashed purple.}\label{fig:all-geo}
	\end{figure}
	
	In order to draw the fundamental domain, we only need to consider two of the generators as the third one is determined by the relation $\gamma_1\circ\gamma_2\circ \gamma_3 =\operatorname{id}$. For example, we choose $\gamma_1,\gamma_3$ and set $\gamma_2= \gamma_1^{-1}\circ\gamma_3^{-1}$. Observe that the fundamental domain of $\gamma_1$ is contained in the strip between the two green geodesics in \Cref{fig:all-geo}, while the one of $\gamma_3$ is the whole portion of $\mathbb H$ outside the two purple geodesics (see Examples \ref{ex:gamma1} and \ref{ex:gamma2} in the Appendix). Intersecting these two fundamental domains we obtain the fundamental domain of $\Delta_{0,3}$, which is displayed in \Cref{fig:fun-dom.03}.
	\begin{figure}[!htb]
		\centering
		\includegraphics[width=0.8\textwidth]{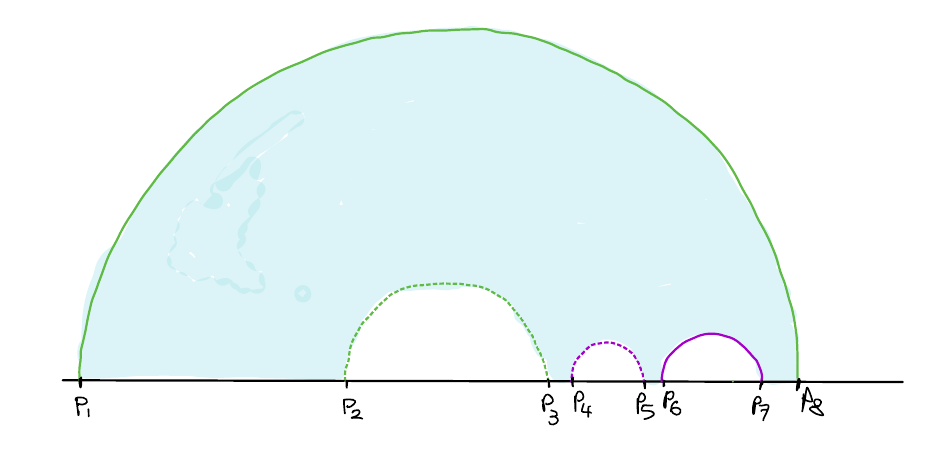}
		\caption{The fundamental domain  of $\Delta_{0,3}$ is the region filled in light blue.}\label{fig:fun-dom.03}
	\end{figure}
	
	Let us now perform the quotient $ \mathbb H\slash{\Delta_{0,3}}$. Act on the fundamental domain displayed in \Cref{fig:fun-dom.03} by $\gamma_1$. This identifies the two green geodesics producing a double funnel in which one side has two wedges removed, see \Cref{fig:double-funnel}.
	\begin{figure}[!htb]
		\centering
		\includegraphics[width=0.7\textwidth]{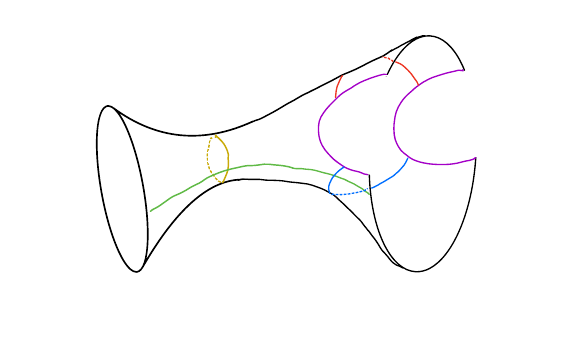}
		\caption{Gluing the fundamental domain along the green geodesic produces a double funnel with two wedges removed (the areas under the two purple geodesics). In this figure we have also drawn the invariant axis of $\gamma_3$ in red, the one of $\gamma_1$ in yellow and the one of $\gamma_2$ in blue.}\label{fig:double-funnel}
	\end{figure}
	
	Act on the double funnel displayed in \Cref{fig:double-funnel} by $\gamma_3$. This identifies the two purple wedges and produces the three-holed sphere in  \Cref{fig:3holed-sphere}.
	\begin{figure}[!htb]
		\centering
		\includegraphics[width=0.6\textwidth]{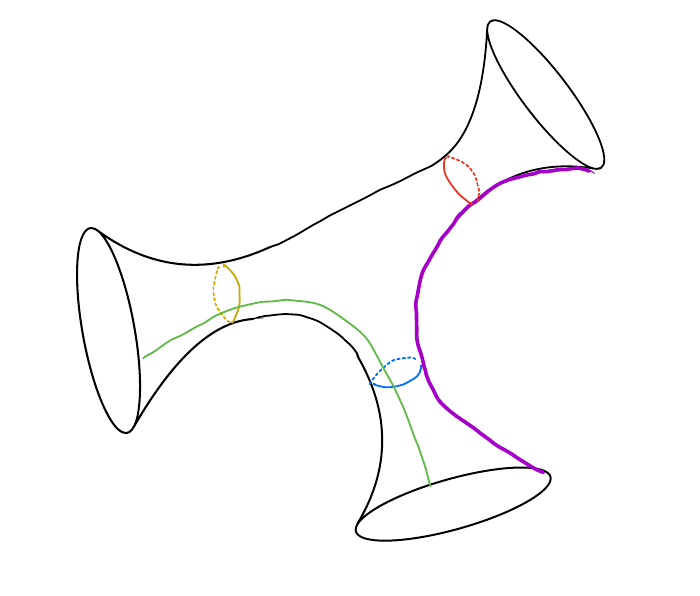}
		\caption{The three holed sphere as quotient of $\mathbb H$ by $\Delta_{0,3}$. The invariant axis of each generator of  $\Delta_{0,3}$ gives a closed geodesic which is in the same conjugacy class as the corresponding hole. The green and purple geodesics correspond to the identified boundaries under the quotient.}\label{fig:3holed-sphere}
	\end{figure}
\end{example}

The above example highlights a few facts that are valid in general. The first fact is that in order to obtain the Riemann sphere with three funnels we considered a group generated by two \textit{hyperbolic elements.}\/ This is because we want the group $\Delta_{0,3}$ to act freely. In order 
for the action of a discretely generated subgroup of $\mathbb PSL_2(\mathbb R)$ to be free, all elements in the group must be either parabolic or hyperbolic (elliptic elements have fixed points in the interior of $\mathbb H$). In general the following result holds:

\begin{lemma}
	Given a Riemann surface of genus $g$ with $s$ boundaries  $\Sigma_{g,s}$, any group $\Delta_{g,s}\subset\mathbb PSL_2(\mathbb R)$ that is isomorphic to $\operatorname{Deck}(\Sigma_{g,s})$ must be generated by $2 g+s$ hyperbolic or parabolic elements.
\end{lemma}

Another interesting fact that emerged from the study of Example \ref{ex:3holedsphere} is that each hyperbolic element $\gamma_i$ fixes a unique geodesic, its invariant axis. The portion of the invariant axis contained in the fundamental domain is called \textbf{bottleneck geodesic}\/ and has finite length $l_{\gamma_i}$. This length characterizes the conjugacy class of $\gamma_i$ as proven in the following:

\begin{lemma}\label{lm:conj-cl}
	The conjugacy classes of hyperbolic elements $\gamma\in\mathbb PSL(2,\mathbb R)$ are in one to one correspondence with 
	closed geodesics of finite length. In particular
	\begin{equation}\label{eq:length-tr}
		e^{\frac{l_\gamma}{2}}+e^{-\frac{l_\gamma}{2}}=\mathrm{Tr}(\gamma).
	\end{equation}
	Moreover, if $x_1,x_2$ denote the fixed points of $\gamma$, then
	\begin{equation}\label{eq:length-cr}
		e^{{l_\gamma}}= \operatorname{cr}(\gamma(z),z;x_1,x_2).
	\end{equation}
\end{lemma}

\begin{proof}
	Given a conjugacy class of hyperbolic element $[\gamma]$, pick the diagonal representative with elements $a,\frac{1}{a}$, $a\in\mathbb R$. Then this acts as a dilation $\gamma(z) = a^2 z$ in $\mathbb H$. As seen in Example \ref{ex:gamma2}, the action of this dilation maps a geodesic centered at $0$ with radius $r$ to geodesics centered at $0$ and with radius $ a^2 r$. The only invariant geodesic is then a vertical segment between $i r$ and $i a^2 r$. The length of this geodesic is $l_\gamma=\ln{a^2}$, therefore proving \eqref{eq:length-tr}. 
	
	Vice-versa, given any geodesic segment of finite length, we can bring it to a vertical one by the action of $\mathbb PSL_2(\mathbb R)$. Then we choose a dilation that keeps the vertical geodesic containing the segment fixed and that maps one extremum to the other. This dilation is unique up to inversion.
	
	The proof of \eqref{eq:length-cr} follows from \eqref{eq:conf-rel1} in the Appendix.
\end{proof}

In Example \ref{ex:3holedsphere}, we saw that for each infinitely far away hole there is a unique closed geodesic in the same conjugacy class as the given hole. 
This is also a general fact which we don't prove:

\begin{lemma}\label{lm:conj-cl-pi1}
	The conjugacy classes in $\pi_1(\Sigma_{g,s})$ are in one to one correspondence with 
	closed geodesics of finite length.
\end{lemma}

Thanks to Lemma \ref{lm:conj-cl-pi1}, we can associate a unique 
bottleneck geodesic to each boundary component  in $\Sigma_{g,s}$, and thanks to  Lemma \ref{lm:conj-cl}, a unique conjugacy class of an element in $\Delta_{g,s}$.  A boundary component is called a \textit{hole} if its  corresponding element in $\Delta_{g,s}$ is hyperbolic and a \textit{puncture} if it is parabolic. In the latter case the length $l_\gamma$ defined in \eqref{eq:length-tr} is $0$.
\section{Teichm\"uller space $\mathcal T(\Sigma_{g,s})$}
\label{se:teich}

\subsection{Ideal triangulations}\label{suse:idtr}

The hyperbolic area of a Riemann surface $\Sigma_{g,s}$ with boundary is  infinite. This corresponds to the fact that when the Fuchsian group $\Delta_{g,s}$ contains hyperbolic elements, it has a fundamental domain that contains segments on the absolute (see Examples \ref{ex:gamma1} and \ref{ex:gamma2}). We would instead like to have domains in $\mathbb H$ that can be triangulated by a finite number of ideal triangles (and therefore have finite hyperbolic area as explained in the Appendix). To achieve this, we consider the finite part $\Sigma_{g,s}^f$ of the Riemann surface, which is an open set obtained by removing from  $\Sigma_{g,s}$ all infinite funnels at the bottleneck geodesics. 
Correspondingly, in $\mathbb H$, we remove from the fundamental domain all parts between the absolute and the images of the bottleneck geodesics. Note that by doing so, we will obtain a region $R$ in $\mathbb H$ that still satisfies one of the properties of a fundamental domain, namely $\gamma(R)\cap R = \emptyset\, \forall \gamma\in G\setminus\{\operatorname{id}\}$, but no longer 
covers all of $\mathbb H$, namely it is no longer true that $\forall z\in\mathbb H, \, \exists \gamma\in G$ such that $\gamma(z)\in R$.
In this section we are going to see that despite this, the region $R$ still contains valuable information that we will use to give a combinatorial description of the Teichm\"uller space.

\begin{example}\label{ex:id-tr}
	We now chop off all infinite funnels from the three holed sphere of Example \ref{ex:3holedsphere} to obtain  a pair of pants which has geodesic boundaries given by the three bottleneck geodesics (see left hand side of \Cref{fig:pair-of-pants}). In $\mathbb H$ (see right hand side of \Cref{fig:pair-of-pants}), the portion of the fundamental domain bounded by the bottleneck geodesics is an octagon $O$ of finite area $2\pi$ (see \eqref{eq:GB}).
	
	\begin{figure}[!htb]
		\centering
		\includegraphics[width=0.9\textwidth]{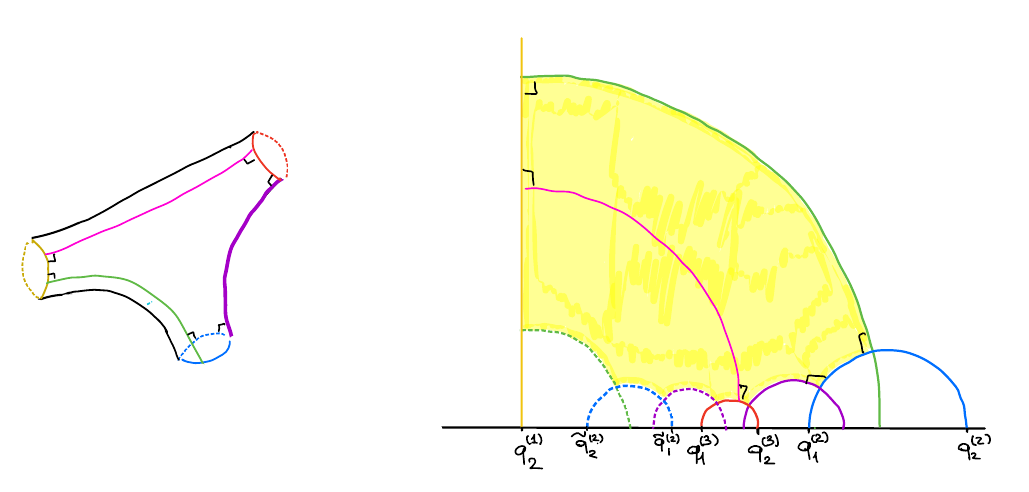}
		\caption{On the left we display the pair of pants cut into two hexagons (on the front, the other on the back). On the right, the corresponding octagon in $\mathbb H$ highlighted in yellow. The dashed blue geodesic is the image of the solid blue one under $\gamma_3^{-1}$. This octagon is also split in two  congruent hexagons.}\label{fig:pair-of-pants}
	\end{figure}
	
	By the Riemann mapping theorem, there exists a conformal transformation mapping the interior of this octagon into the interior of an ideal quadrangle in $\mathbb H$.
	The boundary angles are not preserved because conformality only applies on the interior of the domain. In principle, this conformal transformation could be deduced from Schwarz–Christoffel formula, however the explicit computation can only be carried out in a handful of cases, for example, in the case of a quadrangle, it already involves elliptic integrals, so we won't even attempt to do it in the case of an octagon. 
	
	Let us instead describe this ideal quadrangle in terms of geodesics both on the pair of pants and on $\mathbb H$. On the pair of pants, consider three non-self intersecting geodesics asymptotically winding between two holes (see the black, light green and light blue geodesics on the left hand side of \Cref{fig:ideal}). On $\mathbb H$, two of these geodesics define an ideal quadrangle and the third subdivides this quadrangle into two ideal triangles (see the right hand side of \Cref{fig:ideal}).
	\begin{figure}[!htb]
		\centering
		\includegraphics[width=0.9\textwidth]{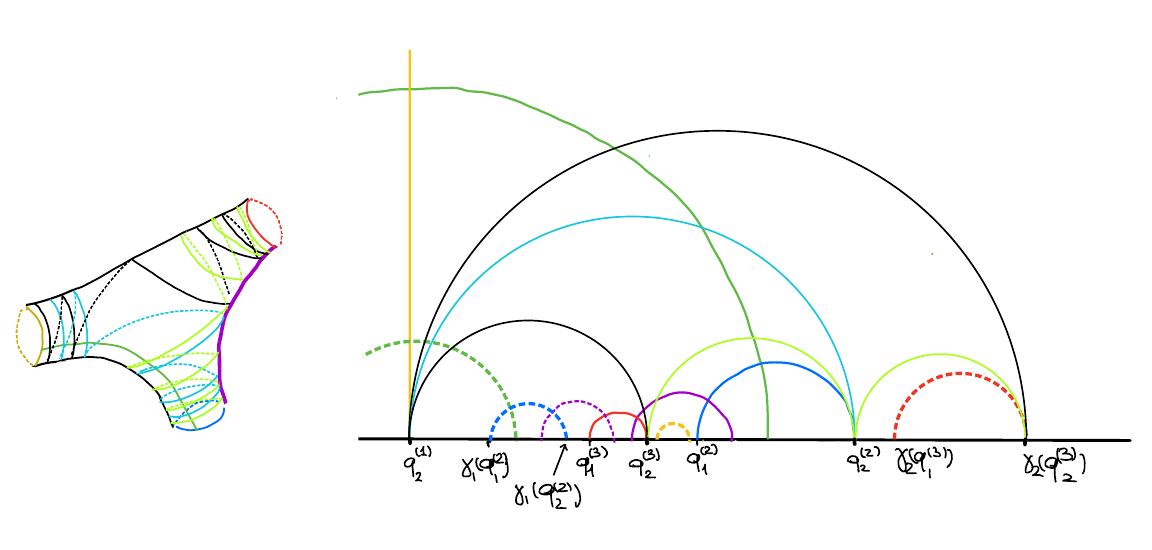}
		\caption{On the left the pair of pants and three infinitely long geodesics triangulating it. On the right their corresponding geodesics in $\mathbb H$.}\label{fig:ideal}
	\end{figure}
	Therefore, we can triangulate the pair of pants by non-self intersecting infinitely winding geodesics.  
\end{example}

In Example \ref{ex:id-tr}, we saw that we can triangulate the pair of pants by non-self intersecting infinitely winding geodesics.  More generally, we give the following:

\begin{definition}
	An ideal triangulation of a Riemann surface $\Sigma_{g,s}$ is a triangulation of its finite part $\Sigma_{g,s}^f$ by non-self-intersecting geodesics which  wind asymptotically around the bottleneck geodesics.
\end{definition}

Any finite portion $\Sigma_{g,s}^f$ of a Riemann surface $\Sigma_{g,s}$ obtained by chopping off the infinite funnels at the bottleneck geodesics admits an ideal triangulation made of non-self intersecting geodesics that wind infinitely towards one or two bottleneck geodesics. We don't prove this fact but it intuitively follows from the fact that we can always subdivide a Riemann surface with boundary into pairs of pants,  see \Cref{fig:Riem-dec}.

\begin{figure}[!htb]
	\centering
	\includegraphics[width=0.4\textwidth]{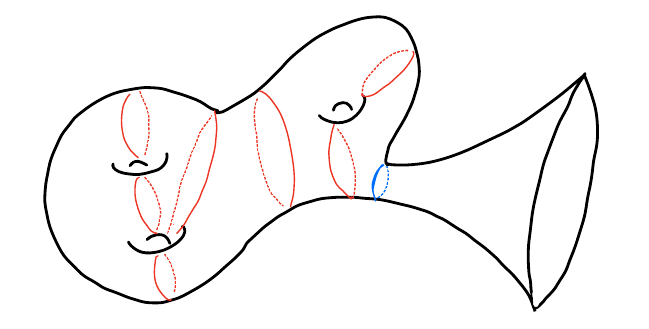}
	\caption{A pair of pants decomposition of $\Sigma_{3,1}$. We have denoted in red the geodesics where we cut the surface and in blue the bottleneck geodesic where we chop off the funnel.}\label{fig:Riem-dec}
\end{figure}   

Therefore,  we can expect to be able to extract the ideal triangulation of $\Sigma_{g,s}^f$ from the triangulations of the pair of pants.

It should be emphasized that such ideal triangulations, when carried through to $\mathbb H$, are not enough to determine $\Delta_{g,s}$ uniquely. 

\begin{example}\label{ex:no-fun-dom}
	Let's try to reconstruct two hyperbolic elements from the ideal quadrangle on the right hand side of \Cref{fig:ideal}. We have re-drawn this quadrangle in 
	\Cref{fig:quadrangle} for convenience.
	
	\begin{figure}[!htb]
		\centering
		\includegraphics[width=0.7\textwidth]{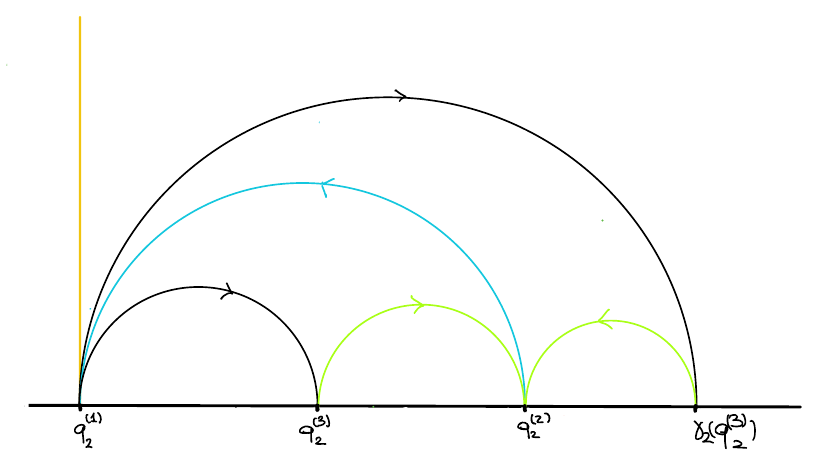}
		\caption{The ideal triangulation of a pair of pants.}\label{fig:quadrangle}
	\end{figure}
	
	We require that one hyperbolic element identifies the black geodesics, and that another hyperbolic element identifies the green ones. However, these data determine $\gamma_1$ and $\gamma_2$ up to two parameters. 
	In order to recover the two hyperbolic elements uniquely, we need to add some information, for example, let us add two more ideal triangles as in \Cref{fig:winding} and require the corresponding further identification of the black geodesics under $\gamma_1$ and green ones under $\gamma_2$. We leave to the reader the exercise of showing that this uniquely reconstructs $\gamma_1,\gamma_2\in\mathbb PSL_2(\mathbb R)$.
	
	\begin{figure}[!htb]
		\centering
		\includegraphics[width=0.7\textwidth]{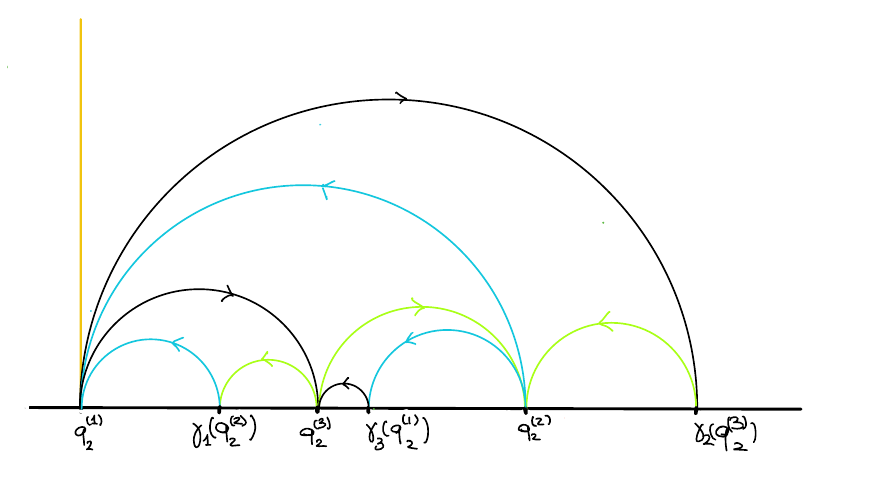}
		\caption{Several triangles identified  by the action of the Fuchsian group on the ideal triangulation of a pair of pants.}\label{fig:winding}
	\end{figure}

	This will come in handy in Section \ref{suse:FGr}.
\end{example}

\subsection{Heuristic description of the Teichm\"uller space}\label{se:eu-T}

We now want to describe the Teichm\"uller space as the space of metrics of constant curvature modulo diffeomorphisms isotopic to the identity. Thanks to the fact that, given any Riemann surface with boundary  $\Sigma_{g,s}$, its finite portion $\Sigma_{g,s}^f$  admits an ideal triangulation that subdivides it into pairs of pants, 
to describe the Teichm\"uller space $\mathcal T(\Sigma_{g,s})$,  we need to first describe $\mathcal T(\Sigma_{0,3})$.

First let us show that by fixing the length of each bottleneck geodesic, we fix the metric on the pair of pants uniquely. To this aim we prove the following simple lemma:

\begin{lemma}\label{lm:hexagon}
	Given a hexagon with right internal angles in $\mathbb H$, the lengths of any three non consecutive boundaries fix uniquely the length of the other three. 
\end{lemma}

\begin{proof} To follow this proof it is best to look at \Cref{fig:hexagon}.
	\begin{figure}[!htb]
		\centering
		\includegraphics[width=0.8\textwidth]{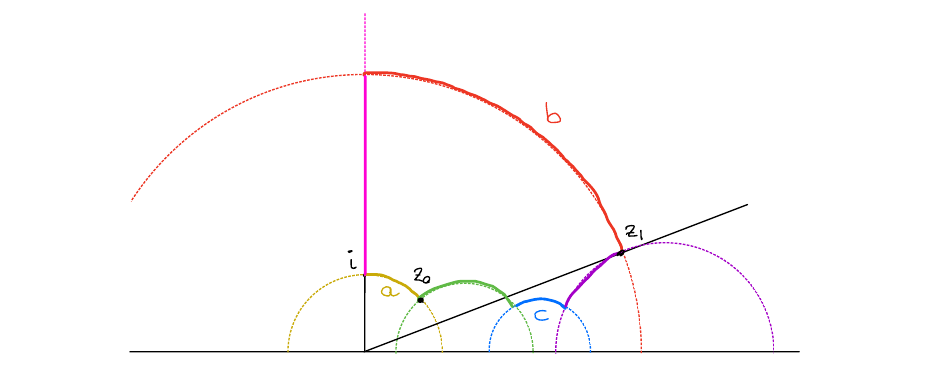}
		\caption{Proof of \Cref{lm:hexagon}.}\label{fig:hexagon}
	\end{figure}   
	
	By the action of $\mathbb PSL(2,\mathbb{R})$, we place the first boundary of unknown length on the imaginary axis, starting at $i$ (dashed magenta). Select the unique geodesic (dashed yellow) orthogonal to the imaginary axis in $i$. Fix  a point $z_0$ on the dashed yellow geodesic in such a way that the segment of extrema $i$ and $z_0$ has known length $a$. Then we take the unique geodesic (dashed green) orthogonal to the yellow geodesic in the point $z_0$. Now take a Euclidean line from the origin (black) such that any point on this line except the origin has given distance $b$ from the imaginary axis.
	For any point $z_1$ on this line, we  consider the unique geodesic orthogonal to the imaginary axis through $z_1$ (in dashed red). This cuts a segment of length $b$ on the red dashed geodesic by construction (the segment is in solid red in the picture). Take the unique geodesic orthogonal to it at $z_1$ (in dashed purple) and the unique geodesic orthogonal to the dashed purple and the dashed green (in dashed blue).
	As we slide the point $z_1$ on the black line, the dashed red, purple and blue geodesics slide. We slide them until the segment on the dashed blue geodesic cut by the intersections with the dashed green and dashed purple has given length $c$. This uniquely fixes the choice of the point on the black line, therefore it fixes the length of the magenta, green and purple segments.
\end{proof}

As a consequence of Lemma \ref{lm:hexagon}, it is clear that the two hexagons in \Cref{fig:pair-of-pants} are congruent (magenta, green and purple boundaries have the same length). Therefore, knowing the length of each boundary geodesic in the pair of pants fixes uniquely each hexagon and therefore the octagon, hence the metric. So a point in $\mathcal T_{0,3}$ is uniquely given by the three lengths of the three bottleneck geodesics.
Therefore
$$
\dim_{\mathbb R}\mathcal T(\Sigma_{0,3})=3.
$$

In general, the finite part of a Riemann surface $\Sigma_{g,s}^f$ is cut into $2 g-2+s$ pairs of pants. When we glue back the Riemann surface from the pairs of pants, at each gluing we will have a freedom of rotating the two geodesics that are glued with respect to each other - we parametrize this freedom by an angle called \textbf{twist}. This proves the following  (by induction):
$$
\dim_{\mathbb R}\mathcal T(\Sigma_{g,s})=6 g-6+ 3s.
$$
Note that some authors fix the lengths of the bottleneck geodesics, namely consider the conjugacy classes of loops corresponding to the holes to be fixed and therefore give the dimension as $6 g-6+ 2s$.

\begin{example}
	In the case of a sphere with $4$ boundaries, we separate it in two pairs of pants (not uniquely!). Then, in order to specify the metric uniquely,  we have to fix the lengths of the four bottleneck geodesics,  the length of the geodesic cutting the sphere with $4$ boundaries into two pairs of pants and the corresponding twist. Namely 
	$$
	\dim_{\mathbb R}\mathcal T(\Sigma_{0,4})=6.
	$$  
\end{example}

\subsection{Ribbon graphs and coordinatization of the Teichm\"uller space}\label{se:ribbon-graphs}

A \textbf{ribbon-graph}, also-called \textbf{fat-graph}, associated to a Riemann
surface  $\Sigma_{g,s}$ of genus $g$ and with $s$ holes is a connected three-valent
graph drawn without self-intersections on $\Sigma_{g,s}$
with a prescribed cyclic ordering
of labeled edges entering each vertex~\cite{Fock1}~\cite{Fock2}. Replacing each vertex by a small disc and each edge
by a thin strip, glued to the discs according to the cyclic ordering, one
obtains a surface with boundary homeomorphic to $\Sigma_{g,s}$, whose $s$
boundary components correspond to the $s$ holes. In other words, the
fat-graph is a spine of $\Sigma_{g,s}$: the surface deformation retracts
onto it.
Such a fat-graph is dual to the ideal triangulation associated to $\Sigma_{g,s}^f$.

\begin{example}
	We can carry out this duality on the ideal triangulation in $\mathbb H$ taking into account the fact that the Fuchsian group identifies edges, see \Cref{fig:fat-graph}. 
	
	\begin{figure}[!htb]
		\centering
		\includegraphics[width=0.8\textwidth]{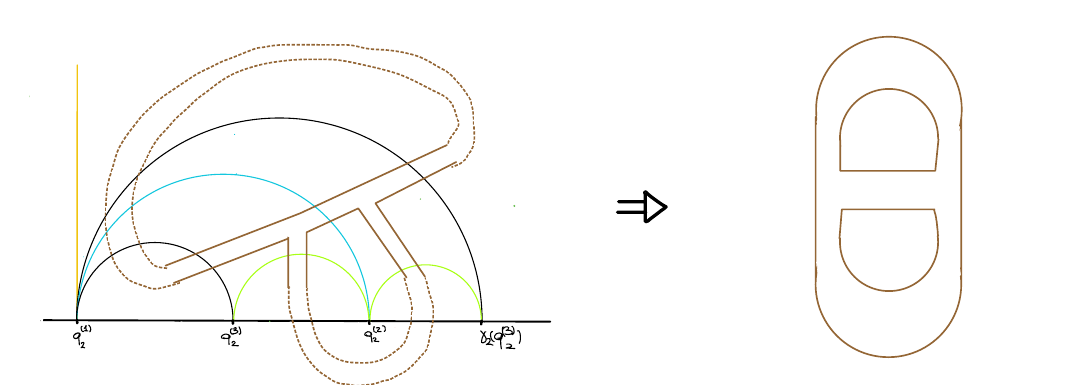}
		\caption{The duality between ideal triangulation (given by the black, blue and green geodesics) and fat-graph in the case of a pair of pants. The dashed part of the fat-graph is due to the pairwise identification of edges.}\label{fig:fat-graph}
	\end{figure}

	This duality allows us to  choose an orientation on the edges of the fat-graph compatible with the choice of orientation of the ideal triangulation of $\Sigma_{g,s}^f$, see \Cref{fig:orientation}. 
	
	\begin{figure}[!htb]
		\centering
		\includegraphics[width=0.7\textwidth]{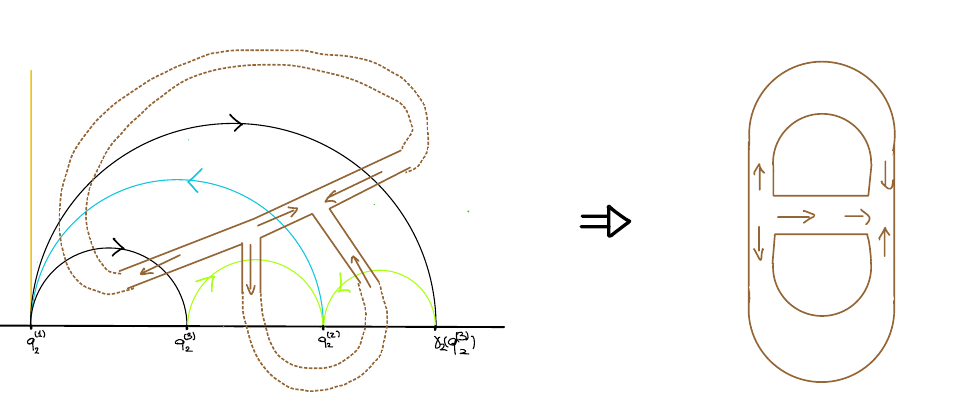}
		\caption{The fat-graph inherits its orientation from the ideal triangulation: each fat-graph edge crosses its triangulation edge from left to right, so every vertex has either all incoming or all outgoing edges.}\label{fig:orientation}
	\end{figure}   
\end{example}

Given a Riemann surface $\Sigma_{g,s}$ of genus $g$ and $s$ holes, its finite part $\Sigma_{g,s}^f$ obtained by chopping off its infinite funnels at the bottleneck geodesics has hyperbolic area $2\pi(2 g-2 +  s)$. This means that we can triangulate $\Sigma_{g,s}^f$ with $4 g-4 + 2 s$ ideal triangles. Taking the dual of the triangulation we obtain a ribbon graph with $4 g-4 + 2 s$ vertices and $s$ faces. Because the Euler characteristic of $\Sigma_{g,0}$ is $2-2g$, we obtain that the number of edges in the ribbon graph is $6g-6+3 s$.

The advantage of considering fat-graphs is that conjugacy classes of elements of $\pi_1(\Sigma_{g,s})$ correspond to loops in the fat-graph.

The geodesic length functions, which are traces of hyperbolic elements in the Fuchsian group $\Delta_{g,s}$ such that 
$$
\Sigma_{g,s}\sim\mathbb H\slash \Delta_{g,s}
$$
are obtained by decomposing each hyperbolic matrix $\gamma\in \Delta_{g,s}$ into a
product of the so--called \textbf{right, left and edge matrices:}
\begin{equation}
	R:=\left(\begin{array}{cc}1&1\\-1&0\\
	\end{array}\right), \qquad
	L:=\left(\begin{array}{cc}0&1\\-1&-1\\
	\end{array}\right),\qquad
	X({s_i}):=\left(\begin{array}{cc}0&-\exp\left({\frac{s_i}{2}}\right)\\
		\exp\left(-{\frac{s_i}{2}}\right)&0\end{array}\right),
	\label{eq:generators}
\end{equation}
where $s_i$ is a coordinate associated to the $i$-th edge in the fat graph.
Given a loop in the fat-graph, namely a loop in the finite part of the Riemann surface $\Sigma_{g,s}^f$, the corresponding matrix in $\Delta_{g,s}$ is obtained by selecting a starting point at the end of an edge, and then writing from right to left all moves we need to do to complete the loop. 
A coordinate $s_i$ is assigned to each of the 
$6g-6+3 s$ edges. The edge coordinates $s_1,\dots, s_{6g-6+3 s}$,  which coordinatize $\mathcal T_{g,s}$, are called \textbf{Thurston shear coordinates}. We omit the proof here; see Section 2 of \cite{Ch-Pen} for a general explanation of these coordinates. 

\begin{example}
	In the case of a pair of pants, the fat-graph is displayed on the left of \Cref{fig:loops-fat}.
	\begin{figure}[!h]
		\centering
		\includegraphics[width=0.6\textwidth]{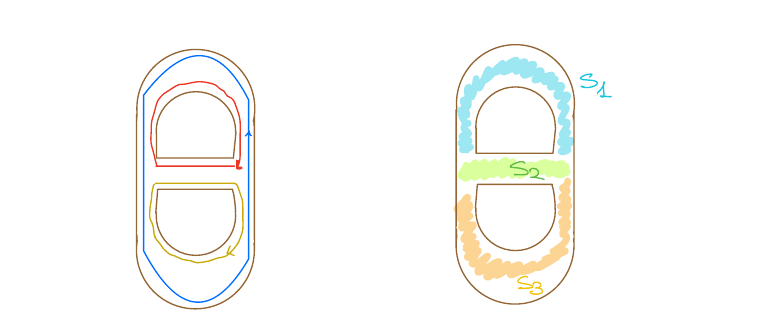}
		\caption{On the left, the three bottleneck geodesics drawn in the fat-graph, on the right the three edges of the fatgraph and their associated coordinates $s_1,s_2,s_3$.}\label{fig:loops-fat}
	\end{figure}   
	
	We label by $s_1,s_2,s_3$ the edges of the fat-graph as in the right-hand-side of \Cref{fig:loops-fat}. Then, we obtain the matrices by choosing a common starting point, for example at the end of edge $s_3$. Then $\gamma_1$ corresponds to going along $s_3$ in the opposite direction, turning right, going along $s_2$ and turning right. Similarly, for $\gamma_2$ and $\gamma_3$. Therefore, starting just under the right vertex in \Cref{fig:loops-fat}, we get
	\begin{align}\label{eq:fuchsian3}
		\gamma_1=R X(s_2) R X(s_3)^{-1}=-R X(s_2) R X(s_3),\nonumber\\
		\gamma_2=-X(s_3)RX(s_1)R,\\
		\gamma_3=L X(s_1)R X(s_2) L.\nonumber
	\end{align}
	Note that we have put a minus sign on the right hand side of $\gamma_1$ and $\gamma_2$ to keep in mind the orientation of the loop w.r.t. the fat-graph. However, because we actually work in $\mathbb PSL_2(\mathbb R)$ this overall sign is irrelevant.
	
	Explicitly, formulae \eqref{eq:fuchsian3} give:
	$$
	\gamma_1=\left(\begin{array}{cc}
		- e^{-\frac{s_2+s_3}{2}}   &  e^{\frac{s_3}{2}} \left(e^{-\frac{s_2}{2}}+e^{\frac{s_2}{2}}\right)\\
		0 & - e^{\frac{s_2+s_3}{2}} 
	\end{array}
	\right),\quad
	\gamma_2=\left(\begin{array}{cc}
		- e^{\frac{s_1+s_3}{2}}   &  0\\
		-e^{-\frac{s_3}{2}} \left(e^{-\frac{s_1}{2}}+e^{\frac{s_1}{2}}\right) & - e^{-\frac{s_1+s_3}{2}} 
	\end{array}
	\right),
	$$
	while $\gamma_3=\gamma_2^{-1}\gamma_1^{-1}$.
	The corresponding fractional linear transformations are:
	\begin{align*}
		\gamma_1(z)= e^{-s_2-s_3}  z-1- e^{s_2},\quad
		\gamma_2(z)=\frac{e^{s_1+s_3} z}{e^{s_1} z+z+1  },\\
		\gamma_3(z)=-\frac{1+e^{s_2}(1+z)}{1+e^{s_2}(1+z)+e^{s_1+s_2}(1+z)}.\qquad\qquad\qquad
	\end{align*}
	For $s_1=\log(2)$, $s_2=0$, $s_3=\log(3)$, we obtain the transformations of Example \ref{ex:3holedsphere}.
	The fixed points of these elements are
	\begin{equation}\label{eq:fixpts}
    \begin{split}
		q_1^{(1)}= \infty,\quad
		q_2^{(1)}=-\frac{e^{s_3}(1+e^{s_2})}{e^{s_2+s_3}-1},\\
		q_1^{(2)}=0 ,\quad 
		q_2^{(2)}= \frac{e^{s_1+s_3}-1}{1+e^{s_1}},\\
        q_1^{(3)}=-1 ,\quad 
		q_2^{(3)}=-\frac{1+e^{-s_2}}{1+e^{s_1}}, 
         \end{split}
	\end{equation}
	which, for $s_1,s_2,s_3\in\mathbb R_+$, are qualitatively placed in the same way as in \Cref{fig:all-geo}. Therefore the description given in Examples \ref{ex:3holedsphere} and \ref{ex:id-tr} remains valid also in this case. 
	In particular, the geodesic lengths of the three bottleneck curves are
	$$
	l_{\gamma_1}= s_2+s_3,\quad
	l_{\gamma_2}= s_1+s_3,\quad
	l_{\gamma_3}= {s_1+s_2}.
	$$
\end{example}

\subsection{Summary of coordinatization for the Teichm\"uller space}\label{suse:procesi}

In subsection \ref{se:eu-T}, we gave a heuristic argument to show that the dimension of the Teichm\"uller space
is $6 g-6+ 3 s$. In subsection \ref{se:ribbon-graphs}, we explained how to associate a fat-graph to the finite part of $\Sigma_{g,s}$ and showed that this fat-graph has 
$6g-6+3 s$ edges, each equipped with a shear coordinate
$s_i$. We also gave an explanation of how such coordinates are enough to uniquely determine a metric of constant negative curvature on the finite part of the Riemann surface,
so that  $\mathcal T(\Sigma_{g,s})$ is completely coordinatized in terms of the shear coordinates $s_1,\dots,s_{6g-6+3 s}$. 


In this subsection, we explain how these coordinates are also natural if we instead think about the  definition of
Teichm\"uller space as the quotient of the representation space of the fundamental group:
$$
\mathcal T(\Sigma_{g,s})=\opn{Hom}'(\pi_1(\Sigma_{g,s}),\mathbb PSL_2(\mathbb R))/\mathbb PSL_2(\mathbb R).
$$
What does it mean to take a coordinate, or more generally, a function on this space? It means that forevery generating loop in $\pi_1(\Sigma_{g,s})$, we associate a fractional linear transformation in $\mathbb PSL_2(\mathbb R)$, or in other words a matrix in $SL_2(\mathbb R)$ defined up to global sign,
and then look for functions of these matrices that are invariant under conjugation. Procesi proved that it is enough to understand trace functions in order to understand the whole ring of invariant functions, and that this ring is generated by finitely many trace functions. 

This is in line with Lemma \ref{lm:conj-cl}: conjugacy classes of loops are in $1:1$ correspondence with geodesics of finite length and the trace of the matrix associated to such a conjugacy class by the choice of a homomorphism in the representation space is related to the length of this geodesic by \eqref{eq:length-tr}.

Of course, we need enough traces to generate the entire coordinate ring. Looking at formula \eqref{eq:fund-g}, we see that the fundamental group of $\Sigma_{g,s}$ is generated by $2 g+s-1$ equivalence classes. Taking the traces of the matrices associated to these equivalence classes would give us only $2 g+s-1$, while we need $6g-6+3 s$ traces to generate the coordinate ring. Therefore, we don't only take traces of the single matrices, but also of their products. In these lecture notes we will call these traces Fricke-Vogt coordinates.

This is explained in the next example.

\begin{example}\label{ex:PVI}
	We start by describing the coordinate ring of 
	$$\opn{Hom}'(\pi_1(\Sigma_{0,4}), SL_2(\mathbb K))/SL_2(\mathbb K),
	$$
	where $\mathbb K$ is either $\mathbb R$ or $\mathbb C$. Because the fundamental group of $\Sigma_{0,4}$ is generated by $4$ loops,
	this is equivalent to describing the 
	coordinate ring of the following space
	$$
	\{(M_1,M_2,M_3,M_4)| M_i\in SL_2(\mathbb K), M_1 M_2 M_3 M_4=\mathbb I\}/SL_2(\mathbb K).
	$$
	This is a $6$ dimensional space, because $3$ matrices in $SL_2(\mathbb K)$ depend on $9$ entries, and quotienting by a three dimensional group reduces the dimension to $6$. Consider the following traces, called \textbf{Fricke-Vogt coordinates}:
	$$ 
	x_1=\opn{Tr}(M_2 M_3),\, x_2=\opn{Tr}(M_1 M_3), \, x_3=\opn{Tr}(M_1 M_2),\, G_i=\opn{Tr}(M_i),\,\text{for}\, i=1,\dots,4,
	$$
	then,
	by iterating the so-called \textbf{skein relation:}
	\begin{equation}
		\label{eq:skein}
		\forall A,B\in SL_2(\mathbb K),\quad 
		\opn{Tr}(A B) +   \opn{Tr}(A B^{-1})=    \opn{Tr}(A)    \opn{Tr}(B),
	\end{equation}
	one can prove that 
	$$
	2=\opn{Tr}(M_1 M_2 M_3 M_4)
	$$
	is equivalent to the following relation due to Fricke (a nice proof can be found in \cite{Magnus})
	\begin{equation}
		\label{eq:Fricke}
		\begin{split}
			&x_1 x_2 x_3 + x_1^2+x_2^2+x_3^2 -(G_4 G_1+G_2 G_3) x_1 -(G_4 G_2+G_1 G_3) x_2\\
			& -(G_4 G_3+G_2 G_1) x_3 +G_1^2 +G_2^2+G_3^2 +G_4^2 + G_1 G_2 G_3 G_4=4.
		\end{split}
	\end{equation}
	Relation \eqref{eq:Fricke} defines a $4$-parameter pencil of affine cubic surfaces  in $\mathbb R^3$ which is isomorphic to 
	$\opn{Hom}'(\pi_1(\Sigma_{0,4}), SL_2(\mathbb K))/ SL_2(\mathbb K)$. Note that for $\mathbb K=\mathbb R$, $G_i, x_j\in \mathbb R$, while 
	for $\mathbb K=\mathbb C$, $G_i, x_j\in \mathbb C$. Note that in the case of the Teichm\"uller space, we have to consider the group $\mathbb PSL_2(\mathbb R)=SL_2(\mathbb R)/\langle \pm\mathbb I\rangle$. In this case, we can still use 
	\eqref{eq:Fricke} as long as we make a choice for the signs once and for all. It turns out that the best choice is to take all traces in $\mathbb R_-$.
	
	Let us now describe $\Sigma_{0,4}$ by 
 the fat-graph in \Cref{fig:4-holed-sphere}.   
	\begin{figure}[!htb]
		\centering
		\includegraphics[width=0.6\textwidth]{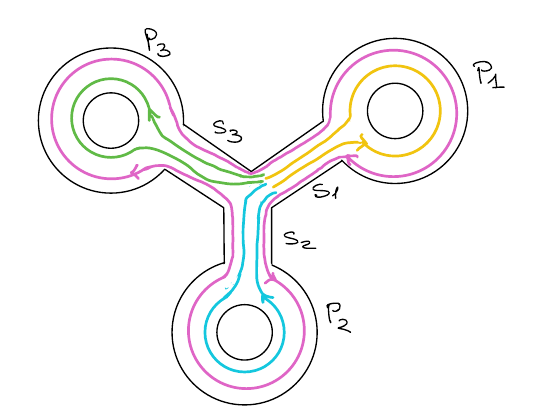}
		\caption{Fat graph of $\Sigma_{0,4}$. The   associated shear coordinates are labeled by $s_1,s_2,s_3$ and $p_1,p_2,p_3$. The fundamental group is generated by the $4$ loops
			$\gamma_1$ in yellow, $\gamma_2$ in cyan, $\gamma_3$ in green and $\gamma_4$ in pink. Note that $\gamma_1\gamma_2\gamma_3\gamma_4=1$.}\label{fig:4-holed-sphere}
	\end{figure}  
	
	We denote by $s_1,s_2,s_3,p_1,p_2,p_3$ the shear coordinates associated to the edges of the fat-graph, where the notation distinguishes between edges that go between different vertices and edges that go back to the same vertex. By selecting the origin of all loops to be at the start of the edge labeled by $s_1$, we obtain the matrices corresponding to the loops as 
	\begin{equation}
	    \begin{split}
	\gamma_1=X(s_1)RX(p_1)RX(s_1),\quad 
	\gamma_2=-R X(s_2)RX(p_2)RX(s_2)L,\\
	\gamma_3=-L X(s_3)RX(p_3)RX(s_3)R,\quad \gamma_4= (\gamma_1\gamma_2\gamma_3)^{-1}.
\end{split}
	\end{equation}
	Let us now calculate $x_i$ and $G_j$ for $M_i=\gamma_i$. We get
	\begin{equation}\label{eq:x-s}
		\begin{split}
			x_1= - e^{\hat s_2+\hat s_3}-  e^{-\hat s_2+\hat s_3} - e^{-\hat s_2-\hat s_3}-\left(e^{\frac{p_3}{2}}+e^{-\frac{p_3}{2}}\right)e^{-\hat s_2}-\left(e^{\frac{p_2}{2}}+e^{-\frac{p_2}{2}}\right)e^{\hat s_3}     ,\\
			x_2= - e^{\hat s_3+\hat s_1}-  e^{-\hat s_3+\hat s_1} - e^{-\hat s_3-\hat s_1}-\left(e^{\frac{p_1}{2}}+e^{-\frac{p_1}{2}}\right)e^{-\hat s_3}-\left(e^{\frac{p_3}{2}}+e^{-\frac{p_3}{2}}\right)e^{\hat s_1}     ,\\
			x_3= - e^{\hat s_1+\hat s_2}-  e^{-\hat s_1+\hat s_2} - e^{-\hat s_1-\hat s_2}-\left(e^{\frac{p_2}{2}}+e^{-\frac{p_2}{2}}\right)e^{-\hat s_1}-\left(e^{\frac{p_1}{2}}+e^{-\frac{p_1}{2}}\right)e^{\hat s_2}     ,\\
			G_i= -e^{\frac{p_i}{2}}-e^{-\frac{p_i}{2}},\quad\text{for }\, i=1,2,3,\,\text{ and }\,
			G_4= -e^{\hat s_1+\hat s_2 + \hat s_3}-e^{-\hat s_1-\hat s_2 - \hat s_3},\\
		\end{split}
	\end{equation}
	where for convenience we have put $\hat s_i = s_i+\frac{p_i}{2}$.
	
	Therefore, we see that for $\mathbb K=\mathbb R$, the affine cubic surface \eqref{eq:Fricke} is coordinatized by the shear coordinates. Actually, because all terms on the r.h.s. in \eqref{eq:x-s} are real analytic,
	if we allow the shear coordinates to be any complex numbers, we obtain a coordinatization of the affine cubic surface \eqref{eq:Fricke} also in the case $\mathbb K=\mathbb C$. 
	
	Vice-versa, we can interpret \eqref{eq:Fricke} as an equation for $G_4$ with two solutions and then we can show that the Jacobian of $x_1x_2,x_3,G_1,G_2,G_3$ as functions of $\hat s_1, \hat s_2, \hat s_3,p_1,p_2,p_3$ is non zero, therefore we can express the shear coordinates in terms of the Fricke-Vogt coordinates. 
\end{example}

\section{Confluence of holes}\label{se:confl}

In this section, we look at the \textbf{chewing–gum moves} introduced in \cite{ChM}, namely  processes where two boundary components in $\Sigma_{g,s}$ collide. 
As the two boundary components approach, the strip between them, called \textbf{chewing-gum strip}, becomes thinner and thinner and, at the same time, longer and longer (in order to preserve the hyperbolic area).
Upon taking the limit of the length of the chewing-gum strip to
infinity, its width becomes $0$, namely it breaks into two \textbf{bordered cusps} in the resulting surface\footnote{There are also degenerate chewing gum  moves that produce disconnected surfaces, but we won't consider these in these lecture notes.}.
Closed geodesics that were passing along the chewing-gum strip become arcs, namely infinitely
long geodesics that start and terminate at bordered cusps. To make sense of a metric on the resulting surface, each bordered cusp is equipped with a horocycle and the length of each infinite arc is only measured along the portion between the horocycles. The resulting decorated surface is called a \textbf{candle cake}. An idea of this process is depicted in \Cref{fig:chewing}.

\begin{figure}[!htb]
	\centering
	\includegraphics[width=0.8\textwidth]{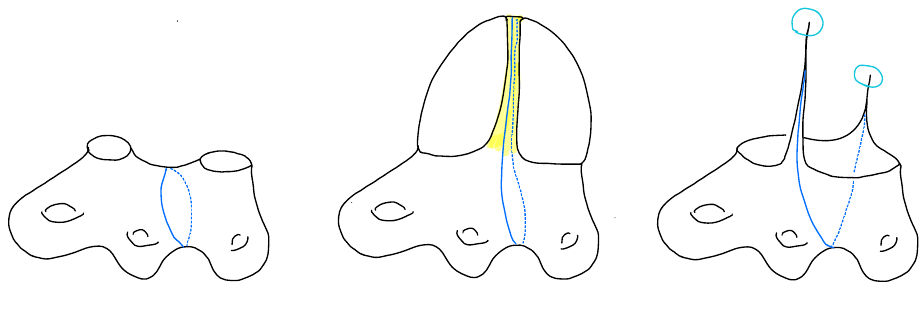}
	\caption{On the left we have the initial surface and a closed blue geodesic. On the center, the boundaries start colliding, so that the strip between them (highlighted in yellow) becomes thinner and longer. The blue geodesic also becomes longer. On the right, at the end of the collision,  the two holes  have merged and created two cusps. The blue geodesic has become an infinite arc.}\label{fig:chewing}
\end{figure}   

We shall denote by $\Sigma_{g,s,m}$ surfaces of genus $g$, $s$ boundaries and $m$ bordered cusps. We call these \textbf{surfaces with marked boundaries}.

To understand what happens to the fat-graph under the collision process we go back to the example of the pair of pants.

\begin{example}\label{ex:s3toinf}
	In subsection \ref{se:ribbon-graphs}, we saw that the Teichm\"uller space of the pair of pants can be described by the fat-graph in \Cref{fig:loops-fat} with shear coordinates $s_1,s_2,s_3$.  To merge two holes, for example the one with the yellow boundary and the one with the blue boundary, we need the bottom part of the fat-graph to become infinitely thin and long as explained above. In other words, the edge labeled by $s_3$ contains our chewing-gum strip. Therefore, in order to perform the chewing-gum move, we break up the edge labeled by $s_3$ as shown in \Cref{fig:chewing-graph}.
	
	\begin{figure}[!htb]
		\centering
		\includegraphics[width=0.6\textwidth]{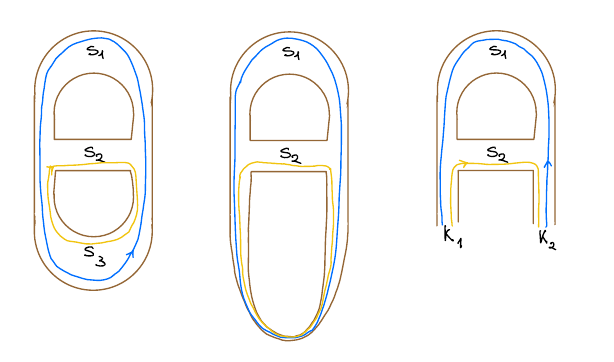}
		\caption{On the left we have the initial fat-graph and two closed geodesics in yellow and blue. On the center, the edge labeled by $s_3$ becomes thinner and longer. The blue geodesic also becomes longer. On the right, at the end of the collision,  the final fat-graph has two bordered cusps. The blue geodesic has become an infinite arc.}\label{fig:chewing-graph}
	\end{figure}   
\end{example}

As illustrated in Example \ref{ex:s3toinf},  after the collision, the ribbon graph breaks, and is replaced by a \textbf{cusped fat graph}.

\begin{definition}\label{def-graph-cusp}
	A connected graph ${\mathcal G}_{g,s,m}$ with a prescribed cyclic ordering of edges
	entering each vertex is called \textit{cusped fat graph} for the surface $\Sigma_{g,s,m}$, if 
	\begin{itemize}
		\item[(a) ]\, it can be embedded  without self-intersections in $\Sigma_{g,s,m}$;
		\item[(b) ]\, all vertices of ${\mathcal G}_{g,s,m}$ are three-valent except exactly $m$
		one-valent vertices (endpoints of the open edges), which are placed at the corresponding
		bordered cusps;
		\item[(c) ]\, it has exactly $s$ faces.
	\end{itemize}
\end{definition}

We will assign shear coordinates $s_i$ to the internal edges and \textbf{pinning variables} $k_i$ to the open edges (see Section \ref{se:FG}). 

\begin{example}\label{ex:s3toinf-1}
	As in Example \ref{ex:s3toinf}, the chewing-gum move splits the edge $s_3$ into an initial part of length $k_1$, a middle part of length $1/\epsilon$ and a final part of length $k_2$:
\begin{equation}
	\label{eq:s3k1k2}
	s_3= k_1-\log[\epsilon]+k_2, \qquad \epsilon\to 0.
\end{equation}
In this limit, the fixed points \eqref{eq:fixpts} become
\begin{equation}\label{eq:fixpts-lim}
	q_2^{(1)}\mapsto -1-e^{-s_2},\quad
	q_2^{(2)}\mapsto \infty,
\end{equation}
so the dashed blue geodesic in \Cref{fig:pair-of-pants} meets the yellow one asymptotically at a point of the real axis, and the solid blue one meets it at $\infty$. This creates two cusps, at $\tilde q_2^{(1)}$ and at $\infty$ (\Cref{fig:cusped-case}, left).

	\begin{figure}[!htb]
		\centering
		\includegraphics[width=0.7\textwidth]{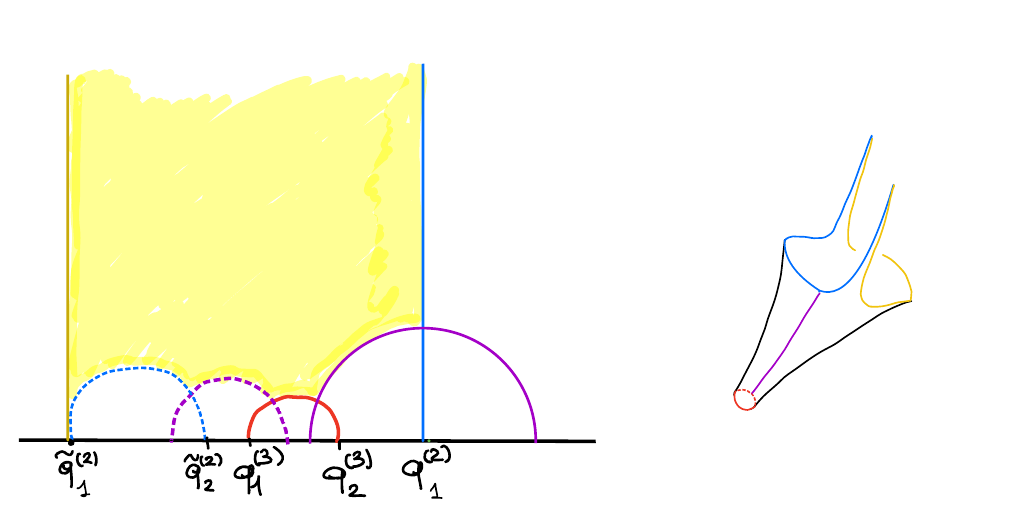}
		\caption{On the left, the yellow region corresponds to the finite part of a pair of pants. On the right we identify the two purple geodesics to produce a cylinder with two bordered cusps on one of the boundaries.}\label{fig:cusped-case}
	\end{figure}   
	
	Note that the two elements $\gamma_1$ and $\gamma_2$ don't survive the limit as elements of $\mathbb P SL_2(\mathbb R)$:
	$$
	\gamma_1(z)\mapsto\tilde  q_2^{(1)},\quad 
	\gamma_2(z)\mapsto \infty,
	$$
	hence we only have one element left, which identifies the dashed and solid purple geodesics. By identifying these geodesics, from the yellow region on the left hand side of \Cref{fig:cusped-case},  we obtain a sphere with two holes, one of which has two bordered cusps as displayed on the right-hand side of \Cref{fig:cusped-case}. 
	Note that again we can triangulate this surface by infinite arcs, see 
	\Cref{fig:cusped-tri}. 
	
	\begin{figure}[!htb]
		\centering
		\includegraphics[width=0.7\textwidth]{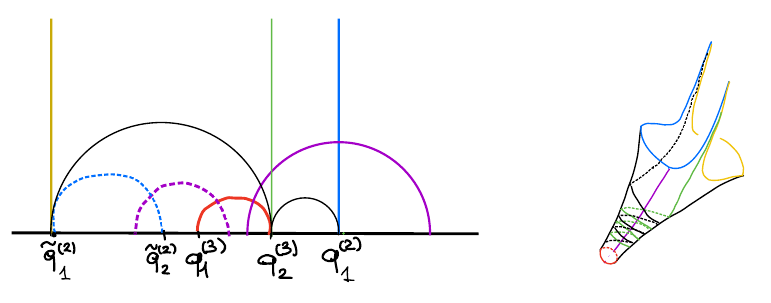}
		\caption{On the right, we have two infinitely winding geodesics in green and in black that triangulate the surface. On the left, we have drawn the corresponding geodesics in $\mathbb H$ in the same colors.}\label{fig:cusped-tri}
	\end{figure}   
	
	We now obtain a region made up of two ideal triangles that are indeed dual to the fat-graph obtained in \Cref{fig:chewing-graph}
	see \Cref{fig:cusped-fat}.
	
	\begin{figure}[!htb]
		\centering
		\includegraphics[width=0.7\textwidth]{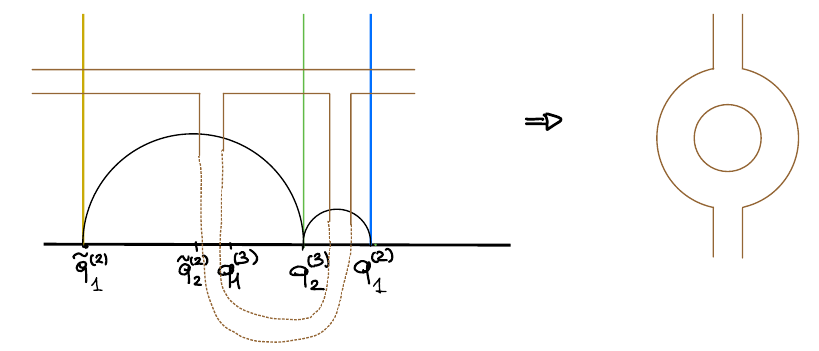}
		\caption{On the left, the ideal triangulation in $\mathbb H$ and its dual fat-graph which is the same as the one on the right, namely, the fat-graph obtained in \Cref{fig:chewing-graph}.}\label{fig:cusped-fat}
	\end{figure}   
	
	Observe that now the ideal triangulation contains segments on the boundary.

\end{example}

Let us now explain how to determine a metric on $\Sigma_{g,s,m}$ uniquely. Namely, how to fix a conformal map from $\Sigma^f_{g,s,m}$ to a region in $\mathbb H$ whose quotient by the action of a suitable group gives $\Sigma^f_{g,s,m}$.

Again, rather than the general case, we focus on our favorite example.

\begin{example}\label{ex:cyl-lenght}
	As we saw in Example \ref{ex:s3toinf-1}, by colliding two holes in a pair of pants we obtain a cylinder with two bordered cusps on one of its boundaries. All bottleneck geodesics are destroyed except for one.
	The information about the length of the only surviving bottleneck geodesic is not sufficient to define the metric uniquely in our cylinder. Let us see what extra data is needed in order to fix the metric on the two-cusped cylinder.
	
	First, let us show that the information about the position of one of the bordered cusps and the length of the red bottleneck curve are enough to determine the yellow region in \Cref{fig:cusped-case} uniquely. To see this, we proceed in a similar way to the proof of Lemma \ref{lm:hexagon}. Let us start by fixing a solid purple geodesic and a point $z_0$ on it - this can be done without loss of generality thanks to the action of $\mathbb PSL_2(\mathbb R)$.
	Pick the unique red geodesic orthogonal to the purple one at $z_0$. Cut a segment of length $l$ from $z_0$ on the red geodesic and pick the unique orthogonal geodesic at the end point of this segment (dashed purple). Now assume we know the position of one of the two cusps, for example let us place this cusp at infinity. 
	Then there is a unique solid blue geodesic from the cusp at infinity orthogonal to the solid purple geodesic. The solid blue geodesic cuts a unique segment on the purple geodesic. We now cut a segment on the dashed purple geodesic of the same length. This allows us to pick a unique dashed blue geodesic orthogonal to the dashed purple one. Finally we pick uniquely the yellow geodesic connecting infinity with the intersection of the dashed blue geodesic with the real axis.
	
	This shows that the yellow region  in \Cref{fig:cusped-case} is uniquely determined by picking the position of one cusp and the length $l$ of the bottleneck geodesic. 
	
	However, we still have a problem: in the yellow region, we have infinitely long geodesics because we have the two bordered cusps on the absolute, therefore any geodesic starting or terminating in any of these bordered cusps will have infinite length.  In order to still be able to uniquely determine the length of all geodesics in the yellow area in \Cref{fig:cusped-case}, we need to fix two horocycles as in  
	\Cref{fig:cusped-case-1}. 
	
	\begin{figure}[!htb]
		\centering
		\includegraphics[width=0.7\textwidth]{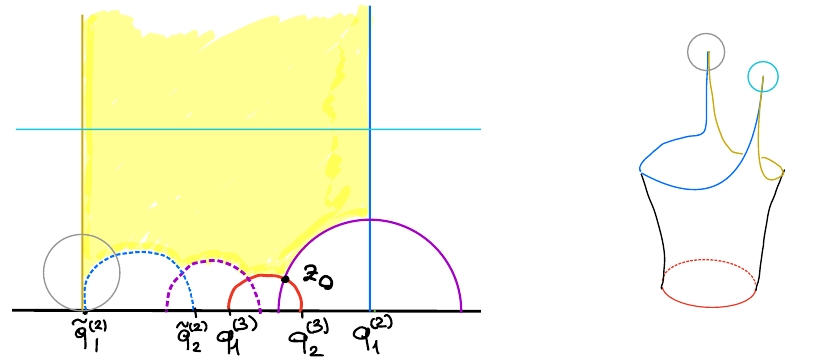}
		\caption{Two horocycles, one at infinity in cyan and one at $\tilde q_1^{(2)}$ shown on the left in $\mathbb H$ and on the right on the cylinder with two bordered cusps, giving rise to a ``candle cake''.}\label{fig:cusped-case-1}
	\end{figure}   
	Fixing the Euclidean diameters of the two horocycles once and forever, we can determine the length of the portion of the infinite geodesics outside the two horocycles. This quantity is called $\lambda$\textbf{-length} and was introduced by Penner.
	
	In conclusion, to completely fix the metric in the yellow area in \Cref{fig:cusped-case} and the lengths of all geodesics in it, we need $4$ real parameters: the length of the bottleneck geodesic, the position of one of the cusps and the two Euclidean diameters of the horocycles.
	
	In \cite{ChM} it was proved that the Euclidean diameters of these horocycles are related to $k_1,k_2$ in \eqref{eq:s3k1k2}. 
\end{example}

Generalizing, colliding $2$ holes in  $\Sigma_{g,s}^f$, we produce $\Sigma_{g,s-1,2}^f$. 
The fat-graph of  $\Sigma_{g,s}$ had $6 g-6 + 3 s$ edges. In the collision, we open one edge and replace its coordinate by two pinning variables associated to the horocycles, so that the resulting fat graph  has $6 g-6 + 3 s+1 =6 g-6 + 3 (s-1)+4$ edges of which $2$ are open.

Therefore we expect that the analogue of the Teichm\"uller space for $\Sigma_{g,s-1,2}$  should have dimension 
$6 g-6+3 s+2 m$. In the next section, we are going to explain what is the correct notion of the analogue of the Teichm\"uller space for $\Sigma_{g,s,m}$.

\subsection{Bordered cusped Teichm\"uller space}

In this subsection, we explain how to modify the 
definition of Teichm\"uller space as 
$$\opn{Hom}'(\pi_1(\Sigma_{g,s}), \mathbb P SL_2(\mathbb R))/\mathbb P SL_2(\mathbb R),
$$
when we have a surface with marked boundary. First of all, the fundamental group is not enough to distinguish between $\Sigma_{g,s,m}$ and $\Sigma_{g,s}$. Therefore, we replace it with the set of homotopy classes of paths that start and end in (not necessarily distinct) bordered cusps. This set forms a groupoid under path composition, and it is called \textbf{fundamental groupoid} $\pi_1(\Sigma,P)$, where $P$ is the set of bordered cusps. 

We still want to take representations of the fundamental groupoid in $\mathbb P SL_2(\mathbb R)$. However, while before the Procesi coordinates were equivalent to lengths of closed geodesics, and therefore we could take a quotient by $\mathbb P SL_2(\mathbb R)$ (whose action preserves these lengths), now our coordinates are given by $\lambda$-lengths of infinite arcs that start and end at some bordered cusps. These lengths are uniquely determined only by fixing a horocycle at each bordered cusp once and forever. So, we can't act by conjugation of the whole $\mathbb P SL_2(\mathbb R)$, because this would mess up the horocycles, we need to act by multiplication by parabolic elements that preserve the horocycle at the origin of the arc and the one at the end (see \Cref{lm:hor-par}):
$$
\gamma\mapsto \gamma_s\cdot\gamma\cdot \gamma_t,
$$
where $\gamma_s$ and $\gamma_t$ are two parabolic elements.

In other words, we give the following definition:
the \textbf{bordered cusped Teichm\"uller space} is the following set:
$$
\mathcal T_{g,s,m}:=\opn{Hom}'(\pi_1(\Sigma_{g,s,m},P), \mathbb PSL_2(\mathbb R))/U_P,
$$
where $P$ is the set of bordered cusps and $U_P:=\sqcup_{p\in P} U_p$, each $U_p$ being the  subgroup of  parabolic elements in $\mathbb PSL_2(\mathbb R)$.

Let us show that the dimension of $\mathcal T_{g,s,m}$ is what we predicted at the end of the previous subsection, namely $6 g-6+3 s+2 m$. 
In fact, let us fix a bordered cusp as base point, say $p_0$, then we have $2g$ matrices for the usual $A$- and $B$-cycles starting and terminating at $p_0$, $s-1$ matrices corresponding to going around all
holes except the one to which the cusp $p_0$ belongs, $m-1$ matrices corresponding to paths
starting at $p_0$ and terminating at other cusps. Each matrix depends on three independent
coordinates, giving $3(2g + s- 1 + m- 1)$, by taking the quotient by $U_P$ we obtain  $6 g-6+3 s+2 m$.

We are now going to discuss what is the coordinate ring of
$\mathcal T_{g,s,m}$. In \Cref{suse:procesi}, we saw that, in the absence of bordered cusps, the coordinate ring was generated by the Fricke-Vogt coordinates, which arise as traces of generating matrices and their products. The trace was a natural operation to take because $\mathcal T_{g,s}$ was obtained from $\opn{Hom}'(\pi_1(\Sigma_{g,s}),\mathbb PSL_2(\mathbb R))$ by taking the quotient by $\mathbb PSL_2(\mathbb R)$. The action of $\mathbb PSL_2(\mathbb R)$ on the representing matrices is conjugation, and the trace is invariant under conjugation.

In the presence of bordered cusps, 
the conjugation by $\mathbb PSL_2(\mathbb R)$ is replaced by multiplication by a parabolic element associated to the bordered cusp at the beginning of the path and a parabolic element associated to the bordered cusp at the end of the path. 
When initial and final points are different, the two  parabolic elements are different (because they have different fixed points) therefore the trace is replaced by another operation that must be invariant under this mixed conjugation action. If we generate the parabolic subgroup by lower triangular elements, then the invariant operation is the one picking the element in position $12$ of a matrix.
We therefore introduce the following
\begin{equation}\label{eq:trK}
	\opn{Tr_K}(A):= \opn{Tr}(A K),\,
	\text{ where }\,
	K=\left(\begin{array}{cc}
		0&0\\
		-1&0\\
	\end{array}
	\right).
\end{equation}
This intuitive reasoning however does not explain why we need the matrix $K$ when we start and end at the same point. More accurately, we can motivate the introduction of the matrix $K$ by observing that if the fat-graph goes around a boundary as in \Cref{fig:k-matrix},
then any geodesic containing the blue segment in \Cref{fig:k-matrix} will be represented by a matrix of the form 
$$
\dots R X(s) R \dots.
$$

\begin{figure}[!htb]
	\centering
	\includegraphics[scale=1]{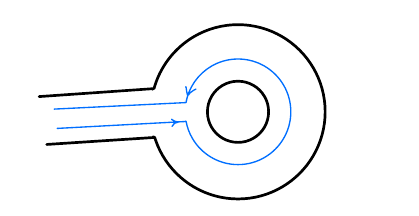}
	\caption{A generic geodesic (in blue) circling a hole.}\label{fig:k-matrix}
\end{figure}

Setting
$$
s=k_1+k_2- \log(\epsilon),
$$
one obtains
$$
\dots R X(s) R \dots = \dots R X(k_1+k_2)S X(-\log(\epsilon)) R \dots =
\dots R X(k_1+k_2)\left(\frac{1}{\epsilon} K +\mathcal O(\epsilon)\right)\dots,
$$    
where $S$ is defined as in \eqref{eq:FG-matrices}.

\begin{example}\label{ex:PV}
	Consider the sphere with $4$ boundaries in Example \ref{ex:PVI}, and substitute
	$$
	p_3=k_1+k_2- 2\log(\epsilon),
	$$
	in \eqref{eq:x-s}. Taking the limit as $\epsilon\to 0$, we produce a sphere with three boundaries one of which carries two bordered cusps as in \Cref{fig:candlePV}.
    
\begin{figure}[!htb]
	\centering
	\includegraphics[scale=0.8]{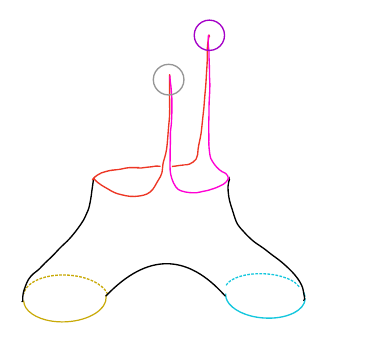}
	\caption{The candle cake corresponding to a sphere with three boundaries one of which carries two bordered cusps.}\label{fig:candlePV}
\end{figure}

Since $\hat s_i=s_i+\frac{p_i}{2}$, we will have also
	$$
	\hat s_3= s_3 + k_1+k_2- \log(\epsilon),
	$$
	therefore calling $\tilde s_3= s_3 + k_1+k_2$, we obtain:
	\begin{equation}\label{eq:x-s-limit}
		\begin{split}
			x_1= - \frac{1}{\epsilon}e^{\hat s_2+\tilde s_3}-  \frac{1}{\epsilon} e^{-\hat s_2+\tilde s_3} - \epsilon e^{-\hat s_2-\tilde s_3}-\left(\frac{1}{\epsilon}e^{\frac{k_1+k_2}{2}}+\epsilon e^{-\frac{k_1+k_2}{2}}\right)e^{-\hat s_2}-\frac{1}{\epsilon}\left(e^{\frac{p_2}{2}}+e^{-\frac{p_2}{2}}\right)e^{\tilde s_3}     ,\\
			x_2= - \frac{1}{\epsilon}e^{\tilde s_3+\hat s_1}- \epsilon e^{-\tilde s_3+\hat s_1} -\epsilon e^{-\tilde s_3-\hat s_1}-\epsilon\left(e^{\frac{p_1}{2}}+e^{-\frac{p_1}{2}}\right)e^{-\tilde s_3}-\left(\frac{1}{\epsilon}e^{\frac{k_1+k_2}{2}}+\epsilon e^{-\frac{k_1 + k_2}{2}}\right)e^{\hat s_1}     ,\\
			x_3= - e^{\hat s_1+\hat s_2}-  e^{-\hat s_1+\hat s_2} - e^{-\hat s_1-\hat s_2}-\left(e^{\frac{p_2}{2}}+e^{-\frac{p_2}{2}}\right)e^{-\hat s_1}-\left(e^{\frac{p_1}{2}}+e^{-\frac{p_1}{2}}\right)e^{\hat s_2}     ,\\
			G_i= -e^{\frac{p_i}{2}}-e^{-\frac{p_i}{2}},\quad\text{for }\, i=1,2,\quad  G_3 = \frac{1}{\epsilon}e^{\frac{k_1+k_2}{2}} +\epsilon e^{-\frac{k_1+k_2}{2}},\\
			G_4= - \frac{1}{\epsilon} e^{\hat s_1+\hat s_2 + \tilde s_3}-\epsilon e^{-\hat s_1-\hat s_2 - \tilde s_3}.\\
		\end{split}
	\end{equation}
	Taking the limits
	$$
	\tilde x_1 := \lim_{\epsilon\to 0}\epsilon x_1, \quad \tilde x_2 := \lim_{\epsilon\to 0}\epsilon x_2,\quad
	\tilde G_3 := \lim_{\epsilon\to 0}\epsilon G_3,\quad \tilde G_4 := \lim_{\epsilon\to 0}\epsilon G_4,
	$$
	and setting $\tilde x_3=x_3$, $\tilde G_1=G_1$, and $\tilde G_2=G_2$, we obtain that $\tilde x_1,\tilde x_2,\tilde x_3,\tilde G_1,\tilde G_2,\tilde G_3,\tilde G_4$ satisfy the following cubic relation
	\begin{equation}
		\label{eq:FrickePV}
		\begin{split}
			&\tilde x_1\tilde  x_2\tilde  x_3 + \tilde x_1^2+\tilde x_2^2 -(\tilde G_4 \tilde G_1+\tilde G_2 \tilde G_3) \tilde x_1 -(\tilde G_4 \tilde G_2+\tilde G_1 \tilde G_3) \tilde x_2 -(\tilde G_4 \tilde G_3) \tilde x_3 +\\
			&\tilde G_3^2 +\tilde G_4^2 + \tilde G_1 \tilde G_2 \tilde G_3 \tilde G_4=0.
		\end{split}
	\end{equation}
	Relation \eqref{eq:FrickePV} defines a $3$-parameter pencil of affine cubic surfaces in $\mathbb R^3$. To see that \eqref{eq:FrickePV} depends only on $3$ parameters, we can rescale $\tilde x_1, \tilde x_2, \tilde G_4, \tilde G_3$ by say $\tilde G_3^{-1}$ to eliminate the parameter $\tilde G_3$.
	
	Note that now we can't claim that the $3$-parameter affine cubic pencil \eqref{eq:FrickePV} is the bordered cusped Teichm\"uller space because the latter is $7$-dimensional, while \eqref{eq:FrickePV} is $5$ dimensional. This fact can also be seen by observing that the shear coordinates $s_1,s_2,s_3,p_1,p_2,k_1,k_2$ can't be obtained in terms of $\tilde x_1,\tilde x_2,\tilde x_3,\tilde G_1,\tilde G_2,\tilde G_3,\tilde G_4$ by inverting relations \eqref{eq:x-s-limit} because, for example, $k_1$ and $k_2$ only appear as linear combination $k_1+k_2$.
\end{example}

Due to the fact that, as seen in the above Example, the limits of Fricke-Vogt coordinates are not enough to coordinatize the  bordered cusped
Teichm\"uller space, in \cite{ChM}
a complete combinatorial description of the bordered cusped
Teichm\"uller space was given by introducing the notion of \textbf{maximal cusped lamination}, a collection of 
geodesic arcs between bordered cusps and closed geodesics homotopic to the boundaries, such that they have no intersections nor
self-intersections in the interior of a Riemann surface, but can be incident to the same bordered cusp, and such that they
triangulate the Riemann surface. More formally:

\begin{definition}\label{def:geom-laminations}
	We call 
	\emph{cusped geodesic lamination} (CGL) on a bordered cusped Riemann surface a set of nondirected curves up to a homotopic equivalence such that
	\begin{itemize}
		\item[(a)] these curves are either closed curves ($\gamma$) or \emph{arcs} ($\mathfrak a$) that start and terminate at bordered cusps 
		(which can be the same cusp);
		\item[(b)] these curves have no (self)intersections inside the Riemann surface (but can be incident to the same bordered cusp);
		\item[(c)] these curves are not empty loops or empty loops starting and terminating at the same cusp.
	\end{itemize} 
\end{definition}

In \cite{ChM} it was proved that,
for any surface $\Sigma_{g,s,m}$ of genus $g$ with $s$ boundaries and $m\geq 1$ bordered cusps,
there exists a complete cusped geodesic lamination consisting of $6g-  6 + 3s + 2m$ geodesic arcs and that
the lengths of the portions of the geodesic arcs contained between the fixed horocycles at their endpoints coordinatize the bordered cusped Teichm\"uller space. We have seen this in the case of the cylinder with two bordered cusps on one boundary, see \Cref{ex:cyl-lenght}, and now we illustrate this construction in the case of a Riemann sphere with three boundaries and two bordered cusps on one boundary.

\begin{example}\label{ex:PV-lam}
	In Example \ref{ex:PV}, the case of a Riemann sphere with three boundaries and two bordered cusps on one boundary was obtained by taking the limit as $\epsilon\to 0$ after the substitution $p_3=k_1+k_2-2\log(\epsilon)$. The corresponding fat-graph and lamination are shown in \Cref{fig:PVlamination}. 
	\begin{figure}[!htb]
		\centering
		\includegraphics[width=0.4\textwidth]{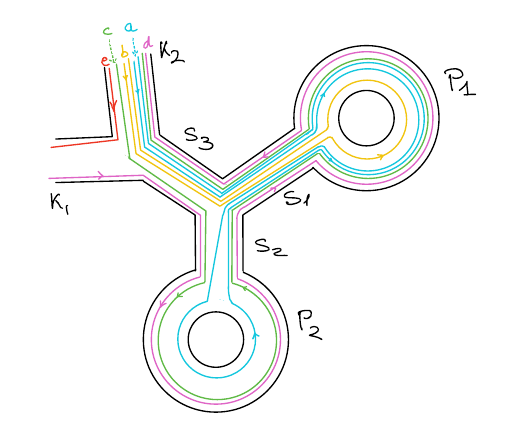}
		\caption{Fat graph of $\Sigma_{0,3,2}$. The associated shear coordinates are labeled by $s_1,s_2,s_3,k_1,k_2,p_1,p_2$. The lamination is composed of two loops around the two boundaries with no bordered cusps and five geodesic arcs denoted by $a,b,c,d,e$.}\label{fig:PVlamination}
	\end{figure}

	Calculating the matrices corresponding to the five geodesic arcs and taking the trace-$K$, we obtain:
	\begin{equation}
		\label{eq:lam-shear}
		\begin{split}
			\lambda_a=&\opn{Tr}_K\left( X(k_ 2) RX(s_ 3) RX(s_ 1) RX(p_ 1) RX(s_ 1) RX(s_ 2) RX(p_ 2) RX(
			s_ 2)\right.\\
			& \left. LX(s_ 1) LX(p_ 1) LX(s_ 1) LX(s_ 3) L X(k_ 2) \right)=e^{k_2+ p_1+\frac{p_2}{2}+2 s_1+s_2+s_3},\\
			\lambda_b=&\opn{Tr}_K\left(X(k_ 2) RX(s_ 3) RX(s_ 1) RX(p_ 1) RX(s_ 1) LX(s_ 3) LX(k_ 2) \right)=e^{k_2+\frac{p_1}{2}+s_1+s_3}\\
			\lambda_c=&\opn{Tr}_K\left(X(k_ 2) RX(s_ 3) RX(s_ 1) RX(p_ 1) RX(s_ 1) RX(s_ 2) RX(p_ 2)\right.\\
             &\left. RX(
			s_ 2) RX(s_ 3) LX(k_ 2) \right)=e^{k_2+\frac{p_1}{2}+\frac{p_2}{2}+s_1+s_2+s_3}\\
			\lambda_d=&\opn{Tr}_K\left(X(k_ 2) RX(s_ 3) RX(s_ 1) RX(p_ 1) RX(s_ 1) RX(s_ 2) RX(p_ 2) RX(
			s_ 2) \right.\\
            &\left. RX(s_ 3) RX(k_ 1) K\right)=e^{\frac{k_1}{2}+\frac{k_2}{2}+\frac{p_1}{2}+\frac{p_2}{2}+s_1+s_2+s_3}\\
			\lambda_e=&  \opn{Tr}_K\left(X(k_ 2) RX(k_ 1) \right)=e^{\frac{k_1}{2}+\frac{k_2}{2}},
		\end{split}
	\end{equation}
	while the lengths of the two loops around the boundaries that don't have bordered cusps are given by
	$$
	G_1= \opn{Tr}(R X(p_1)R)= e^{\frac{p_1}{2}}+ e^{-\frac{p_1}{2}},\quad
	G_2= \opn{Tr}(R X(p_2)R)= e^{\frac{p_2}{2}}+ e^{-\frac{p_2}{2}}.
	$$
	
	Note that all $\lambda$-lengths are now monomials in the exponentiated  shear coordinates and therefore we can invert all formulae, namely we can express the shear coordinates in terms of $\lambda$-lengths. 
\end{example}

In Example \ref{ex:PV-lam}, we observed that all $\lambda$-lengths of arcs in the geodesic lamination are actually  monomials in the exponentiated  shear coordinates. This is a general fact \cite{ChM}. 
\section{Fock--Goncharov theory}\label{se:FG}

The basic idea underlying the Fock--Goncharov coordinatization is very similar to what we have seen so far: we triangulate the surface by ideal triangles and decompose any path (up to homotopy) in the surface into sections according to which triangle it crosses and how. We always assume paths are in general position with respect to the triangulation, so that they cross edges transversely and do not pass through vertices. Each triangle contributes to a matrix factor (which depends on the way in which the path crosses it), and the ordered multiplication of all these factors gives us the matrix associated to the given path. The advantage of Fock--Goncharov theory is that it allows us to associate matrices in any simple real Lie group $G$. For the purposes of these notes, we restrict to $G=\mathbb PSL_n(\mathbb R)$. 

In order to understand how each triangle contributes to a matrix factor, let us focus on a single triangle $\triangle_{123}$ of vertices labeled $1,2,3$ in clockwise direction as in \Cref{fig:transport}:

\begin{figure}[!htb]
	\centering
	\includegraphics[scale=1.5]{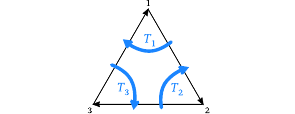}
	\caption{The triple of transport matrices on the oriented triangle $\triangle_{123}$. $T_1$ corresponds to the map of oriented sides $12\mapsto31$, $T_2$ to $23\mapsto12$ and $T_3$ to $31\mapsto23$.}\label{fig:transport}
\end{figure}

In this section, we explain how to construct the three transport matrices $T_1,T_2,T_3\in\mathbb PSL_n(\mathbb R)$ corresponding to crossing the triangle as in  \Cref{fig:transport}. By definition, these matrices will be such that  $T_1T_2T_3=\mathbb I$.

Before doing so in general, let us think for a moment about the example of the pair of pants. We saw that the pair of pants is triangulated by oriented ideal triangles. 
Because the upper half-plane $\mathbb H$ arises as one of the two connected components of 
$\mathbb P\mathbb C^1\setminus \mathbb P\mathbb R^1$, we can interpret the absolute as the real projective line  - pragmatically, any point $x$ of $\mathbb R$ is interpreted as $(x:1)$ in homogeneous coordinates, and $\infty$ corresponds to $(1:0)$. 
Therefore,  the images of the vertices of the ideal triangulation in $\mathbb H$ naturally define points in $\mathbb P\mathbb R$. We can think of these points as lines in $\mathbb R^2$, so that at each vertex of the triangulation, we associate a one dimensional subspace of $\mathbb R^2$, or, in other words, a complete flag.
More generally, in any rank, a complete flag is defined as follows:

\begin{definition}
	A \emph{complete flag} $F_\bullet$ in a vector space $V$ of dimension $n$ is a collection of consecutively embedded subspaces
	\begin{equation*}
		\{0=F_0 \subset F_1 \subset \ldots \subset F_{n-1} \subset F_n=V\},\quad \dim(F_k)=k.
	\end{equation*}
\end{definition}

In Fock--Goncharov theory, a complete flag in $\mathbb R^n$ is assigned to every vertex of the triangulation and it is required that these flags are in general position, according to the following definition:

\begin{definition}
	\label{df:gen-pos}
	Let $V$ be a vector space of dimension $n$ over $\mathbb R$. 
	Given  three flags $F^{(1)}_\bullet, F^{(2)}_\bullet,F^{(3)}_\bullet$ in $V$, they are said to be in \textbf{general position} if they are pairwise transverse and 
	\begin{equation}
		\label{mindimkflags}
		\operatorname{dim} \left( \bigcap_{j = 1}^3 F^{(j)}_{i_j} \right)  = \operatorname{max}\left\{ \sum_{j=1}^3 \dim(F^{(j)}_{i_j}) - 2 n, 0 \right\},
	\end{equation}
	for $1 \leq i_j \leq n$.
\end{definition}

Note that given three flags $F^{(1)}_\bullet, F^{(2)}_\bullet,F^{(3)}_\bullet$ in general position, any two of them are pairwise transverse, namely  
\begin{equation*}
	\operatorname{dim} \left(  F^{(k)}_{i}\cap F^{(l)}_{j} \right)  = \operatorname{max}\left\{ \dim(F^{(k)}_{i})+\dim( F^{(l)}_{j}) - n, 0 \right\},
\end{equation*}
for $1 \leq i,j \leq n$ for all $j$ and for any $k\neq l, k,l \in \{1,2,3\}$.

\begin{example}\label{ex:2x2flags}
	Take any of the triangles in \Cref{fig:winding}, for example the one of vertices $q^{(1)}_1,\gamma_2(q_2^{(3)}),q_2^{(2)}$ see formulae \eqref{eq:fixpts}. This allows us to assign the following flags in $\mathbb R^2$ to these vertices:
	$$
	F_1^{(1)}=\left\langle\left(
	\begin{array}{c}
		-\frac{e^{s_3}(1+e^{s_2})}{e^{s_2+s_3}-1} \\
		1
	\end{array}
	\right)\right\rangle \subset F_2^{(1)}=\mathbb R^2,
	$$
	$$
	F_1^{(2)}=\left\langle\left(
	\begin{array}{c}
		\frac{e^{s_1+s_3}(1+e^{s_2})}{1+e^{s_1}}  \\
		1
	\end{array}
	\right)\right\rangle \subset F_2^{(2)}=\mathbb R^2,
	$$
	$$
	F_1^{(3)}=\left\langle\left(
	\begin{array}{c}
		\frac{e^{s_1+s_3}-1}{1+e^{s_1}} \\
		1
	\end{array}
	\right)\right\rangle \subset F_2^{(3)}=\mathbb R^2,
	$$
	
	These flags are always pairwise transverse for real $s_1,s_2,s_3$. 
\end{example}

\begin{example}\label{ex:3x3flags}
	Consider $\mathbb R^3$ with the canonical basis $e_1,e_2,e_3$ and the three flags
	$$
	0=F^{(1)}_0\subset \langle e_1 \rangle=F^{(1)}_1\subset  \langle e_1, e_2 \rangle=F^{(1)}_2 \subset \mathbb R^3=F^{(1)}_3,
	$$
	$$
	0=F^{(2)}_0\subset \langle e_3 \rangle=F^{(2)}_1\subset  \langle e_3, e_2 \rangle=F^{(2)}_2 \subset   \mathbb R^3=F^{(2)}_3,
	$$
	$$
	0=F^{(3)}_0\subset \langle e_1 +\alpha e_2+\beta e_3 \rangle=F^{(3)}_1\subset  \langle  e_1 +\alpha e_2+\beta e_3 , e_2+\gamma e_3  \rangle=F^{(3)}_2 \subset  \mathbb R^3=F^{(3)}_3,
	$$
	let us see under what conditions they are in general position.
	First we check pairwise transversality. The flags  $F^{(1)}_\bullet$ and $F^{(2)}_\bullet$ are opposite and therefore transverse.
	The flags $F^{(2)}_\bullet$ and $F^{(3)}_\bullet$ are also transverse for any choice of $\alpha,\beta,\gamma$.
	However, the flags  $F^{(1)}_\bullet$ and $F^{(3)}_\bullet$ 
	are transverse only for $\beta\neq 0$ and $\beta\neq \alpha\gamma$. This is because
	$$
	\dim(F_1^{(3)}\cap F^{(1)}_2)=0 \Leftrightarrow \beta\neq 0
	$$
	and
	$$
	\dim(F_2^{(3)}\cap F^{(1)}_1)=0 \Leftrightarrow \beta\neq \alpha\gamma.
	$$
	Therefore we need to assume  $\beta\neq 0$ and $\beta\neq \alpha\gamma$ to have pairwise transversality. In order to have general position, we see that we need $\gamma\neq 0$. In fact
	$$
	\dim(F_2^{(1)}\cap F^{(2)}_2\cap F^{(3)}_2)=0 \Leftrightarrow \gamma\neq 0.
	$$
\end{example}

\begin{remark}\label{rmk:gen-flag}
	Observe that any triple of flags in general position can always be brought by global conjugation in the form of Example \ref{ex:3x3flags}.
\end{remark}

For each edge of the triangle, the two complete transversal flags at its extrema uniquely determine a canonical choice of lines $\Lambda_1,\dots,\Lambda_n$ that give a splitting of $\mathbb R^n$:

\begin{lemma}
	\label{lemma2flagsDecomposition}
	Let $F^{(1)}, F^{(2)}$ be two complete flags in general position.
	Then, there exists a unique splitting of $\mathbb R^n$ into lines 
	\begin{equation}
		\label{VdecompositionW}
		\mathbb R^n= \Lambda_1 \oplus \dots \oplus \Lambda_n,
	\end{equation}
	such that, for any $1 \leq i \leq n$,
	\begin{equation}
		\label{flagsF&G}
		\begin{aligned}
			F^{(1)}_i &= \Lambda_1 \oplus \dots \oplus \Lambda_i,\\
			F^{(2)}_{i} &= \Lambda_{n} \oplus \dots \oplus \Lambda_{n-i+1}.
		\end{aligned}
	\end{equation}
\end{lemma}

We don't prove this lemma, we show how it works in our two examples.

\begin{example}\label{ex:2x2flags1}
	Consider the three flags $F^{(1)}_\bullet, F^{(2)}_\bullet,F^{(3)}_\bullet$ of Example \ref{ex:2x2flags}. Then,     
	we have the following splittings of $\mathbb R^2$:
	\begin{equation*}
		\begin{split}
			F^{(1)}_\bullet , F^{(2)}_\bullet:& \quad  F^{(1)}_1 \oplus   F^{(2)}_1 \\
			F^{(2)}_\bullet , F^{(3)}_\bullet:& \quad  F^{(2)}_1 \oplus   F^{(3)}_1 \\
			F^{(3)}_\bullet , F^{(1)}_\bullet:& \quad  F^{(3)}_1 \oplus   F^{(1)}_1.
		\end{split}
	\end{equation*}
\end{example}

\begin{example}\label{ex:3x3flags1}
	Consider the three flags $F^{(1)}_\bullet, F^{(2)}_\bullet,F^{(3)}_\bullet$ of Example \ref{ex:3x3flags} with $\beta\neq 0$, $\beta\neq \alpha\gamma$ and $\gamma\neq 0$. Then, we have the following splittings of $\mathbb R^3$:
	\begin{equation}\label{eq:split3x3}
		\begin{split}
			F^{(1)}_\bullet , F^{(2)}_\bullet:&\quad   \langle e_1\rangle\oplus  \langle e_2\rangle \oplus  \langle e_3\rangle\\
			F^{(2)}_\bullet , F^{(3)}_\bullet:&\quad     \langle e_3\rangle\oplus  \langle e_2+\gamma e_3\rangle \oplus  \langle e_1+\alpha e_2+\beta e_3\rangle \\
			F^{(3)}_\bullet , F^{(1)}_\bullet:&\quad   \langle e_1+\alpha e_2+\beta e_3\rangle  \oplus  \langle \gamma e_1+(\alpha \gamma-\beta)e_2\rangle \oplus \langle e_1\rangle
		\end{split}
	\end{equation}
\end{example}

As explained at the beginning of this section, one would like to associate some matrices $T_1,T_2$, $T_3\in \mathbb PSL_n(\mathbb R)$ to paths crossing the triangle. The basic idea is to interpret these matrices as the ones that map the splitting associated to one side to the splitting associated to another side. 
As illustrated in the following example, such matrices will depend on the choice of the basis on each side, not only on the splitting.

\begin{example}\label{ex:3x3flags2}
	Here we compute three matrices that map between the splittings of example \ref{ex:3x3flags1}. In order to make this calculation, we need to pick some bases, for example
	$$
	w_1=\mu_1 e_1,\quad w_2=\mu_2 e_2,\quad  w_3=\mu_3 e_3,
	$$
	$$
	u_1=\eta_1 ( e_1+\alpha e_2+\beta e_3),\quad u_2=\eta_2 (\gamma e_1+(\alpha \gamma-\beta)e_2),\quad  u_3=\eta_3 e_1,
	$$
	$$
	v_1   =\sigma_1 e_3,\quad v_2=\sigma_2(e_2+\gamma e_3), \quad v_3= \sigma_3( e_1+\alpha e_2+\beta e_3).
	$$
	These bases depend on the choice of the real non zero numbers $\mu_1,\mu_2,\mu_3,\eta_1,\eta_2,\eta_3,\sigma_1,\sigma_2,\sigma_3$. We now calculate $T_1$ such that $T_1 w_i^T=u_i^T$, $T_2$ such that $T_2 v_i^T=w_i^T$ and $T_3$ such that $T_3 u_i^T=v_i^T$. They
	are given by\footnote{The choice of acting by $T_i$ on row vectors rather than column vectors will be clearer later when we describe the combinatorial construction.}
	\begin{equation}\label{eq:T1T2T3}
		\begin{split}
			T_1=\left(\begin{array}{ccc}
				\frac{\eta_1}{\mu_1}   & \frac{\eta_1}{\mu_2}\alpha   &  \frac{\eta_1}{\mu_3}\beta \\
				\frac{\eta_2}{\mu_1}\gamma & \frac{\eta_2}{\mu_2}(\alpha\gamma-\beta) & 0\\
				\frac{\eta_3}{\mu_1} & 0&0 \\
			\end{array}\right),\quad
			T_2=\left(\begin{array}{ccc}
				\frac{\mu_1}{\sigma_1}(\alpha\gamma-\beta)&
				-  \frac{\mu_1}{\sigma_2}\alpha&\frac{\mu_1}{\sigma_3}\\
				-\frac{\mu_2}{\sigma_1}\gamma& \frac{\mu_2}{\sigma_2}& 0
				\\
				\frac{\mu_3}{\sigma_1}&0&0\\
			\end{array}\right),\\
			T_3=\left(
			\begin{array}{ccc}
				\frac{\sigma_1}{\eta_1} \frac{1}{\beta} &
				-\frac{\sigma_1}{\eta_2}  \frac{\alpha}{\beta(\alpha\gamma-\beta)}&\frac{\sigma_1}{\eta_3}    \frac{\alpha}{(\alpha\gamma-\beta)}\\
				\frac{\sigma_2}{\eta_1} \frac{\gamma}{\beta}&
				-\frac{\sigma_2}{\eta_2} \frac{1}{\beta}&0\\
				\frac{\sigma_3}{\eta_1} &0&0\\
			\end{array}
			\right).\qquad\qquad\qquad\end{split}
	\end{equation}
	As expected $T_1T_2T_3=\mathbb I$. 
	
	It is easy to see that $T_1,T_2,T_3$ depend on $8$ independent parameters, of which $2$ are overall scale, therefore actually only $6$ parameters to determine $T_1,T_2,T_3\in\mathbb P SL_n(\mathbb R)$ uniquely. This is consistent with the fact that to find matrices in $\mathbb P SL_n(\mathbb R)$  one needs to determine \textit{projective bases,} namely bases up to overall rescaling.
\end{example}

As seen in Example \ref{ex:3x3flags2}, 
in order to have a unique choice of $T_1,T_2,T_3\in \mathbb PSL_n(\mathbb R)$, the information about the $n$ lines on each side is not enough. We need more information: we need to complete the splitting of $\mathbb P\mathbb R^n$ to a \textbf{projective basis}. As we shall see in the next subsection, this corresponds to adding the information of an extra line forevery edge. This line is called \textbf{pinning}.

\subsection{Projective bases}

Let $V$ be a vector space over a field $\mathbb K$ of dimension $n$ and $\mathbb P(V)$ its projective space, i.e. $\mathbb P(V) =V \setminus \{0\} / \mathbb K^{*}$.
Let $p : V \setminus \{0\} \longrightarrow \mathbb P(V)$ be the canonical projection, i.e.
\begin{align*}
	p : V \setminus \{0\} &\longrightarrow \mathbb P(V)\\
	v &\longmapsto p(v) := \langle v \rangle.
\end{align*}

\begin{definition}\label{def:pr-b}
	A \textbf{projective basis} for $V$ is a choice of $n+1$ lines $\lambda_1, \dots, \lambda_n, \lambda_{n+1}\in \mathbb P V$ such that
	any $n$ of them split $V$, namely 
	\begin{equation}
		\label{projectiveBasisCondition}
		V= \lambda_{i_1} \oplus \dots \oplus \lambda_{i_{n}}.
	\end{equation}
	for any pairwise distinct integers $0 < i_1 < \dots < i_{n} \leq n+1$.
\end{definition}

A projective basis corresponds to a basis of $V$ up to common rescaling factor such that a specific linear combination is fixed, as the following lemma shows.

\begin{lemma}
	\label{lemSumProjectiveFramings}
	Let $\lambda_1, \dots, \lambda_{n+1}$ define a projective basis for $\mathbb PV$.
	Then, for any fixed $( {\alpha}_1, \dots, {\alpha}_n ) \in (\mathbb K^*)^n$, there exist $\{v_1, \dots, v_n\}$ basis of $V$ such that
	\begin{align}
		\label{projectiveBasis1}
		\lambda_i &= p(v_i) \quad \text{for all }1 \leq i \leq n,\\
		\label{projectiveBasis2}
		\lambda_{n+1} &= p({\alpha}_1 v_1 + \dots + {\alpha}_n v_n).
	\end{align}
	The basis $v_1,\dots,v_n$ is unique up to global rescaling.
\end{lemma}

\begin{proof}
	Since $\lambda_i \in \mathbb P(V)$ there exists $\overline v_i \in V$ such that $\lambda_i = p(\overline v_i)$ for all $1 \leq i \leq n+1$ - in other words, there exists a vector $\overline v_i$ such that $\lambda_i=\langle\overline v_i\rangle$.
	Applying condition \eqref{projectiveBasisCondition} with $i_l = l$ for $l = 1, \dots, n$, we have that
	$$
	V= \lambda_{1} \oplus \dots \oplus \lambda_{n},
	$$  
	so that $\overline v_1, \dots, \overline v_n$ must be linearly independent and 
	provide a basis of $V$.
	Therefore, there exist $\bar{\alpha}_1, \dots, \bar{\alpha}_{n+1} \in \mathbb K^*$ such that $\bar{\alpha}_1 \overline v_1 + \dots + \bar{\alpha}_{n+1}\overline v_{n+1} = 0$.
	Now, for any fixed $( {\alpha}_1, \dots, {\alpha}_n ) \in (\mathbb K^*)^n$, define $v_i := -\frac{\bar{\alpha}_i}{{\alpha}_i \bar{\alpha}_{n+1}} \overline v_i$, so that $p(v_i) = p(\overline v_i)=\lambda_i$ for 
	$i=1,\dots,n$ and 
	$$
	\lambda_{n+1}=p(\overline v_{n+1})= p \left(- \frac{\bar{\alpha}_1}{\bar{\alpha}_{n+1}} \overline v_1 - \dots - \frac{\bar{\alpha}_n}{\bar{\alpha}_{n+1}} \overline  v_{n} \right)=
	p({\alpha}_1 v_1 + \dots + {\alpha}_n v_n),
	$$
	as we wanted to prove. 
	
	Now suppose we have another basis $u_1,\dots, u_n$ such that 
	$\lambda_i=\langle u_i \rangle$ for $i=1,\dots,n$ and $\lambda_{n+1}= \langle {\alpha}_1 u_1 + \dots + {\alpha}_n u_n\rangle $ then $u_i=\beta_i v_i$ for some non zero constants $\beta_1,\dots, \beta_n$ and therefore 
	$\lambda_{n+1}= \langle {\alpha}_1\beta_1 v_1 + \dots + {\alpha}_n\beta_n v_n\rangle =\langle {\alpha}_1 v_1 + \dots + {\alpha}_n v_n\rangle $, therefore  $\beta_1=\beta_2=\dots =\beta_n$ as we wanted to prove.
\end{proof}

\subsection{Standard bases}\label{suse:st-b}

To avoid repeating the computations of $T_1,T_2,T_3$ forevery possible choice of pinnings, Fock and Goncharov introduce a standard projective basis associated with each side of the configuration. The idea is to perform the main calculation once in these standard projective bases and then convert to and from the bases determined by the chosen pinnings.

More precisely, to compute the matrices $T_1,T_2,T_3$, the process is decomposed into three steps. First, one maps the projective basis determined by a given pinning $\Lambda_{IJ}$ to the standard projective basis on that same side $IJ$. Second, one computes the \textbf{standard matrix}  that maps the standard projective basis on the side $IJ$
to the standard projective basis on the side $KL$; this is the combinatorial part of the computation that does not depend on the particular pinnings. Finally, one maps the standard projective basis on the side $KL$
to the projective basis associated with the chosen pinning 
$\Lambda_{KL}$.
In this way, the dependence on the specific pinnings is isolated in the first and last transformations, while the standard matrix between the two standard bases only needs to be computed once.

We now explain how to define a standard projective basis for each side of a triangle and how to construct the matrices that map the standard projective basis on one side to the standard projective basis on another. 

First of all, cover $\triangle_{123}$ by its unique tessellation of $n^2$ identical equilateral triangular tiles, alternated between upward and downward. This gives rise to a \textbf{triangle graph} $\triangle_n$ as defined here:\\

\begin{definition}
	We call the $n$-\textbf{triangle graph} the simple planar graph $\triangle_n$ whose set of vertices $\operatorname{Vtx}(\triangle_n)$ is	
	\begin{equation*}
		\operatorname{Vtx}( \triangle_n ) :=  \{ \left. ( {a}, {b}, {c}) \in \mathbb R^3 \, \right| \, {a} + {b} + {c} = n \text{ and } {a}, {b}, {c} \geq 0\} \cap \mathbb Z^3, 
	\end{equation*}
	and whose set of edges $\operatorname{Edg}(\triangle_n)$ is
	\begin{equation*}
		\operatorname{Edg}( \triangle_n ) := \{ \left. \{ {v}, {v}' \} \subset \operatorname{Vtx}( \triangle_n ) \right| \, v= (a,b,c) \text{ and } v'=(a',b',c') \text{ satisfy \eqref{eqn:conditionvertices}} \},
	\end{equation*}
	where
	\begin{align}
		a' &= a, & & b' = b - 1, & & c' = c + 1 & &\text{ or } \nonumber\\
		\label{eqn:conditionvertices}
		a' &= a + 1, & & b' = b, & & c' = c - 1 & &\text{ or } \\
		a' &= a - 1, & & b' = b + 1, & & c' = c .\nonumber
	\end{align}
\end{definition}

The collection of all edges of $\triangle_n$ provides a partition of $\triangle_{123}$ into $n^2$ faces, which we call \textbf{tiles}.
Specifically, any of these tiles is determined by three edges such that any two of them intersect in exactly one vertex.
In particular, each tile can either be visually \textit{upward}, or \textit{downward}. Note that we can endow $\triangle_n$ with the same orientation as $\triangle_{123}$, in the case of these lecture notes always clockwise.

A clear way to visualize the graph $\triangle_n$ is to place it in an orthogonal frame in $\mathbb R^3$ in such a way that the vertex $1$ has coordinate $(n,0,0)$, the vertex $2$ has coordinate $(0,n,0)$ and the vertex $3$ has coordinate $(0,0,n)$, so that  the vector $(1,1,1)$ pointing at the reader, as shown in \Cref{fig.triangle-graph}.

In this way, each vertex in $\triangle_n$ is coordinatized by a triple of integers $(i,j,k)$ such that $i+j+k=n$. These coordinates are called \textbf{barycentric coordinates}.

Exploiting this presentation, there is a convenient way to label the upward and downward tiles of $\triangle_{n}$.
It is obtained by projecting the vertices of $\triangle_{n-1}$ and $\triangle_{n-2}$ onto the plane of  $\triangle_{n}$ along the $(1,1,1)$-axis, as shown in \Cref{fig.triangle-graph}.

\begin{figure}[!htb]
	\centering
	\includegraphics[scale=1]{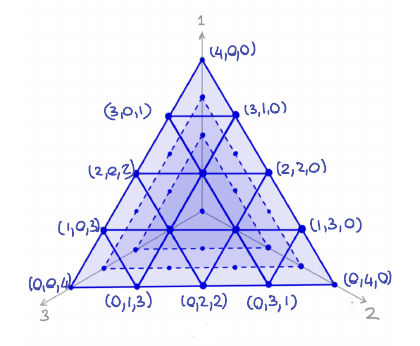}
	\caption{The super-imposed triangle graphs $\triangle_4,\triangle_3,\triangle_2$. }\label{fig.triangle-graph}
\end{figure}

This way, the projected vertices of $\triangle_{n-1}$ and $\triangle_{n-2}$ are in one-to-one correspondence with, respectively, the upward tiles of ${\triangle}_{n}$ and the downward tiles inside ${\triangle}_{n}$, since they correspond to their centers. In this way, the centers of the \textbf{upward} tiles are given by triples of coordinates $(a,b,c)$ such that $a+b+c=n-1$, while the ones of the \textbf{downward} tiles are given by triples of coordinates $(a,b,c)$ such that $a+b+c=n-2$.

Now we use the fact that each vertex of $\triangle_{123}$ has a flag. Then any vertex of coordinates $(a,b,c)$ in the triangle graph $\triangle_n$ is associated to the subspace $F^{(1)}_{n-a}\cap F^{(2)}_{n-b}\cap F^{(3)}_{n-c}$. This gives a line $\lambda_{abc}$ if $(a,b,c)$ is a center of an upward tile in $\triangle_n$ and a plane $\pi_{abc}$ if $(a,b,c)$ is the center of a downward tile. 

By construction, a plane $\pi_{abc}$ contains the lines $\lambda_{(a+1)bc}, \lambda_{a(b+1)c}, \lambda_{ab(c+1)}$ attached to its three vertices.

\begin{example}\label{ex:3x3flags3} 
	The triangle graph $\triangle_3$ and its lines and planes are shown in \Cref{fig.triangle-graph3}.
	
	\begin{figure}[!htb]
		\centering
		\includegraphics[scale=0.8]{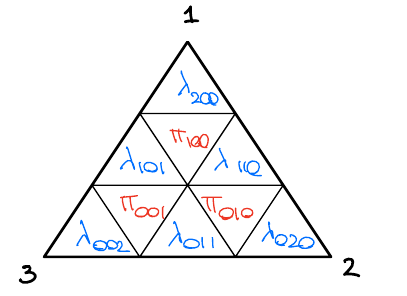}
		\caption{The triangle graph $\triangle_3$ and its lines and planes.}\label{fig.triangle-graph3}
	\end{figure}
	
	To find lines and planes, we use the formula
	$$
	F^{(1)}_{3-a}\cap F^{(2)}_{3-b}\cap F^{(3)}_{3-c}
	$$ 
	where the lines $\lambda_{abc}$ correspond to $a+b+c=2$ and the planes 
    $\pi_{abc}$ to $a+b+c=1$.
	
	For the flags of Example \ref{ex:3x3flags}, we obtain:   
	\begin{equation*}
		\begin{split}
			\lambda_{002}= \langle e_1+\alpha e_2+\beta e_3 \rangle,\quad 
			\lambda_{101}=\langle \gamma e_1+(\alpha\gamma-\beta) e_2 \rangle ,\quad 
			\lambda_{200}= \langle e_1\rangle,\quad \lambda_{110}=\langle e_2\rangle,\quad \lambda_{020}=\langle e_3\rangle,\\  
			\lambda_{011}= \langle e_2+\gamma e_3\rangle,\quad
			\pi_{100}=\langle  e_1, e_2  \rangle, \quad
			\pi_{010}=\langle e_3, e_2 \rangle, \quad
			\pi_{001}=\langle  e_1 +\alpha e_2+\beta e_3 , e_2+\gamma e_3 \rangle.
		\end{split}
	\end{equation*}

	Notice that the lines $\lambda_{abc}$ are exactly the ones that appear in the splittings of $\mathbb R^3$ given in Example \ref{ex:3x3flags1}, namely \eqref{eq:split3x3} can be rewritten as
	\begin{equation*}
		\begin{split}
			F^{(1)}_\bullet , F^{(2)}_\bullet:&\quad  
			\lambda_{200}\oplus \lambda_{110}\oplus  \lambda_{020},
			\\
			F^{(2)}_\bullet , F^{(3)}_\bullet:&\quad
			\lambda_{020}\oplus \lambda_{011}\oplus  \lambda_{002},
			\\
			F^{(3)}_\bullet , F^{(1)}_\bullet:&\quad  
			\lambda_{002}\oplus \lambda_{101}\oplus  \lambda_{200}.
		\end{split}
	\end{equation*}    
\end{example}

\subsection{Snakes and standard bases}\label{suse:proj-bases}

Motivated by the fact that the center of each upward tile is associated to a line, we add another layer to our combinatorial description: we consider triangles with vertices corresponding to the lines, see \Cref{fig:config}.  
\begin{figure}[!htb]
	\centering
	\includegraphics[width=\linewidth]{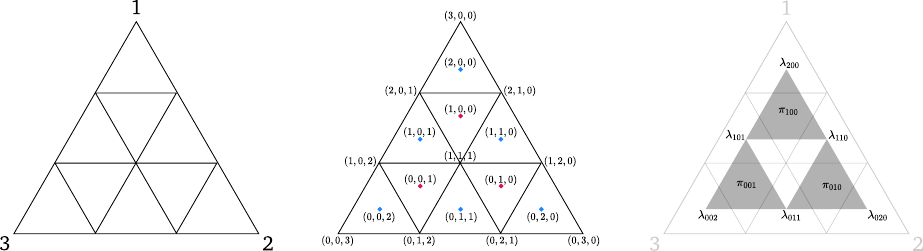}
	\caption{For $n=3$, from left to right: tessellation of $\triangle_{123}$, barycentric coordinates for vertices of the tessellation and centers of the tiles, configuration of subspaces with the gray triangles (and one white triangle surrounded by them).}\label{fig:config}
\end{figure}

For the rest of this section, we forget the tessellation focusing on these $\binom{n}{2}$ gray triangles---and the resulting $\binom{n-1}{2}$ white downward ones among them---looking at specific paths called snakes that run over their sides. Notice that the upward gray and downward white triangles give precisely the $n-1$ tessellation of a triangle connecting $\{\lambda_{(n-1)00},\lambda_{0(n-1)0},\lambda_{00(n-1)}\}$.
\begin{definition}
	A \emph{snake} $\mathbf{p}$ is an oriented piecewise path composed of exactly $n-1$ sides of gray triangles, which starts from a tile sharing a vertex with $\triangle_{123}$ and ends on a tile in contact with the opposite side.
\end{definition}

Notice that the length requirement implies no section of the snake can be parallel to the snake's target side of $\triangle_{123}$. 
We call $\mathbf{p}_{IJ}$, the unique snake running from $I$ to $J$ parallel to side $IJ$ of $\triangle_{123}$, a $\partial$-snake.

We now explain how to define a \textbf{standard projective basis} associated to each snake.

Denote by a Greek letter a generic triple of barycentric coordinates: e.g., $\lambda_{ijk}$ is equally denoted by $\lambda_\alpha$.
As shown in \Cref{fig:2projbasis}, each segment of a snake connects two vertices $\alpha, \beta$ of an upward triangle.

\begin{figure}[!htb]
	\centering
	\includegraphics[width=90mm]{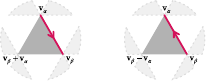}
	\caption{Segments of two oppositely oriented snakes. The
		vertices of the gray triangle correspond to 3 coplanar lines $\lambda_\alpha,\lambda_\beta,\lambda_\gamma$ and $\mathbf{v}_\gamma=\mathbf{v}_\beta\pm\mathbf{v}_\alpha$ depending on whether the segment
		is oriented clockwise or counterclockwise with respect to its gray triangle.}\label{fig:2projbasis}
\end{figure}

The corresponding lines  $\lambda_\alpha,\lambda_\beta$ are coplanar to $\lambda_\gamma$, where $\gamma$ is the remaining vertex of the upward triangle.
By coplanarity, a choice of vector $\mathbf{v}_\alpha\in \lambda_\alpha$ uniquely determines $\mathbf{v}_\beta\in\lambda_\beta$ by the following orientation rule
\begin{equation}\label{rule}
	\lambda_\gamma\ni\mathbf{v}_\gamma=
	\begin{cases}
		\mathbf{v}_\beta+\mathbf{v}_\alpha,\quad\circlearrowright\\
		\mathbf{v}_\beta-\mathbf{v}_\alpha,\quad\circlearrowleft.
	\end{cases}
\end{equation}
To understand why this rule is natural, let us pick three mutually transverse lines $\lambda_1,\lambda_2,\lambda_3$ in the plane as in \Cref{fig:st-plus} and \Cref{fig:st-minus}. Choose $w_1\in\lambda_1$ and denote its end point by $w_1$ with abuse of notation. Then there are two canonical choices for $w_2$: either the projection of $w_1$ onto $\lambda_2$ along the direction $\lambda_3$ as in \Cref{fig:st-minus} or along the direction orthogonal to $\lambda_3$ as in \Cref{fig:st-plus}.

\begin{figure}[!htb]
	\centering
	\includegraphics[width=90mm]{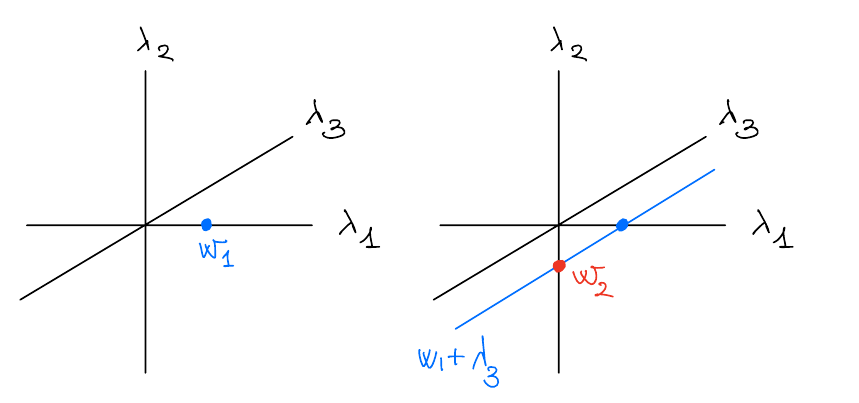}
	\caption{Take any $w_1\in\lambda_1$ and denote its end point by $\tcb{w_1}$ with abuse of notation, then consider the blue line that is parallel to $\lambda_3$ through  $\tcb{w_1}$. This blue line intersects $\lambda_2$ at a point $\tcr{w_2}$. Take $w_2\in\lambda_2$ to be the vector whose end point is $\tcr{w_2}$. In this case $w_1-w_2\in \lambda_3$.}\label{fig:st-minus}
\end{figure}

\begin{figure}[!htb]
	\centering
	\includegraphics[width=90mm]{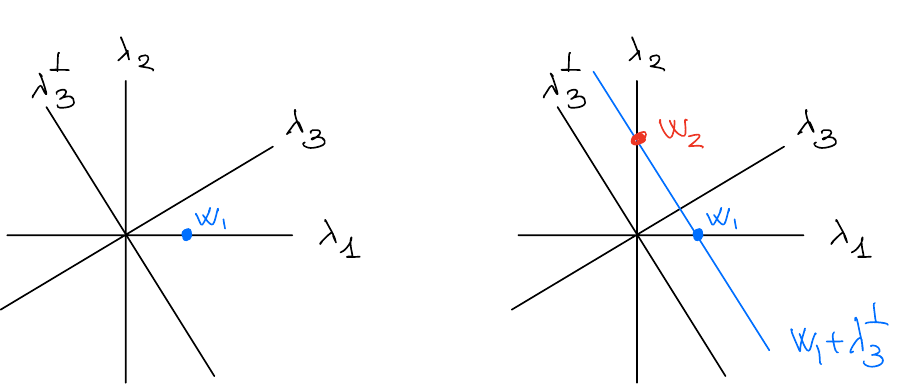}
	\caption{Take any $w_1\in\lambda_1$ and denote its end point by $\tcb{w_1}$ with abuse of notation, then consider the blue line that is orthogonal to $\lambda_3$ through  $\tcb{w_1}$. This blue line intersects $\lambda_2$ at a point $\tcr{w_2}$. Take $w_2\in\lambda_2$ to be the vector whose end point is $\tcr{w_2}$. In this case $w_1+w_2\in \lambda_3$.}\label{fig:st-plus}
\end{figure}

Therefore, a snake inductively determines a basis of $\real^n$ up to global rescaling: Once the first vector is chosen, iteratively applying the rule, the resulting $n$ vectors are defined uniquely. Rescaling the first vector gives a global scaling factor. Their linear independence is a consequence of the flags being assumed in general position. 

\begin{example}\label{ex:3x3flags4}
	For the lines of Example \ref{ex:3x3flags3}, let us apply the rule \eqref{rule} to pick a standard basis for each snake $\mathbf{p}_{31}$, $\mathbf{p}_{23}$ and $\mathbf{p}_{12}$.
	
	Denote by $\tilde u_1,\tilde u_2,\tilde u_3$ the standard basis associated to the snake $\mathbf{p}_{31}$. We choose an arbitrary $\tilde u_1\in \lambda_{002}$, for example 
	$$
	\tilde u_1=\eta(e_1+\alpha e_2+\beta e_3),\text{ for some }\eta\in \mathbb R^*,
	$$
	then by applying \eqref{rule}, we need to pick $\tilde u_2\in\lambda_{101}$ such that $\tilde u_1+\tilde u_2\in\lambda_{011}$. This means that we need to pick $\nu$ such that
	$$
	\nu(\gamma e_1+(\alpha\gamma-\beta) e_2)+ \eta(e_1+\alpha e_2+\beta e_3)\in \langle  e_2+\gamma e_3 \rangle,
	$$
	which gives (note that thanks to the general position of the flags, $\gamma\neq 0$ as seen in Example \ref{ex:3x3flags})
	$$
	\tilde u_2=-\eta\left(e_1+\frac{\alpha\gamma-\beta}{\gamma}e_2\right).
	$$
	Similarly we can show that $\tilde u_3=\eta e_1$. Hence, we see that the basis $\tilde u_1,\tilde u_2,\tilde u_3$ is uniquely defined up to global rescaling by $\eta$.
	
	In similar way, we construct the basis  $(\tilde v_1,\tilde v_2,\tilde v_3)$ corresponding to the snake $\mathbf{p}_{23}$ and  $(\tilde w_1,\tilde w_2,\tilde w_3)$ corresponding to $\mathbf{p}_{12}$:
	$$
	\tilde v_1= \sigma e_3,\quad \tilde v_2= -\sigma\left(\frac{1}{\gamma} e_2+e_3\right),
	\quad \tilde v_3=\sigma\left(\frac{1}{\beta} e_1+
	\frac{\alpha}{\beta}e_2+e_3\right),
	$$
	$$
	\tilde w_1=\mu e_1,\quad \tilde w_2=\mu \frac{\alpha\gamma-\beta}{\gamma}e_2,\quad \tilde w_3=\mu(\alpha\gamma-\beta) e_3.
	$$
	Let us compute the standard matrices $\tilde T_1,\tilde T_2,\tilde T_3$ such that $\tilde T_1\tilde w_i^T=\tilde u_i^T$, $\tilde T_2\tilde v_i^T=\tilde w_i^T$ and  $\tilde T_3\tilde u_i^T=\tilde v_i^T$. They all turn out to be of exactly the same form:
	\begin{equation}\label{eq:st-ma}
		\tilde T_1=\frac{\eta}{\mu}\mathbb T,\quad
		\tilde T_2= (\alpha\gamma-\beta)\frac{\mu}{\sigma}\mathbb T,\quad
		\tilde T_3= \frac{\sigma}{\eta\beta}\mathbb T,
	\end{equation}
	where
	\begin{equation}
		\mathbb T=
		\left(\begin{matrix}
			1&1+Z&Z\\-1&-1&0\\ 1&0&0
		\end{matrix}\right),\qquad\text{where} \quad
		Z:= \frac{\beta}{\alpha\gamma-\beta}.
	\end{equation}
	We conclude this example by observing that 
	\begin{equation}
		\label{eq:triple_ratio_greek_letters}
		\frac{\beta}{\alpha\gamma-\beta} =\operatorname{cr}_3(F^{(1)}_\bullet, F^{(2)}_\bullet, F^{(3)}_\bullet),
	\end{equation}
	where $\operatorname{cr}_3$ denotes the triple ratio associated with the three flags $F^{(1)}_\bullet, F^{(2)}_\bullet$ and $F^{(3)}_\bullet$. We explain  triple ratios in Subsection \ref{suse:triple-ratios}.
\end{example}

\subsection{Snake calculus}\label{suse:calc}

The fact that all the standard matrices $\tilde T_1,\tilde T_2, \tilde T_3$ in Example \ref{ex:3x3flags4} turn out to have the same form and depend only on the combination $ Z= \frac{\beta}{\alpha\gamma-\beta}$ is not due to the specific choice we made of the flags - indeed, as mentioned in Remark \ref{rmk:gen-flag}, any triple of flags in general position can be brought to this form. This is a feature of the combinatorial nature of the standard bases and is at the foundation of \textbf{snake calculus}, namely the combinatorial technique to compute the standard matrices that we are now going to explain. Here we follow closely \cite{DMM}.

In subsection \ref{suse:proj-bases}, we showed that each snake inductively determines a projective basis of $\mathbb R^n$: chosen the first
vector and iteratively applying the rule, the resulting n vectors are defined up to a
global scaling factor. Their linear independence is a consequence of the flags being
assumed generic. 

Given any two snakes, one can calculate change-of-basis matrices  between their
corresponding projective bases. The idea behind snake calculus is to factorize matrices in terms of the elementary moves \MakeUppercase{\romannumeral 1}, \MakeUppercase{\romannumeral 2}, \MakeUppercase{\romannumeral 3} in
\Cref{fig:snakemoves}.

\begin{figure}[!htb]
	\centering
	{\includegraphics[width=\textwidth]{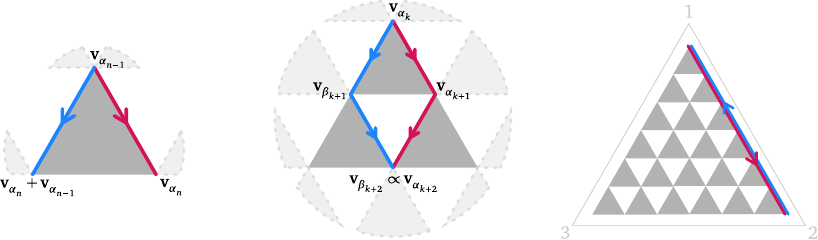}}
	\caption{From left to right, elementary snake moves \MakeUppercase{\romannumeral 1},\MakeUppercase{\romannumeral 2} and \MakeUppercase{\romannumeral 3} mapping {\color{2red}red} to {\color{2blue}blue} segments of a sample snake with $\mathbf{v}_1\in\lambda_{n00}$. Notice that move \MakeUppercase{\romannumeral 1} can only be performed on the last segment of a snake, i.e. when no subsequent segments can be affected. In this sense, move \MakeUppercase{\romannumeral 2} can be thought of as the extension of move \MakeUppercase{\romannumeral 1} to any other segment.}\label{fig:snakemoves}
\end{figure}

Move \MakeUppercase{\romannumeral 1} flips the last segment pivoting its source center across a gray triangle, by rule \eqref{rule} yielding:
\begin{equation}
	\begin{bmatrix}
		\mathbf{v}_{\alpha_1}\\\vdots\\\mathbf{v}_{\alpha_{n-1}}\\\mathbf{v}_{\alpha_n}
	\end{bmatrix}
	\mapsto
	\begin{bmatrix}
		\mathbf{v}_{\alpha_1}\\\vdots\\\mathbf{v}_{\alpha_{n-1}}\\\mathbf{v}_{\alpha_n}+\mathbf{v}_{\alpha_{n-1}}
	\end{bmatrix}
	=\underbrace{\left(
		\begin{array}{cc}
			\mathbb I_{n-2} &  \\
			& \begin{matrix}
				1 & 0 \\
				1 & 1
			\end{matrix}
		\end{array}\right)}_{L_{n-1}}
	\begin{bmatrix}
		\mathbf{v}_{\alpha_1}\\\vdots\\\mathbf{v}_{\alpha_{n-1}}\\\mathbf{v}_{\alpha_n}
	\end{bmatrix}.
\end{equation}
Move \MakeUppercase{\romannumeral 2} flips any two non parallel consecutive segments.
Analogously to move \MakeUppercase{\romannumeral 1}, sweeping the gray triangle yields  $\mathbf{v}_{\alpha_{k+1}}\mapsto\mathbf{v}_{\beta_{k+1}}=\mathbf{v}_{\alpha_{k+1}}+\mathbf{v}_{\alpha_{k}}$. However, this drags the second segment in a flip that pivots its target center: we expect the transformed $2$-segment  portion of the snake to end on a different vector within the same line, i.e. $\mathbf{v}_{\beta_{k+2}}\propto\mathbf{v}_{\alpha_{k+2}}$. Denoting by $Z$ the proportionality constant, the full move reads as
\begin{equation}\label{moveII}
	\begin{bmatrix}
		\mathbf{v}_{\alpha_1}\\\vdots\\
		\mathbf{v}_{\alpha_{k}} \\
		\mathbf{v}_{\alpha_{k+1}} \\ \mathbf{v}_{\alpha_{k+2}}\\
		\vdots\\
		\mathbf{v}_{\alpha_n}
	\end{bmatrix}
	\mapsto
	\begin{bmatrix}\mathbf{v}_{\alpha_1}\\\vdots\\ \mathbf{v}_{\alpha_{k}} \\ \mathbf{v}_{\beta_{k+1}} \\ \mathbf{v}_{\beta_{k+2}}\\
		\vdots\\
		\mathbf{v}_{\alpha_n}
	\end{bmatrix}
	=\left(\begin{array}{ccc}
		\mathbb I_{k-1} & & \\
		& \begin{matrix}
			1 & 0 & 0\\
			1 & 1 & 0\\
			0 & 0 & Z
		\end{matrix} &  \\ 
		&  & Z\ \mathbb I_{n-k-2}\end{array}\right)
	\begin{bmatrix}
		\mathbf{v}_{\alpha_1}\\\vdots\\
		\mathbf{v}_{\alpha_{k}} \\
		\mathbf{v}_{\alpha_{k+1}} \\ \mathbf{v}_{\alpha_{k+2}}\\
		\vdots\\
		\mathbf{v}_{\alpha_n}
	\end{bmatrix}.
\end{equation}

We give the following:
\begin{definition}\label{definition:blocks}
	Let $E_{rs}$ be the matrix unit, i.e., $(E_{rs})_{ij}=\delta_{ri}\delta_{sj}$.
	For $\mathbb  I$ denoting the identity matrix, $k\in\{1,\ldots,n-1\}$ and a parameter $t\in\real_{>0}$\,, 
	define the $SL_n(\real)$ matrices 
	\begin{align}
		L_k&=\mathbb  I+E_{k+1,k},\\   
		H_k(t)&=t^{-\frac{n-k}{n}}\mathrm{diag}\big(\underbrace{1,\ldots,1}_\text{$k$ times},t,\ldots,t\big),
	\end{align}
	and the $SL_n(\real)$ antidiagonal matrix
	\begin{equation}
		(S)_{ij}=(-1)^{n-i}\delta_{i,n+1-j}.
	\end{equation}
\end{definition}
The matrix appearing in \eqref{moveII} is a multiple of $L_k H_{k+1}$.

There are $\binom{n-1}{2}$ type \MakeUppercase{\romannumeral 2} moves, one for each downward white triangle, and the corresponding proportionality constants are the so-called \textbf{Fock--Goncharov variables}.
Topologically, notice that Fock--Goncharov variables are in bijection with inner vertices of the tessellation of $\triangle_{123}$: there is exactly one such vertex inside any white triangle. We thus denote them $Z_{ijk}$ by the barycentric coordinates of the unique corresponding vertex, $i,j,k\in\integer_{>0}$.

\begin{example}\label{ex:3x3flags5}
	For $n=3$, there is just a single Fock--Goncharov variable $Z_{111}$. Then the matrix $\tilde T_1$ mapping the snake $\mathbf p_{12}$ to the snake $\mathbf p_{31}$ can be factorized as in \Cref{fig:T1n2}: 
	\begin{equation}
		\begin{aligned}
			\tilde T_{1}&=S\ L_{2} L_{1} H_2(Z_{111})L_{2}\\
			&=\begin{pmatrix}
				0 & 0 & 1\\0 & -1 & 0\\1 & 0 & 0
			\end{pmatrix}\begin{pmatrix}
				1 & 0 & 0\\0 & 1 & 0\\0 & 1 & 1
			\end{pmatrix}\begin{pmatrix}
				1 & 0 & 0\\1 & 1 & 0\\0 & 0 & 1
			\end{pmatrix}\begin{pmatrix}
				Z_{111}^{\nicefrac{-1}{3}} & 0 & 0\\0 & Z_{111}^{\nicefrac{-1}{3}} & 0\\0 & 0 & Z_{111}^{\nicefrac{2}{3}}
			\end{pmatrix}\begin{pmatrix}
				1 & 0 & 0\\0 & 1 & 0\\0 & 1 & 1
			\end{pmatrix}=\\
			& = Z_{111}^{\nicefrac{-1}{3}}\begin{pmatrix}
				1&1+Z_{111}&Z_{111}\\
				-1&-1&0\\ 1&0&0\\
			\end{pmatrix}.
		\end{aligned}
	\end{equation}
	Note that this has the same form as we have seen in Example \ref{ex:3x3flags4}, therefore for the three flags of Example \ref{ex:3x3flags}, $Z_{111}=\frac{\beta}{\alpha\gamma-\beta}$.
	
	\begin{figure}[!htb]
		\centering
		\includegraphics[width=\linewidth]{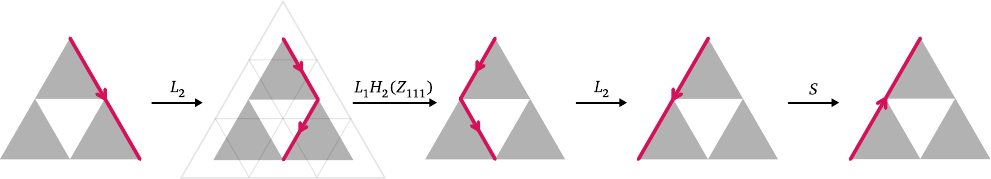}
		\caption{Sequence of snake moves factorizing $\tilde T_{1}$ for $n=3$. At step 2, the tessellation's only inner vertex of barycentric coordinates $(1,1,1)$ labels the Fock--Goncharov variable $Z_{111}$. At step 4, the $\partial$-snake runs counterclockwise and an $S$ matrix is needed.}\label{fig:T1n2}
	\end{figure}
\end{example}

\subsection{Fock--Goncharov coordinates as triple ratios}\label{suse:triple-ratios}
As mentioned after formula \eqref{moveII}, there is a Fock--Goncharov variable forevery white triangle. Each such white triangle is adjacent to three gray triangles whose vertices are three coplanar lines. Note that the white triangle together with its three adjacent triangles forms a triangle that is isomorphic to the inner triangle of $\triangle_3$ as in \Cref{fig:FG-variable}. 

\begin{figure}[!htb]
	\centering
	\includegraphics[width=\linewidth]{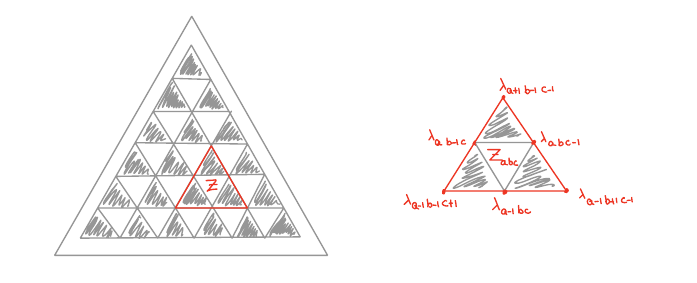}
	\caption{On the left, we highlight a white triangle and its three adjacent gray ones. On the right we display the isomorphic $\triangle_3$.}\label{fig:FG-variable}
\end{figure}

The six lines in \Cref{fig:FG-variable} all lie in the same three dimensional space. Indeed they lie in the intersection $F^{(1)}_{n-a+1}\cap F^{(2)}_{n-b+1}\cap F^{(3)}_{n-c+1}$, that due to the transversality condition \eqref{mindimkflags}, has dimension $3$. Consider any generators for these six lines, namely any vectors $v_{ijk}\in\mathbb R^3$ such that $\lambda_{ijk}=\langle v_{ijk} \rangle$. Since each triple of lines 
on the same side of the red triangle defines a splitting of $\mathbb R^3$, the corresponding three generators are linearly independent and span a parallelepiped. 

Given any three linearly independent vectors $v_\alpha,v_\beta,v_\gamma\in \mathbb R^3$, the oriented volume of the parallelepiped they span is the determinant of the matrix whose columns are $v_\alpha,v_\beta,v_\gamma$. We denote this volume as
$$
v_\alpha\wedge v_\beta\wedge v_\gamma:= \det(v_\alpha,v_\beta,v_\gamma).
$$
Hence, we associate to the white triangle the following \textbf{triple ratio:}
\begin{equation}
	\begin{split}
		\operatorname{cr}_3(a,b,c):=&\frac{v_{a+1,b-1,c-1} \wedge v_{a,b,c-1}  \wedge v_{a-1,b-1,c+1} }{v_{a+1,b-1,c-1} \wedge v_{a,b,c-1}  \wedge v_{a-1,b+1,c-1}}\\
		& \frac{ (v_{a-1,b-1,c+1}  \wedge v_{a,b-1,c}  \wedge v_{a-1,b+1,c-1} ) (v_{a-1,b+1,c-1}  \wedge v_{a-1,b,c}  \wedge v_{a+1,b-1,c-1})
		}{
			(v_{a-1,b+1,c-1}  \wedge v_{a-1,b,c}  \wedge v_{a-1,b-1,c+1} ) 
			(v_{a-1,b-1,c+1}  \wedge v_{a,b-1,c}  \wedge v_{a+1,b-1,c-1})
		}.
	\end{split}
\end{equation}

Note that the above expression only depends on the lines and not on the choice of the generating vectors. This quantity is an invariant of the flags $F^{(1)}_\bullet,F^{(2)}_\bullet,F^{(3)}_\bullet $. Indeed, it only depends on the three sub-flags 

\begin{equation*}
	\begin{split}
		F^{(1)}_{n-a-1}\subset F^{(1)}_{n-a}\subset F^{(1)}_{n-a+1},\\
		F^{(2)}_{n-b-1}\subset F^{(2)}_{n-b}\subset F^{(2)}_{n-b+1},\\
		F^{(3)}_{n-c-1}\subset F^{(3)}_{n-c}\subset F^{(3)}_{n-c+1}.\\
	\end{split}
\end{equation*}

\subsection{Pinnings}

As explained just after Example \ref{ex:3x3flags2}, each side $IJ$ of $\triangle_{123}$ comes with a splitting of $\mathbb R^n$ into $n$ lines and a \emph{pinning} corresponds to the choice of a line $\Lambda_{IJ}$ such that $\Lambda_{IJ}$ together with the lines associated to the splitting form a projective basis.

\begin{example}\label{ex:3x3flags6}
	Now that we calculated the standard bases on each side of the triangle, 
	let us see what are the pinnings. Let us focus on the snake $\mathbf{p}_{23}$. Corresponding to this snake, we have the splitting $\lambda_{020}\oplus \lambda_{011}\oplus  \lambda_{002}$. In the standard basis $\tilde v_1,\tilde v_2,\tilde v_3$, these lines are given by
	$$
	\lambda_{020}=\langle\tilde v_1\rangle,\quad
	\lambda_{011}=\langle\tilde v_2\rangle,\quad
	\lambda_{002}=\langle\tilde v_3\rangle.
	$$
	Now let us find the line $\Lambda_{23}$  such that the condition $v_1+v_2+v_3\in \Lambda_{23}$ specifies the  basis $v_1,v_2,v_3$ found in Example \ref{ex:3x3flags2}:
	$$
	v_1   =\sigma_1 e_3,\quad v_2=\sigma_2(e_2+\gamma e_3), \quad v_3= \sigma_3( e_1+\alpha e_2+\beta e_3).
	$$
	up to an overall factor (see Definition \ref{def:pr-b}). 
	
	Since
	$$\left(\begin{array}{cc}
		v_1^T     \\
		v_2^T     \\ v_3^T     \\ 
	\end{array}\right)=
	\left(\begin{array}{ccc}
		\frac{\sigma_1}{\sigma} & &  \\
		&- \frac{\sigma_2}{\sigma} \gamma &\\
		& &  \frac{\sigma_3}{\sigma} \beta
	\end{array}\right)
	\left(\begin{array}{cc}
		\tilde v_1^T     \\
		\tilde v_2^T     \\ \tilde v_3^T     \\ 
	\end{array}\right),
	$$
	the line $\Lambda_{23}$ is given by
	$$
	\Lambda_{23} =\langle \sigma_1 \tilde v_1-\sigma_2\gamma\tilde v_2+\sigma_3\beta \tilde v_3 \rangle.
	$$
	Similarly due to 
	$$
	\left(\begin{array}{cc}
		u_1^T     \\
		u_2^T     \\ u_3^T     \\ 
	\end{array}\right)=
	\left(\begin{array}{ccc}
		\frac{\eta_1}{\eta} & &  \\
		&- \frac{\eta_2}{\eta} \gamma &\\
		& &  \frac{\eta_3}{\eta} 
	\end{array}\right)
	\left(\begin{array}{cc}
		\tilde u_1^T     \\
		\tilde u_2^T     \\ \tilde u_3^T     \\ 
	\end{array}\right),\quad
	\left(\begin{array}{cc}
		w_1^T     \\
		w_2^T     \\ w_3^T     \\ 
	\end{array}\right)=
	\left(\begin{array}{ccc}
		\frac{\mu_1}{\mu} & &  \\
		& \frac{\mu_2}{\mu} \frac{\gamma}{\alpha\gamma-\beta} &\\
		& &  \frac{\mu_3}{\mu} \frac{1}{\alpha\gamma-\beta}
	\end{array}\right)
	\left(\begin{array}{cc}
		\tilde w_1^T     \\
		\tilde w_2^T     \\ \tilde w_3^T     \\ 
	\end{array}\right),
	$$
	we obtain 
	$$
	\Lambda_{31}= \langle \eta_1 \tilde u_1 -{\eta_2}{\gamma} \tilde u_2 + \eta_3\tilde u_3\rangle,\quad\text{and}\quad
	\Lambda_{12}=\langle\mu_1 w_1+\mu_2\frac{\gamma}{\alpha\gamma-\beta}\tilde w_2+\mu_3 \frac{1}{\alpha\gamma-\beta}\tilde w_3\rangle.
	$$
	This is compatible with the fact that the matrices $T_1,T_2,T_3$ in \eqref{eq:T1T2T3} are related to the matrices $\tilde T_1,\tilde T_2,\tilde T_3$ in \eqref{eq:st-ma} by mixed diagonal multiplication.
\end{example}

As illustrated in the above Example \ref{ex:3x3flags5}, each oriented side $IJ$ comes with two projective bases, one from the pinning and the other from the corresponding $\partial$-snake $\mathbf{p}_{IJ}$, and the  change-of-basis matrix between them is given by a diagonal matrix that depends on $n-1$ proportionality constants.
These $n-1$ proportionality constants are thought of as additional Fock--Goncharov variables $Z_{ijk}$, labeled by the vertices on the interior of $IJ$.\\
Adding these extra variables from all three sides to the ones produced by type \MakeUppercase{\romannumeral 2} moves, we get a total of
$3(n-1)+\binom{n-1}{2}=\frac{(n+4)(n-1)}{2}$ Fock--Goncharov variables (\Cref{fig:FG}).

\begin{figure}[!htb]
	\centering
	\includegraphics[width=.9\textwidth]{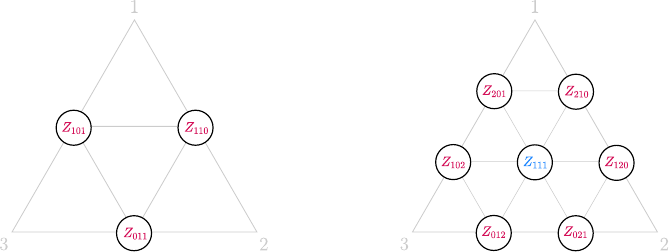}
	\caption{Fock--Goncharov variables $Z_\alpha$ for $\mathcal{P}_{\PGL_2(\real)}(\triangle_{123})$ on the left and $\mathcal{P}_{\PGL_3(\real)}(\triangle_{123})$ on the right. {\color{2blue}Blue} variables are associated with moves \MakeUppercase{\romannumeral 2} and {\color{2red}red} ones with side pinnings.}\label{fig:FG}
\end{figure}

As a whole, they are in bijection with the tessellation's vertices except $1,2,3$.

\begin{definition}\label{def-tt}
	The transport matrices $T_1,T_2,T_3$ are the following $\PGL_n(\real)$ matrices
	\begin{equation}\label{T-formulae}
		\begin{aligned}
			& T_1
			=S\prod_{k=1}^{n-1}\Big[H_{n-k}(Z_{k,0,n-k})\Big]\ \ L_{n-1}\prod_{j=1}^{n-2}\left[\prod_{i=1}^j\Big[L_{n-i-1}H_{n-i}(Z_{n-j-i,i,j})\Big]L_{n-1}\right]\ \ \prod_{k=1}^{n-1}H_k(Z_{n-k,k,0}),\\
			& T_2
			=S\prod_{k=1}^{n-1}\Big[H_{n-k}(Z_{n-k,k,0})\Big]\ \ L_{n-1}\prod_{j=1}^{n-2}\left[\prod_{i=1}^j\Big[L_{n-i-1}H_{n-i}(Z_{j,n-j-i,i})\Big]L_{n-1}\right]\ \ \prod_{k=1}^{n-1}H_k(Z_{0,n-k,k}),\\
			& T_3=S\prod_{k=1}^{n-1}\Big[H_{n-k}(Z_{0,n-k,k})\Big]\ \ L_{n-1}\prod_{j=1}^{n-2}\left[\prod_{i=1}^j\Big[L_{n-i-1}H_{n-i}(Z_{i,j,n-j-i})\Big]L_{n-1}\right]\ \ \prod_{k=1}^{n-1}H_k(Z_{k,0,n-k}).
		\end{aligned}
	\end{equation}
\end{definition}
As expected, a direct computation confirms that $T_1T_2T_3=\mathbb  I$.
Together with their inverses, $T_1, T_2$ and $T_3$ suffice to map between any two sides.
Notice that, for the permutation map $\sigma$ acting on matrices $T(Z_{ijk})$ depending on Fock--Goncharov variables $Z_{ijk}$ as
\begin{equation}
	\sigma T(Z_{ijk}):=T(Z_{jki}),
\end{equation}
we have $T_2=\sigma T_1$ and $T_3=\sigma^2 T_1$.

Observe the following useful relation:
\begin{equation}\label{eq:HS-SH}
	H_k(Z) S = S H_{n-k}(Z^{-1})
\end{equation}

\begin{example}
	Applying formulae \eqref{T-formulae} to the case $n=3$, we find
	$$
	T_1= S H_2(Z_{102}) H_1(Z_{201}) L_2 L_1 H_2(Z_{111}) L_2 H_1(Z_{210}) H_2(Z_{120}),
	$$
	and using \eqref{eq:HS-SH}, we have
	\begin{equation}
		\label{eq:T1forn3}
		\begin{split}
			T_1=&  H_1(Z_{102}^{-1}) H_2(Z_{201}^{-1})S L_2 L_1 H_2(Z_{111}) L_2 H_1(Z_{210}) H_2(Z_{120})=\\
			&H_1(Z_{102}^{-1}) H_2(Z_{201}^{-1}) \tilde T_1 H_1(Z_{210}) H_2(Z_{120}),
		\end{split}
	\end{equation}
	where $\tilde T_1$ was calculated in Example \ref{ex:3x3flags5}. Notice that the two factors on the left and on the right  of $\tilde T_1$ in \eqref{eq:T1forn3} are both diagonal as expected from our computations comparing $T_1,T_2,T_3$ in \eqref{eq:T1T2T3} to the matrices $\tilde T_1,\tilde T_2,\tilde T_3$ in \eqref{eq:st-ma}.
\end{example}

\subsection{Gluing triangles}\label{se:glue}

As explained at the beginning of this Section, we triangulate the finite part of the surface and associate to any path a product of transport matrices, one for each triangle the path crosses. However, the coordinate description of these transport matrices was given assuming that all triangles are oriented clockwise. This means that every time we glue two adjacent triangles, the shared edge inherits two opposite orientations, one from each triangle. Since the coordinate descriptions on either side of the edge are related by a reversal of vertex ordering, we must insert a matrix $S$ to pass between the two coordinate systems. As a result, the matrix associated to any path takes the form of an alternating product
\begin{equation}\label{eq:fact-FG}
	S^{t} T^{(k)}_{i_k} S T^{(k-1)}_{i_{k-1}} S\dots T^{(2)}_{i_2} S T_{i_1}^{(1)} S^{s}, \qquad t,s\in\{0,1\},
\end{equation}
where we have enumerated by $1,\dots,k$ the triangles in the order they are crossed by the given path, so that $T^{(j)}_{i_j}$ denotes the $i_j$-th transport matrix associated to the $j$-th triangle, and the initial and final factors of $S$ are determined by the choice of orientation we pick for the initial and final snakes. The factorization may or may not start and end with an $S$ matrix according to the orientation of the initial and final edges.
This structure is illustrated concretely in Example \ref{ex:2tri} below.

\begin{example}\label{ex:2tri}
	Consider the situation of two adjacent triangles as on the left hand side of \Cref{fig:glue2t}. To calculate the matrix $M$ corresponding to the highlighted path, we need to separate the triangles and label clockwise the vertices of each triangle. 
	\begin{figure}[!htb]
		\centering
		\includegraphics[width=.9\textwidth]{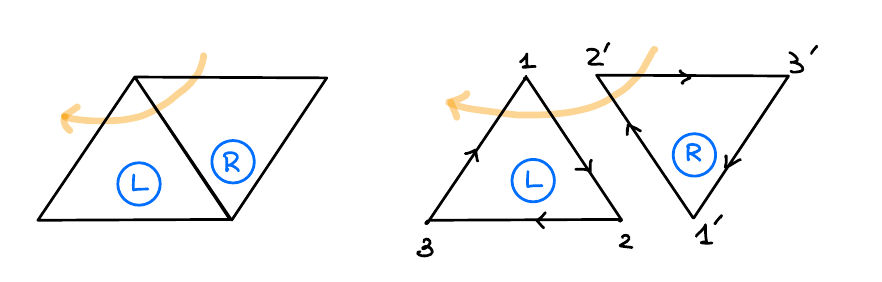}
		\caption{On the left two adjacent triangles, on the right,  the separated  triangles with labeled vertices.}\label{fig:glue2t}
	\end{figure}
	Then the matrix we are looking for will be given by
	$$
	M= T_1^{(L)}ST_2^{(R)},
	$$
	where $T_i^{(L)}$ denote the transport matrices associated to the left triangle and $T_i^{(R)}$ denote the ones associated to the triangle on the right. Using formulae \eqref{T-formulae} for each triangle, we find
	\begin{equation*}
		\begin{split}
			M=&S\prod_{k=1}^{n-1}\Big[H_{n-k}(Z_{k,0,n-k}^{(L)})\Big]\ \ L_{n-1}\prod_{j=1}^{n-2}\left[\prod_{i=1}^j\Big[L_{n-i-1}H_{n-i}(Z_{n-j-i,i,j}^{(L)})\Big]L_{n-1}\right] \\
            &
          \prod_{k=1}^{n-1}H_k(Z_{n-k,k,0}^{(L)})  \,S\,  S\prod_{k=1}^{n-1}\Big[H_{n-k}(Z_{n-k,k,0}^{(R)})\Big]\ \ L_{n-1}\\
          &\prod_{j=1}^{n-2}\left[\prod_{i=1}^j\Big[L_{n-i-1}H_{n-i}(Z_{j,n-j-i,i}^{(R)})\Big]L_{n-1}\right]\ \ \prod_{k=1}^{n-1}H_k(Z_{0,n-k,k}^{(R)})
		\end{split}
	\end{equation*}
	where we have added the labels ${}^{(L)}$ to the Fock--Goncharov variables in the left triangle and ${}^{(R)}$ to the ones in the right triangle. Using the fact that $S^2=(-1)^{n-1}\mathbb I$, and the diagonality of the matrices $H_k$, we obtain
	\begin{equation*}
		\begin{split}
			M=& S\prod_{k=1}^{n-1}\Big[H_{n-k}(Z_{k,0,n-k}^{(L)})\Big]\ \ L_{n-1}\prod_{j=1}^{n-2}\left[\prod_{i=1}^j\Big[L_{n-i-1}H_{n-i}(Z_{n-j-i,i,j}^{(L)})\Big]L_{n-1}\right] \\
			&
		(-1)^{n-1}	\prod_{k=1}^{n-1}H_k(Z_{n-k,k,0}^{(L)}Z_{k,n-k,0}^{(R)}) \ \ L_{n-1}\\
            &
            \prod_{j=1}^{n-2}\left[\prod_{i=1}^j\Big[L_{n-i-1}H_{n-i}(Z_{j,n-j-i,i}^{(R)})\Big]L_{n-1}\right]\ \ \prod_{k=1}^{n-1}H_k(Z_{0,n-k,k}^{(R)})
		\end{split}
	\end{equation*}
	Therefore the matrix $M$ does not depend on the Fock--Goncharov variables $Z_{n-k,k,0}^{(L)}$ and $Z_{k,n-k,0}^{(R)}$ separately, but only on their product. 
\end{example}

As seen in the previous Example \ref{ex:2tri}, in each successive product of matrices $ T_{i_{j-1}}^{(j-1)}ST_{i_j}^{(j)}$ in the factorization \eqref{eq:fact-FG}, we have that the Fock--Goncharov variables corresponding to the edges that are glued don't appear separately, but only as products. Therefore, when gluing triangles, we replace the Fock--Goncharov variables corresponding to the edges that are glued by their products. This procedure is called \textbf{amalgamation}.

\subsubsection{Amalgamated variables as cross ratios}

The internal Fock--Goncharov coordinates on a single triangle are defined in terms of triple ratios associated to the triple of flags at the vertices. However, there is a subtle point that deserves attention: the flags are associated to vertices, and to compare them we should transport them to the same point. Indeed, in the Fock--Goncharov setting, one works with a flat principal $\mathbb PSL_n(\mathbb R)$-bundle on the surface, so that the flags live in different fibers. Therefore, to compute their triple ratio - which is an invariant of a triple of flags in a single vector space - one must first transport them to a common fiber via parallel transport. The result of parallel transport depends on the homotopy class of the path. In the case of a single triangle, there is no ambiguity because all paths between two given points are homotopic and the triple ratio is well-defined without any additional choices.
However, when gluing triangles, non simply connected surfaces may arise and therefore we no longer have a canonical way to compare flags. This problem can be addressed by adding extra triangles to take into account the monodromy due to transporting along paths that are not homotopic to each other. To illustrate this idea, we go back to the example of the cylinder with two bordered cusps on one boundary.

\begin{example}
    The fat graph of the cylinder with two bordered cusps on one boundary and its dual triangulation are depicted in \Cref{fig:fat-tri-12}:
    	\begin{figure}[!htb]
		\centering
		\includegraphics[width=.4\textwidth]{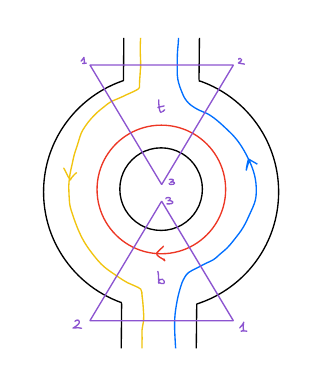}
		\caption{The fat graph of the cylinder with two bordered cusps on one boundary and its dual triangulation. The generating paths of the fundamental groupoid are displayed in red, blue and yellow.}\label{fig:fat-tri-12}
	\end{figure}
    
     When comparing the flags of the top triangle to the ones of the bottom triangle, we need to transport them to a common base point. For example, we could transport the flags of the top triangle to a base point in the bottom triangle along the blue path or along the yellow path. The resulting flags would be different. Therefore, since we have two different ways to transport the flags of the top triangle to the bottom one, we replace the top triangle with two triangles as in the left hand side of \Cref{fig:extra-triangle}. We then transport the flags of one of them along the yellow path and the ones of the other along the blue path as in the right hand side of  \Cref{fig:extra-triangle}. 
     
\begin{figure}[!htb]
		\centering
		\includegraphics[width=0.9\textwidth]{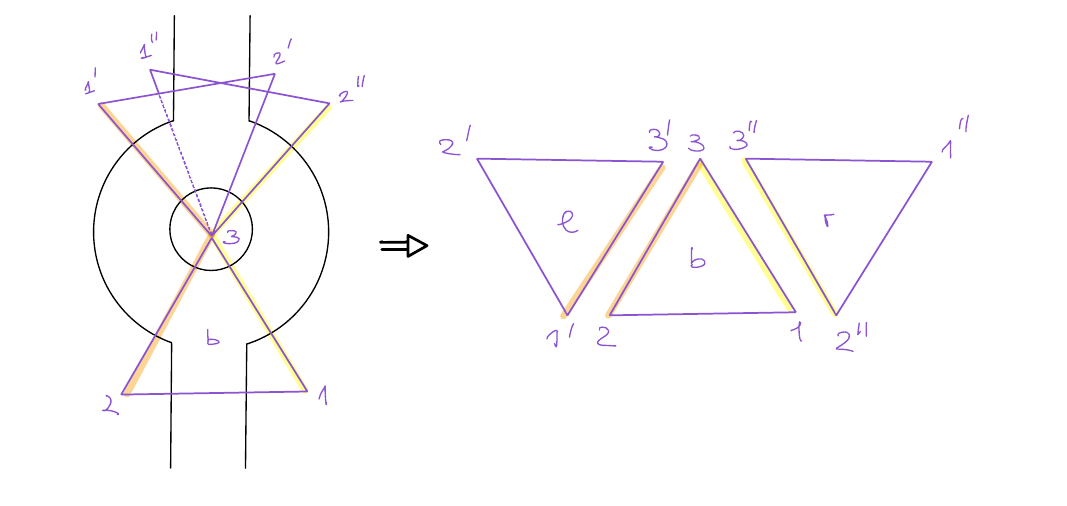}
		\caption{On the left, the fat graph of the cylinder with two bordered cusps on one boundary and its dual triangulation with one extra triangle. On the right the resulting gluings of the three triangles.}\label{fig:extra-triangle}
	\end{figure}

Therefore, in order to describe the geometric meaning of the amalgamated variables in terms of invariants of the flags, we need one extra triangle, in such a way that the flag associated to the point $1''$ is not the same as the flag associated to the point $2'$. 
\end{example}

As seen in the above example, if we attach copies of triangles, we can always assume that any two triangles are glued along only one edge. This allows us to explain the geometric meaning of the amalgamated variables. We do this in the case of $n=3$.

Let us consider two triangles, left and right, as in the left hand side of \Cref{fig:amalg-cross}. We denote by an index $L$ the lines and the Fock--Goncharov variables in the left triangle and by an index $R$ the ones in the right one. 
For the left triangle we pick the same flags as in Example \ref{ex:3x3flags}. For the right triangle, because we identify the vertices $1'=2$ and $2'=1$, we pick $F^{(1')}=F^{(2)}$, $F^{(2')}=F^{(1)}$ and $F^{(3')}$ in general position, namely
	$$
	 \langle e_1 +\delta e_2+\epsilon e_3 \rangle=F^{(3')}_1\subset  \langle  e_1 +\delta e_2+\epsilon e_3 , e_2+\eta e_3  \rangle=F^{(3')}_2 \subset  \mathbb R^3=F^{(3')}_3.
	$$
 In \Cref{ex:3x3flags4}, we already calculated the standard projective basis corresponding to the edge $12$ in the left triangle (denoted here as $w_i^{(L)}$):
$$
 w_1^{(L)}=\mu e_1,\quad  w_2^{(L)}=\mu \frac{\alpha\gamma-\beta}{\gamma}e_2,\quad  w_3^{(L)}=\mu(\alpha\gamma-\beta) e_3.
	$$
    Following the same method, we can calculate the standard projective basis corresponding to the edge $2'1'$ in the right triangle:
    $$
    w_1^{(R)}=\rho e_1,\quad w_2^{(R)} = \rho\frac{\delta\eta-\epsilon}{\eta} e_2,\quad w_3^{(R)}= \rho (\delta\eta-\epsilon) e_3.
    $$
These bases are related by a diagonal matrix
\begin{equation}\label{eq:basis-relation}
 \left(\begin{array}{c}
     w_1^{(L)}  \\
        w_2^{(L)} \\
       w_3^{(L)}  \\
 \end{array}
  \right)=   \frac{\mu}{\rho}
 \left(\begin{array}{ccc}
 1 & &\\
 &\frac{\alpha\gamma-\beta}{\gamma} \frac{\eta}{\delta \eta-\epsilon}& \\
& &\frac{\alpha\gamma-\beta}{\delta \eta-\epsilon} \\
\end{array}  \right)
  \left(\begin{array}{c}
     w_1^{(R)}  \\
        w_2^{(R)} \\
       w_3^{(R)}  \\
 \end{array}
  \right).
\end{equation}
On the other hand, as seen in \Cref{ex:2tri}, the matrix comparing the two standard projective bases should coincide (up to a multiplicative factor) with   
  \begin{equation}\label{eq:basis-relation1}
     \begin{split}
 H_1(Z_{210}^{(L)}Z_{120}^{(R)})
    H_2(Z_{120}^{(L)}Z_{210}^{(R)})&=
    \opn{diag}\left(\left(Z_{210}^{(L)}Z_{120}^{(R)}\right)^{-\frac{2}{3}},\left(Z_{210}^{(L)}Z_{120}^{(R)}\right)^\frac{1}{3},\left(Z_{210}^{(L)}Z_{120}^{(R)}\right)^\frac{1}{3}\right)\\
   & \opn{diag}\left(\left(Z_{120}^{(L)}Z_{210}^{(R)}\right)^{-\frac{1}{3}},\left(Z_{120}^{(L)}Z_{210}^{(R)}\right)^{\frac{2}{3}},\left(Z_{120}^{(L)}Z_{210}^{(R)}\right)^{\frac{2}{3}} \right).
 \end{split} 
\end{equation}
Therefore, normalizing the diagonal matrix in \eqref{eq:basis-relation} in such a way that it has determinant equal to $1$, and comparing with the diagonal matrix in \eqref{eq:basis-relation1}, we obtain:
$$
Z_{210}^{(R)}Z_{120}^{(L)}=\frac{\alpha\gamma-\beta}{\gamma} \frac{\eta}{\delta \eta-\epsilon},\qquad
Z_{120}^{(R)} Z_{210}^{(L)}=\frac{\gamma}{\eta}.
$$
The right hand sides of these expressions are cross ratios of the lines highlighted in  \Cref{fig:amalg-cross}:

\begin{figure}[!htb]
		\centering
\includegraphics[width=1\textwidth]{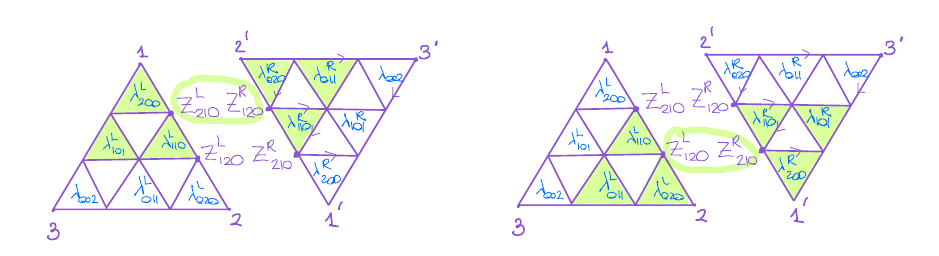}
		\caption{On the left, the amalgamation process of $Z_{210}^{(L)}$ and $Z_{120}^{(R)}$ and the corresponding lines entering the cross ratio are highlighted. On the right, the same for  $Z_{120}^{(L)}$ and $Z_{210}^{(R)}$.}\label{fig:amalg-cross}
	\end{figure}
    
Indeed, we have:
$$
\opn{cr}(\lambda_{200}^{(L)},\lambda_{110}^{(L)},\lambda^{(L)}_{101},\lambda^{(R)}_{011})= \frac{\alpha\gamma-\beta}{\gamma} \frac{\eta}{\delta \eta-\epsilon}
$$
and
$$
\opn{cr}(\lambda_{110}^{(L)},\lambda_{020}^{(L)},\lambda^{(L)}_{011},\lambda^{(R)}_{101})= \frac{\eta}{\gamma}
$$
Observe that since the edge $12$ in the left triangle coincides with the edge $2'1'$ in the right one, some lines are the same:
$$
\lambda_{200}^{(L)}=\lambda_{020}^{(R)},\quad 
\lambda_{110}^{(L)}=\lambda_{110}^{(R)},\quad 
\lambda_{020}^{(L)}=\lambda_{200}^{(R)}. 
$$

We stress that if we were to repeat the computation of the amalgamated variables for the cylinder with two bordered cusps on one boundary using only two triangles (without adding an extra one) and identifying two edges, all amalgamated variables would equal one, since the flags from one triangle would coincide with those of the other. Adding a third triangle is therefore needed to obtain non-trivial amalgamated variables from this construction. 

However, from a practical point of view, if the amalgamated variables are taken as free parameters from the outset, two triangles already suffice to parametrize the matrices. We explain this in detail in the next Example.

\begin{example}\label{ex:non-amalgamated}
	In the case of a cylinder with two bordered cusps on the boundary, looking at \Cref{fig:fat-tri-12}, 
the matrices $Y,B,R$ corresponding to the yellow, blue and red paths respectively  are
	\begin{equation}\label{CYRfactorization}
		\begin{aligned}
			Y&=  \ S \ T_2^{(b)} \ S \ T_1^{(t)},\\
			B&= \ T_2^{(t)} \ S \ T_1^{(b)} \ S,\\
			R&=\ T_1^{(t)^{-1}} \ S  \  T_3^{(b)} \ S \ T_2^{(t)^{-1}}.
		\end{aligned}
	\end{equation}
	For example, in the case $n=3$, these matrices only depend on $Z_{210}^{(t)}$, $Z_{120}^{(t)}$, $Z_{210}^{(b)}$, $Z_{120}^{(b)}$ and the following amalgamated variables:
	\begin{equation}
		Z_{Y1}=Z^{(t)}_{201}Z^{(b)}_{021}, \quad Z_{Y2}=Z^{(t)}_{102}Z^{(b)}_{012}, \quad 
		Z_{C1}=Z^{(b)}_{201}Z^{(t)}_{021},\quad 
		Z_{C2}=Z^{(b)}_{102}Z^{(t)}_{012},
	\end{equation}
    which we assume to be arbitrary.
    Therefore, in the case when we consider only two triangles, we have a total of $2\left(\begin{array}{c}
        n-1  \\
         2 
    \end{array}\right)$ internal variables, $2(n-1)$ amalgamated variables and $2(n-1)$ pinning variables.

Let us now consider three triangles instead, as in \Cref{fig:tre-triangoli}.

\begin{figure}[!htb]
		\centering
\includegraphics[width=.7\textwidth]{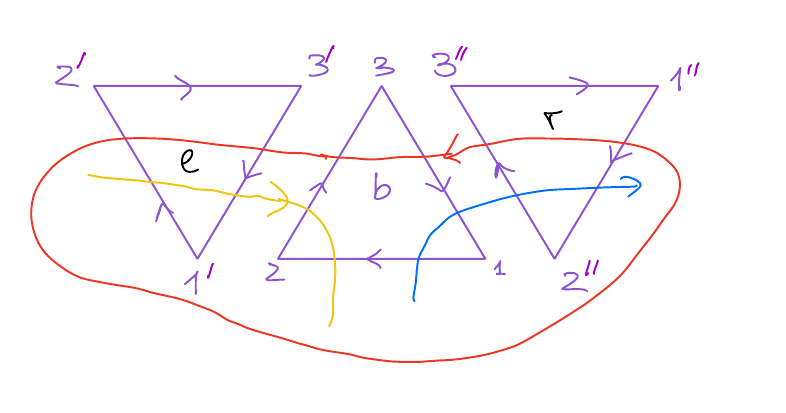}
		\caption{The three triangles as in \Cref{fig:extra-triangle} and the corresponding red, yellow and blue paths.}\label{fig:tre-triangoli}
	\end{figure}
    
Now, the matrices $Y,B,R$ corresponding to the yellow, blue and red paths respectively are given by 
\begin{equation}\label{CYRfactorization1}
		\begin{aligned}
			Y&=  \ S \ T_2^{(b)} \ S \ T_1^{(l)},\\
			B&= \ T_2^{(r)} \ S \ T_1^{(b)} \ S,\\
			R&=\ T_1^{(l)^{-1}} \ S  \  T_3^{(b)} \ S \ T_2^{(r)^{-1}}.
		\end{aligned}
	\end{equation}
    Note that the Fock--Goncharov variables corresponding to the edges $2'3'$ and $3'' 1''$ don't enter the matrices \eqref{CYRfactorization1}. Moreover, we have to identify the left and right triangles, therefore we set 
\begin{equation}\label{eq:identlr}
      Z_{ijk}^{(l)}= Z_{ijk}^{(r)},\quad\forall i,j,k.
\end{equation}
This implies that we have a total of $2\left(\begin{array}{c}
        n-1  \\
         2 
    \end{array}\right)$ internal variables (because of the identification \eqref{eq:identlr}),  $2(n-1)$ amalgamated variables and $2(n-1)$ pinning variables, because the ones corresponding to the edges $2'3'$ and $3" 1"$ don't enter the matrices \eqref{CYRfactorization1}. Therefore the number of variables involved in the two descriptions is the same. Finally, thanks to the identification \eqref{eq:identlr}, we have that $T_i^{(l)}=T_i^{(r)}$ for $i=1,2,3$ and therefore \eqref{CYRfactorization} and \eqref{CYRfactorization1} give rise to the same formulae.    
\end{example}

\begin{remark}
    The fact that three triangles appear in the description of the paths on the surface $\Sigma$ of figure \ref{fig:fat-tri-12} has a natural interpretation in terms of universal covers.
    Indeed, as explained at the beginning of Section \ref{se:RS}, for any $x_0 \in \Sigma$ one can consider the universal cover $\widetilde{\Sigma}$ of $\Sigma$ based at $x_0$, together with its projection map $p: \widetilde{\Sigma} \to \Sigma$.
    By identifying homotopy classes of paths with their target, one obtains that $\widetilde{\Sigma}$ is again a Riemann surface.
    Thus, one may triangulate $\widetilde{\Sigma}$ in the same way as $\Sigma$.
    In fact, one can say more: rather than choosing an arbitrary triangulation, one can take a given triangulation of $\Sigma$ and pull it back to $\widetilde{\Sigma}$ via $p$.
    This gives a triangulation of $\widetilde{\Sigma}$ made by infinitely many triangles.
    However, the action of the deck transformation group $\operatorname{Deck}(\Sigma)$ identifies many of these triangles, so that only two (labeled t for top and b for bottom in \Cref{fig:fat-tri-12}) remain as representatives in the quotient.
    Since the action of $\operatorname{Deck}(\Sigma)$ encodes exactly the kernel of the covering map $p: \widetilde{\Sigma} \to \Sigma$, one has $\Sigma \cong \widetilde{\Sigma}\slash_{\textrm{Deck}(\Sigma)}$.
    Consequently,  the two triangles in the quotient correspond precisely to the triangles on $\Sigma$.
    Therefore, the initial use of three triangles reflects what happens on $\widetilde{\Sigma}$: when computing the variables, one has to consider actual flags and transport them along paths using the corresponding holonomy.
    However, the particular layer chosen in $\widetilde{\Sigma}$ is not essential for such computations, since these values repeat under the symmetries induced by $\operatorname{Deck}(\Sigma_{g,s})$.
    Hence, the projection map assigns variables to only two triangles.
\end{remark}

\subsection{Relation between Fock--Goncharov variables and shear coordinates in the case $n=2$}\label{suse:FGr}

In the shear coordinate description of the Teichm\"uller space and of the bordered cusped Teichm\"uller space, the finite part of a surface $\Sigma_{g,s,m}$ is triangulated by ideal triangles, and any path in $\Sigma_{g,s,m}^f$ is mapped to a matrix factorized in terms of
edge, left and right matrices defined in \eqref{eq:generators}. These can be expressed in terms of the following specialization to $n=2$ of the matrices in \Cref{definition:blocks}:
\begin{equation}\label{eq:FG-matrices}
	S:= \left(\begin{array}{cc}
		0    & -1 \\
		1  & 0
	\end{array}\right),\quad 
	L_1:= \left(\begin{array}{cc}
		1    & 0 \\
		1  & 1
	\end{array}\right),\quad
	H(z):= \left(\begin{array}{cc}
		\frac{1}{\sqrt{z}}    & 0 \\
		0  & \sqrt{z}
	\end{array}\right).
\end{equation}
Indeed, 
$$
X(s)=S H(e^{s}),\quad 
L=S L_1 S L_1,\quad 
R=- S L_1.
$$
This shows that any factorization of a matrix representing a path obtained by the shear coordinate description can be recast into a factorization in terms of the Fock--Goncharov description. 

Observe that in the case $n=2$ there are no internal Fock--Goncharov coordinates, so that all Fock--Goncharov coordinates correspond to the ones produced by the pinnings. We saw that on edges that belong to two triangles, these amalgamate into variables that are expressed by cross ratios. 

The exponentiated shear coordinates can also be expressed in terms of cross ratios of quadruples of lines associated to the vertices of the ideal triangulation and their images. 
Let us see how this works in the pair of pants example.

\begin{example}
	In the case of the pair of pants, in Example \ref{ex:no-fun-dom}, we saw that in order to reconstruct the Fuchsian group uniquely, we need to take the ideal triangulation and add two copies of one of the triangles. Then the 
	exponentiated shear coordinates  are expressed in terms of the cross ratios of the vertices of these four triangles as follows:
	\begin{equation}
		\begin{split}
			& \textrm{cr}(q_2^{(1)},q_2^{(3)},\gamma_1(q_2^{(2)}),q_2^{(2)})=-e^{-s_1},\\
			& \textrm{cr}(q_2^{(3)},q_2^{(2)},\gamma_3(q_2^{(1)}),q_2^{(1)})=-e^{-s_2},\\
			& \textrm{cr}(q_2^{(1)},q_2^{(2)},q_2^{(3)},\gamma_2(q_2^{(3)}))
			=-e^{-s_3}. 
		\end{split}
	\end{equation}
\end{example}

\subsection{Chewing-gum moves as inverse amalgamation}

In the case of a surface with no bordered cusps, all edges of triangles are amalgamated. In this case, we are dealing with the Fock--Goncharov moduli space of $\mathbb PSL_n(\mathbb R)$, which coincides with the higher Teichm\"uller space. 

\begin{example}
 Consider the sphere with 4 boundaries for which the fat graph is drawn in \Cref{fig:4-holed-sphere}. The fundamental group $\pi_1(\Sigma_{0,4})$ is generated by the four loops in the same figure and in the $n\times n$ case the higher Teichm\"uller space
 $$
\mathcal T_{0,4}=\opn{Hom}'(\pi_1(\Sigma_{0,4}),\mathbb PSL_n(\mathbb R))/\mathbb PSL_n(\mathbb R),
$$
has dimension $2(n^2-1)$. The Procesi coordinates are given by all traces of words in the generators, up to degree $n$, modulo the relations from the characteristic polynomial and the constraint that the product is $1$.
In the Fock--Goncharov setting we have $4$ triangles with $(n-1)(n-2)/2$ internal variables each and six amalgamated edges, therefore we have a total of 
$$
2(n-1)(n-2)+6(n-1)=2(n^2-1),
$$
as expected. Notice that in the Fock--Goncharov description, we have changed a little bit the starting and ending point of the generating loops in such a way to have them between triangles. This is irrelevant thanks to the overall conjugation by $\mathbb PSL_n(\mathbb R)$. Moreover, even if we choose a starting and an ending point, we have freedom of cyclic permutations inside the traces, which means that we can always amalgamate all variables. 

\begin{figure}[!htb]
		\centering
\includegraphics[width=.6\textwidth]{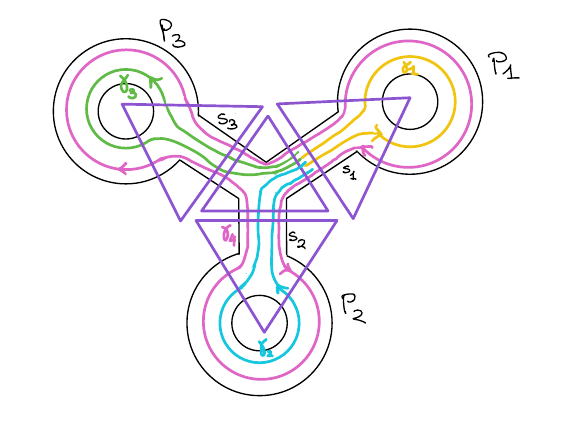}
		\caption{The fat graph of a sphere with four boundary components, its dual triangulation and the loops corresponding to the fundamental group generators.}\label{fig:2dimloops}
	\end{figure}
For example, choosing the base point of all loops on the edge labeled by $s_1$, the transport matrix factorization can be read off from \Cref{fig:2dimloops}: denoting by {$O$} the matrix corresponding to the ochre loop, {$B$} the matrix of the blue loop, {$G$} the one of the green loop and ${P}$  that of the pink one, we have
\begin{equation}\label{2loopfactiorization}
\begin{aligned}
    {O}&=-\ S \ T_3^{(r)} \ S \ T_2^{(r)} \ S,\\
    {B}&=- \ T_2^{(c)} \ S \ T_3^{(d)}\ S \ T_2^{(d)} \ S \ T_2^{(c)-1},\\
    {
    G}&=- \ T_1^{(c)-1} \ S \ T_3^{(l)} \ S \ T_2^{(l)} \ S \ T_1^{(c)},\\
    {P}&=\ T_1^{(c)-1} \ S \ T_2^{(l)-1} \ S \ T_3^{(l)-1} \ S  \ T_3^{(c)-1} \ S \ T_2^{(d)-1} \ S  \ T_3^{(d)-1} \ S  \ T_2^{(c)-1} \ S  \  T_2^{(r)-1} \ S \ T_3^{(r)-1} \ S,
\end{aligned}
\end{equation}
where $T_i^{(a)}$ stands for the transport matrix $T_i$ in the Fock--Goncharov variables $Z_{\alpha}^{(a)}$ of the triangle $(a)$, the labels indicating the corresponding positions of the triangles: l for left, r for right, d for down and c for central. 

With the choice of starting and ending point between the edge $23$ of the right triangle and the edge $12$ of the central one, all pinning variables are automatically amalgamated except the ones corresponding to these two edges. However, taking traces, we can cyclically reorder in order to amalgamate. For example
\begin{equation*}
    \begin{split}
    \opn{Tr}(OB) & = \opn{Tr}(\ S \ T_3^{(r)} \ S \ T_2^{(r)} \ S  \ T_2^{(c)} \ S \ T_3^{(d)}\ S \ T_2^{(d)} \ S \ T_2^{(c)-1} )\\
    &    = \opn{Tr}\left(\ \prod_{k=1}^{n-1}\Big[H_{n-k}(Z_{0,n-k,k}^{(r)})\Big] \dots \prod_{k=1}^{n-1}H_k(Z_{k,0,n-k}^{(r)})
 \ S \ T_2^{(r)} \ S  \ T_2^{(c)} \ S \ T_3^{(d)}\right.\\
 &\left.\ S \ T_2^{(d)} \ S \
\prod_{k=1}^{n-1}H_k(Z_{0,n-k,k}^{(c)^{-1}})\dots \prod_{k=1}^{n-1}H_{n-k}(Z_{n-k,k,0}^{(c)^{-1}}) S\right)=\\
&= \opn{Tr}\left(\ \prod_{k=1}^{n-1}H_{n-k}(Z_{n-k,k,0}^{(c)}) \prod_{k=1}^{n-1}\Big[H_{n-k}(Z_{0,n-k,k}^{(r)})\Big] \dots \prod_{k=1}^{n-1}H_k(Z_{k,0,n-k}^{(r)})
 \ S \ T_2^{(r)} \ S \right.  \\
 &\left.T_2^{(c)} \ S \ T_3^{(d)} \ S \ T_2^{(d)} \ S \
\prod_{k=1}^{n-1}H_k(Z_{0,n-k,k}^{(c)^{-1}})\dots S\right),
  \end{split}
\end{equation*}
so that the pinning variables of  the edge $23$ of the right triangle and the edge $12$ of the central one become amalgamated as well.
\end{example}

In the case of surfaces with bordered cusps instead, some edges remain un-amalgamated as in \Cref{ex:non-amalgamated}. Indeed, in the definition of bordered cusped  Teichm\"uller space 
$$
\opn{Hom}'(\pi_1(S,P), \mathbb P SL_n(\mathbb C))/U_P,
$$
the quotient by mixed multiplication by unipotent radicals appears. If we generate the unipotent radicals by lower triangular elements, then the natural invariant operation needed is the one picking the element $1n$ in the matrix. Therefore in the $n\times n$ case, we define 
 $$
 K=\left(\begin{array}{cccc}
   0&\dots  &0 & 0 \\
   0&0& \dots   & 0 \\
     -1 &0&\dots &0 \\ 
 \end{array}\right),
 $$
and again we take $\opn{Tr}_K$. Because  $\opn{Tr}_K$ is no longer invariant under cyclic permutation,  for paths that have different starting and ending points, we can no longer amalgamate the edges at the start and at the end. Therefore some pinning variables remain.

Looking back at the chewing-gum moves of Section~\ref{se:confl} from the perspective of Fock--Goncharov theory, we can now recognize them as, in a precise sense, the \emph{inverse of amalgamation}.

Recall from Section~\ref{se:glue} that amalgamation arises whenever two triangles of the ideal triangulation are glued along an edge. Before gluing, each side of each triangle (name them left and right) carries its own collection of $n-1$ pinning variables $Z^{(L)}_{\alpha}$, $Z^{(R)}_{\alpha}$. As we showed in \Cref{ex:2tri}, when the two sides are identified the resulting transport matrix depends only on the {amalgamated variables} $Z_\alpha:=Z^{(L)}_{\alpha} Z^{(R)}_{\alpha}$. The $2(n-1)$ pinning variables on the two separated sides collapse to $n-1$ amalgamated variables on the now internal edge of the fat-graph. 

The chewing-gum move performs exactly the opposite operation. As we saw in Section \ref{se:confl}, see Examples \ref{ex:s3toinf}, \ref{ex:s3toinf-1} and \ref{ex:PV}, to merge two holes of $\Sigma_{g,s}$ we select one internal edge of the fat-graph carrying a shear coordinate $s$ in the rank-$2$ case
and break it open via the substitution
\begin{equation}\label{eq:chewing-gum-substitution}
s\;=\; k_1 - \log(\epsilon) + k_2, \qquad \epsilon \to 0.
\end{equation}
In the rank $n$ case, each internal edge carries $n-1$ amalgamated Fock--Goncharov variables $Z_1,\dots,Z_{n-1}$. In the chewing-gum move, we set
\begin{equation}\label{eq:chewing-gum-substitution1}
Z_i\;=\;\frac{1}{\epsilon} Z_i^{(L)} Z_i^{(R)}, \qquad \epsilon \to 0.
\end{equation}

Therefore, while the amalgamation glues two open edges of the fat-graph into a closed one and combines pinning variables into their amalgamated products, the chewing-gum move tears a closed edge apart into two open ones and decouples the amalgamated variable back into independent pinning variables.

This inverse relationship clarifies the position of the bordered cusped Teichm\"uller space within Fock--Goncharov theory. The higher Teichm\"uller space of a closed-boundary surface $\Sigma_{g,s}$ is described by a fat-graph all of whose edges are closed, hence carry only amalgamated variables; the bordered cusped Teichm\"uller space 
$$
\opn{Hom}'(\pi_1(S,P), \mathbb P SL_n(\mathbb C))/U_P,
$$
is obtained by performing $m$ (non-degenerate) chewing-gum moves, each of which un-amalgamates one closed edge into two open edges carrying the full complement of $2(n-1)$ pinning variables. This corresponds to the Goncharov--Shen moduli space of pinnings as proved in \cite{FMN}.

\section*{Appendix}\label{app:hyp-geo}
\addcontentsline{toc}{section}{Appendix}

In these lecture notes, we will only use the Poincar\'e upper half-plane 
\begin{align*}
	\mathbb H &:= \{z \in \mathbb C| \operatorname{Re}(z) > 0\},\\
	\partial \mathbb H &:= \{z \in\mathbb C|\operatorname{Im}(z) = 0\} \cup \{ \infty \},
\end{align*}
which is endowed with the hyperbolic metric 
$$
d s^2=\frac{|d z|^2}{\operatorname{Im}(z)^2}.
$$ 
The geodesics in $\mathbb H$ are either semi--circles with center on the real axis or half--lines
parallel to the imaginary axis.

\subsection{Hyperbolic distance in $\mathbb H$}

Given any two points $P,Q\in\mathbb H$,  their hyperbolic distance is
$$
d_{\mathbb H}(P,Q):=\int_\gamma\frac{\sqrt{d x^2+d y^2}}{y^2},
$$
where $\gamma$ is the unique geodesic through $P$ and $Q$.

Here we list a few useful relations between hyperbolic and Euclidean distances:
the first 
comes from Proposition 2.14 in \cite{Ser}:
\begin{equation}\label{eq:hyp-eu1}
	\tanh{\frac{d_{\mathbb H}(P,Q)}{2}}=
	\frac{|P-Q|}{|P-\overline Q|}
\end{equation}
the second can be derived from \eqref{eq:hyp-eu1} with some work:
\begin{equation}\label{eq:hyp-eu}
	\left| \sinh{\frac{d_{\mathbb H}(P,Q)}{2}}\right|^2=
	\frac{|P-Q|^2}{4 \im(P)\im(Q)}
\end{equation}

Finally, using some simple Euclidean geometry, one can derive from \eqref{eq:hyp-eu1} a formula in terms of cross ratio:
$$
d_{\mathbb H}(P,Q)=  \ln\textrm{cr}(P',Q';Q,P),
$$
where $P',Q'$ are the end points of the geodesic between $P$ and $Q$, see \Cref{fig:4-points}.

\begin{figure}[!htb]
	\centering
	\includegraphics[width=.5\textwidth]{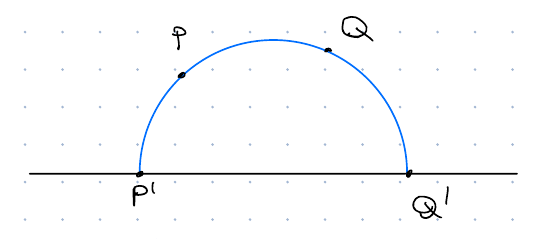}
	\caption{Points $P',Q'$ on the absolute.}\label{fig:4-points}
\end{figure}

Because the points $P',P,Q,Q'$ lie on the same circle, their cross ratio is a real number.

We remind the reader that
the cross ratio of $4$ finite points is defined by
$$
\textrm{cr}(P',Q';Q,P):= \frac{P'-Q}{P'-P}\frac{Q'-P}{Q'-Q},
$$
and if one of the points is infinite, we simply take the limit, for example
$$
\textrm{cr}(P',\infty;Q,P):=\frac{P'-Q}{P'-P}.
$$
Therefore, in the case of points placed on a vertical geodesic, because $P',P,Q$ have the same real part, assuming $P_y:=\text{Im}(P)<\text{Im}(Q)=:Q_y$, we obtain
$$
d_{\mathbb H}(P,Q)=  \ln\frac{Q_y}{P_y}.
$$

\begin{definition}
	A hyperbolic circle of center $z_0$ and radius $\rho$ is the locus of points
	$$
	\{z\in\mathbb H| d_{\mathbb H}(z,z_0) =\rho\}.
	$$
\end{definition}

\begin{lemma}
	Euclidean circles in the upper half plane are hyperbolic circles with a shifted center. If $z_1$ denotes the center of the Euclidean circle and $r$ its radius, the corresponding center $z_0$ of the hyperbolic circle is 
	$z_0=\textrm{Re}(z_1) + i \sqrt{\im(z_1)^2-r^2}$ and the hyperbolic 
	radius is $\rho=\operatorname{arctanh}\left(\frac{r}{\im(z_1)}\right)$. 
\end{lemma}

\begin{proof}
	First let us prove that the hyperbolic distance between any point $P$ on the Euclidean circle and $z_0$ only depends on $r$ and $\im(z_1)$.
	Apply \eqref{eq:hyp-eu1} to $P=z_1+ r \exp{i\theta}$ and $Q=z_0$:
	\begin{equation}
		\begin{split}
			\tanh{\frac{d_{\mathbb H}(P,Q)}{2}}=
			\frac{|z_1+r \exp{i\theta} - {z_0}|}{|z_1+r \exp{i\theta} - \overline{z_0}|}=
			\frac{\left|i \im(z_1)- i \sqrt{\im(z_1)^2-r^2}+r \exp{i\theta}\right| }{\left| i \im(z_1)+ i \sqrt{\im(z_1)^2-r^2}+r \exp{i\theta}\right|}=\\
			\quad=\sqrt{\frac{\im(z_1)- \sqrt{\im(z_1)^2-r^2}}{\im(z_1)+ \sqrt{\im(z_1)^2-r^2}}}.\\
		\end{split}
	\end{equation}
	This proves that Euclidean circles are hyperbolic circles. Let us now compute the hyperbolic radius. To do this, we consider two points $P_1,P_2$ that lie on the intersection between the Euclidean circle and the vertical geodesic through $z_0=\textrm{Re}(z_1) + i \sqrt{\im(z_1)^2-r^2}$. Then $P_1=z_1-i r$, $P_2=z_1+i r$ and applying \eqref{eq:hyp-eu1}  we obtain:
	$$
	\tanh\frac{d_{\mathbb H}(P_1,P_2)}{2}=\frac{r}{\im(z_1)}.
	$$
	This concludes the proof. \end{proof}

\begin{example}
	The hyperbolic circle of center $i$ and radius $\rho=2$ intersects the imaginary axis at $i e^{2}$ and $i e^{-2}$, see \Cref{fig:hyp-circle}.
	
	\begin{figure}[!h]
		\centering
		\includegraphics[width=.5\textwidth]{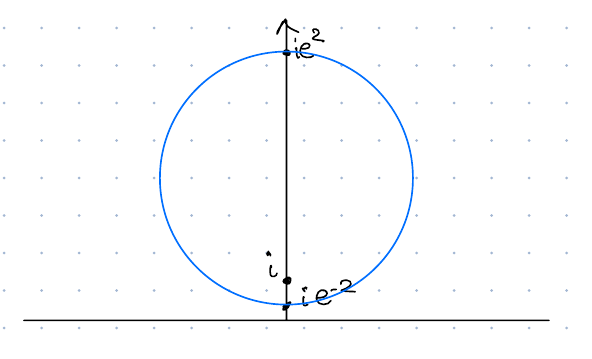}
		\caption{The hyperbolic circle of center $i$ and radius $2$.}\label{fig:hyp-circle}
	\end{figure}
\end{example}

\subsection{Clines}

\begin{definition}
	A \textit{cline} is a Euclidean circle or line. Any cline can be described algebraically as the zero locus of the following expression
	$$
	c z \overline{z} + \alpha z + \overline{\alpha}\overline{z} + d=0
	$$
	where $\alpha$ is a complex constant, and $c,d$
	are real numbers.
\end{definition}

If $c=0$, the cline is a line, while for $c\neq 0$ it is a circle of center $z_0$ and radius $r$ where
$$
z_0=\left(-\frac{\Re{\alpha}}{c},\frac{\Im{\alpha}}{c}\right), \quad
r=\sqrt{\frac{\alpha \overline\alpha-c d}{c^2}}
$$

\begin{definition}
	A horocycle is a Euclidean circle in $\mathbb H$ tangent to the absolute.
\end{definition}

There are two types of horocycles in $\mathbb H$: the ones based at infinity, which are horizontal lines (so these are infinite Euclidean circles tangent at $\infty$) and the ones based at points on the real line, which are standard Euclidean circles tangent to the real line.
Given a horocycle, any geodesic originating at the point of tangency (interpreted broadly including the point at infinity in the case of a horizontal line) intersects the horocycle orthogonally in $\mathbb H$.

\begin{definition}
	Given a geodesic $g$ in $\mathbb H$, we say that a cline $l$ is hyperbolically parallel to $g$ if its points are at constant hyperbolic distance from $g$.
\end{definition}

\begin{example}\label{ex:parallel}
	Consider a vertical geodesic $g$ that crosses the real line at a point $x_0$. Then any Euclidean line $l_0$ originating at $x_0$ is hyperbolically parallel to $g$, namely all points on $l_0$ have the same distance from $g$. This is displayed in   
    \Cref{fig:hyp-parallel}.
	
	\begin{figure}[!h]
		\centering
		\includegraphics[width=\textwidth]{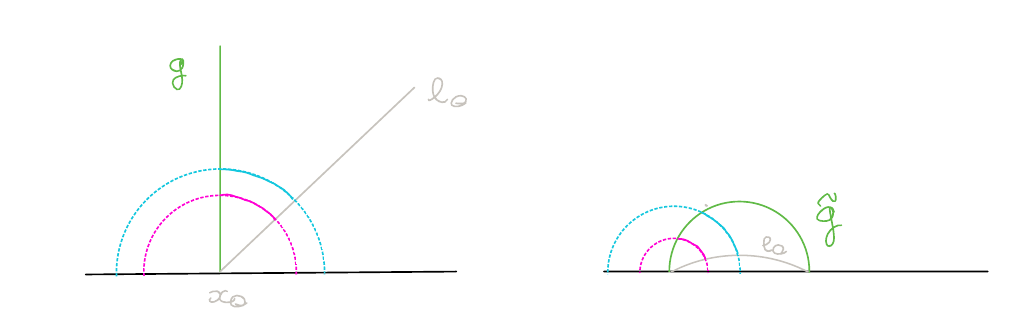}
		\caption{On the left, given a vertical geodesic $g$, the line $l_0$ is hyperbolically parallel to it. Namely any two geodesic arcs (in solid pink and solid cyan) orthogonal to $g$ and $l_0$ have the same length. On the right we display the same for a geodesic $\tilde g$ given by a half circle.}\label{fig:hyp-parallel}
	\end{figure}
\end{example}


\subsection{Action of $\mathbb P SL_2(\mathbb R)$ on $\mathbb H$}
The set of analytic bijections $\mathbb H\to\mathbb H$ coincides with $\mathbb P SL_2(\mathbb R)$ and acts by 
fractional linear transformations, namely  conformal 
isometries of $\mathbb H$ (see for example \cite{Ser}). In particular, fractional linear transformations map clines to clines and horocycles to horocycles. 

We will denote by  $\gamma$ a generic matrix in $SL_2(\mathbb R)$ and by $\gamma(z)$ the corresponding fractional linear transformation:
$$
\gamma=\left(\begin{array}{cc}
	a & b \\
	c& d
\end{array}\right),\quad \gamma(z):=\frac{a z+b}{c z+ d}, \quad a d - b c=1.
$$
Note that operations on matrices $\gamma$ such as multiplication, conjugation, etc. can be mapped to operations on $\gamma(z)$ as compositions. However, the action of $\gamma(z)$ is obviously non-linear.

\begin{example}
	Affine transformations
	$$
	\gamma(z)=z+b, \quad \gamma=\left(\begin{array}{cc}
		1 & b \\
		0& 1
	\end{array}\right)
	$$
	and inversions
	$$
	\gamma(z)=-\frac{1}{z}, \quad \gamma=\left(\begin{array}{cc}
		0 & -1 \\
		1& 0
	\end{array}\right),
	$$
	are examples of fractional linear transformations.
\end{example}

\begin{example}\label{ex:parallel1}
	\noindent Since any geodesic $\tilde g$ can be mapped to a vertical geodesic $g$ by an element of $\gamma\in\mathbb P SL_2(\mathbb R)$, then if $l_0$ is hyperbolically parallel to $g$, then $\gamma^{-1}(l_0)$ is 
	hyperbolically parallel to $\tilde g$. Note that since $l_0$ is a Euclidean line, then
	$\gamma^{-1}(l_0)$ must be a cline. In particular if $\tilde g$ is a half-circle geodesic, then  $\gamma^{-1}(l_0)$ is the portion in $\mathbb H$ of a Euclidean circle through the points of intersection of $\tilde g$  with the real axis, see right hand side of \Cref{fig:hyp-parallel}.
\end{example}

\begin{definition}
	Given $\gamma\in SL_2(\mathbb R)$, the corresponding element in 
	$\mathbb P SL_2(\mathbb R)$ is called 
	\begin{equation*}
		\begin{array}{cc}
			\text{hyperbolic} & \text{if }|\text{Tr}(\gamma)|   
			>2, \\
			\text{parabolic} & \text{if }|\text{Tr}(\gamma)|  =2,\\
			\text{elliptic} & \text{if }|\text{Tr}(\gamma)|  <2. \\
		\end{array}
	\end{equation*}
\end{definition}

\begin{lemma}\label{lm:hyp-par-el}
	Any $\gamma\in \mathbb PSL_2(\mathbb R)$ admits at most two fixed points. In particular
	\begin{align*}
		\gamma \text{ hyperbolic }\Rightarrow         \text{ two fixed points in } \partial\mathbb H,\\
		\gamma \text{ parabolic }\Rightarrow         \text{ one fixed point in } \partial\mathbb H,\\
		\gamma \text{ elliptic }\Rightarrow         \text{ one fixed point in } \mathbb H
	\end{align*}
\end{lemma}

\begin{proof}
	Given an element $ \gamma(z)=\frac{a z+b}{c z+ d}$, observe that if $c=0$, $\gamma$ cannot be elliptic because we need $a d - b c=1$. If it is parabolic, then it is a translation, $\gamma(z)=z+b$, the only fixed point being $\infty$. If it is hyperbolic, then it admits the following fixed points:
	$$
	x_1=\frac{b}{d-a},\quad x_2=\infty.
	$$
	Let's now assume $c\neq 0$, then imposing that $x_1,x_2$ are fixed points, we obtain 
	$$
	x_{1} = \frac{a-d-\sqrt{(a+d)^2-4}}{2 c},\quad
	x_{2} = \frac{a-d+\sqrt{(a+d)^2-4}}{2 c}.
	$$
	For parabolic elements, $x_1=x_2\in\mathbb R$, for hyperbolic elements $x_1\neq x_2$ and both are real, while for elliptic elements, $x_1=\bar x_2$, so that only one of these is in $\mathbb H$.
\end{proof}

\subsubsection{Action of elliptic elements on $\mathbb H$}
Any elliptic element $\gamma_e$ acts by hyperbolic rotation in a neighborhood around its fixed point in $\mathbb H$. To see this, denote the fixed point by $z_0$. For any other point $P\in \mathbb H$, let $\rho$ be the hyperbolic distance between $P$ and $z_0$. Then $\gamma_e(P)$ must lie on the hyperbolic circle of center $z_0$ and radius $\rho$ as $\gamma_e$ is an isometry. Moreover, because $\gamma_e$ is also conformal, the angle between any two points on the same circle centered at $z_0$ must be preserved.

\begin{example}
	Consider the following
	$$
	\gamma_e(z)=\frac{z+\sqrt{3}}{1-\sqrt{3}z}
	$$
	then its fixed point in $\mathbb H$ is $i$. 
	The action of $\gamma_e$ is given in \Cref{fig:elliptic}.
	\begin{figure}[!htb]
		\centering
		\includegraphics[width=.5\textwidth]{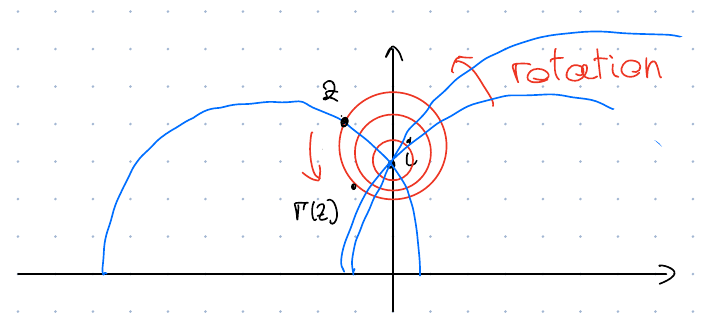}
		\caption{The action of the elliptic element $\gamma_e$.}\label{fig:elliptic}
	\end{figure}
\end{example}

\subsubsection{Action of parabolic elements}\label{suseA:para}

For parabolic elements, any horocycle at the fixed point is invariant. Indeed, we can always apply a transformation in $\mathbb PSL(2,\mathbb R)$ to map a parabolic element to one of the form $\gamma_p(z)= z+r$, where $r\in\mathbb R$. Horocycles based at the fixed point $\infty$ are horizontal lines which are simply translated under the action of $\gamma_p$. 
This shows that parabolic elements generically act either as translations or by horolations, i.e. rotations along the horocycles at the fixed points. 

\begin{lemma}\label{lm:hor-par}
  Let $x \in \mathbb{R}$, then there exists a one-parameter subgroup $U_x \subset {\mathbb PSL}_2(\mathbb{R})$
 of parabolic elements (together with the identity)
 \[
    U_x = \left\{
        P_t :=
        \begin{pmatrix} 1 - tx & tx^2 \\ -t & 1 + tx \end{pmatrix}
        : t \in \mathbb{R}
    \right\},
\]
each of which fixes $x$
and preserves every horocycle tangent to $\mathbb{R}$ at $x$. 
\end{lemma}

\subsubsection{Action of hyperbolic elements on $\mathbb H$}
Since fractional linear transformations map clines to clines, any Euclidean circle through two fixed points of a hyperbolic element $\gamma$ is preserved.
This gives  qualitative information about  $\gamma(z)$; we take the unique cline containing $z$ and the two fixed points, then $\gamma(z)$ belongs to the same cline, as in \Cref{fig:clines}.
\begin{figure}[!htb]
	\centering
	\includegraphics[width=.5\textwidth]{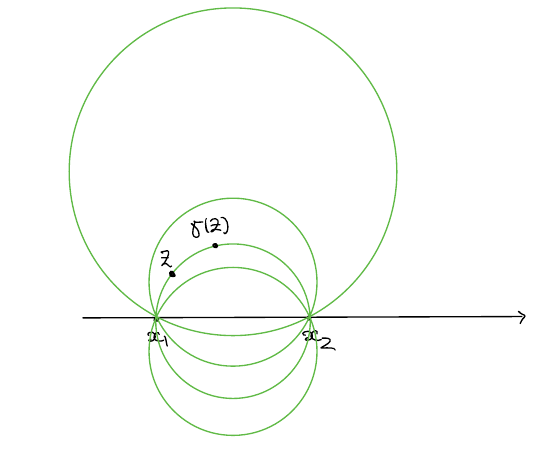}
	\caption{Various clines through $x_1$ and $x_2$, fixed points under $\gamma$.}\label{fig:clines}
\end{figure}

In particular, this implies that the geodesic connecting $x_1$ and $x_2$ is invariant under the action of $\gamma$. This geodesic is called the \textit{invariant axis}.

Given a hyperbolic element $\gamma(z)=\frac{a z+b}{c z+ d}$ with $c\neq 0$, and fixed points $x_1, x_2$, the cross ratio
\begin{equation}\label{eq:conf-rel}
	\textrm{cr}(\gamma(z),z;x_1,x_2):= \frac{\gamma(z)-x_1}{\gamma(z)-x_2}\frac{z-x_2}{z-x_1},
\end{equation}
is real, does not depend on $z$ and coincides with the ratio of the eigenvalues of $\gamma$:
\begin{equation}\label{eq:rr}
	\textrm{cr}(\gamma(z),z;x_1,x_2)= \frac{a+d+\sqrt{(a+d)^2-4}}{a+d-\sqrt{(a+d)^2-4}}=
	\frac{c x_2+d}{c x_1+d}.
\end{equation}
For $c=0$, one of the fixed points of $\gamma$ is $\infty$. In that case, 
\begin{equation}\label{eq:conf-rel1}
	\textrm{cr}(\gamma(z),z;x_1,\infty):= \frac{\gamma(z)-x_1}{z-x_1}=a^2.
\end{equation}

The action of hyperbolic elements on geodesics is very simple: geodesics are mapped to geodesics. The intersection points of a given geodesic with the absolute are mapped to the intersection points of the image geodesic with the absolute.

Let us see how hyperbolic elements act on horocycles. By the action of $\mathbb PSL_2(\mathbb R)$ we may bring any hyperbolic element to the form  $\gamma(z)= a z +b$ and 
assume that the horocycle is tangent to the real axis at the point $b/(1-a)$. Take the intersection of the horocycle with the vertical geodesic originating at $b/(1-a)$. This will be a point of the form 
$b/(1-a)+ i \alpha$. Then $\gamma$ will map $b/(1-a)$ to itself and 
$b/(1-a)+ i \alpha$ to $b/(1-a)+ i a \alpha$. Hence, if $a>1$, the horocycle dilates, and if $a<1$ it contracts.


\subsection{Hyperbolic polygons}
A polygon in $\overline{\mathbb H}$ is a closed  region whose boundary is given by a finite set of geodesic segments called edges. Let $k$ denote the number of edges and denote by $\alpha_1,\dots,\alpha_k$ the internal angles, then the Gauss–Bonnet formula gives us the area of the polygon as
\begin{equation}\label{eq:GB}
	(k-2)\pi -\sum_{j=1}^k \alpha_j.
\end{equation}
A polygon is called ideal if all of its vertices belong to the absolute $\partial\mathbb H$. In this case, all internal angles $\alpha_j$ are $0$.

\subsection{Fuchsian groups and their fundamental domains}

Discrete subgroups of $\mathbb PSL(2,\mathbb R)$ are called Fuchsian. In $\mathbb PSL(2,\mathbb R)$, any discrete group $G$ acts properly discontinuously, namely for any $K\subset \mathbb H$ compact,
the set
$$
\{g\in G| g(\dot K) \cap \dot K\neq\emptyset \},
$$
where $\dot K$ denotes the interior of $K$, is finite. 

Note that a Fuchsian group acts freely on $\mathbb H$ iff it does not contain elliptic elements. 

\begin{definition}\label{def:fun-d}
	If $G$ is a Fuchsian group, a subset $R\subset\mathbb H$ is a fundamental domain for $G$ if
	\begin{enumerate}
		\item $\gamma(R)\cap R = \emptyset\, \forall \gamma\in G\setminus\{id\}$,
		\item $\forall z\in\mathbb H, \, \exists \gamma\in G$ such that $\gamma(z)\in R$. 
	\end{enumerate}
\end{definition}

\begin{example}\label{ex:gamma1}
	Consider the following hyperbolic element
	$$
	\gamma_1(z):= \frac{z}{3}-2,
	$$
	it has fixed points $ q^{(1)}_1=\infty, \, q^{(1)}_2=-3$, and the invariant axis is the vertical line connecting $q^{(1)}_1$ to $\infty$, see \Cref{fig:gamma1}.
	\begin{figure}[!htb]
		\centering
		\includegraphics[width=.7\textwidth]{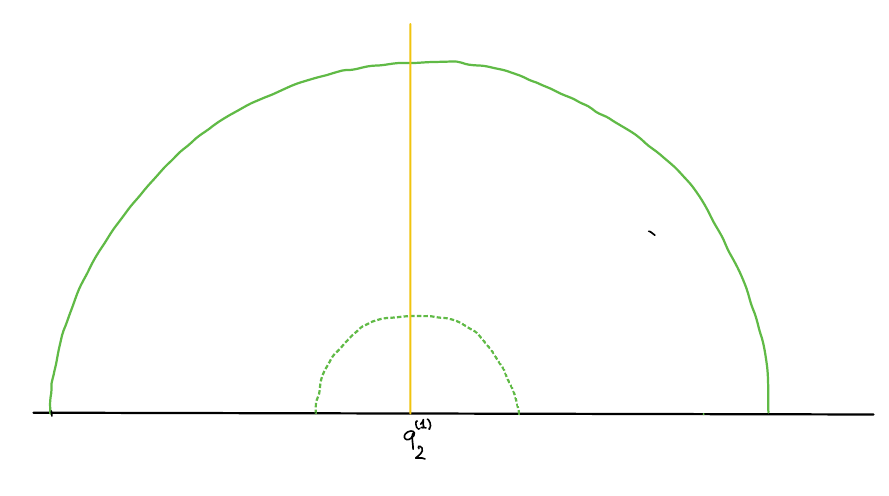}
		\caption{The action of $\gamma_1$.}\label{fig:gamma1}
	\end{figure}
	To pick a fundamental domain, we select any geodesic orthogonal to the invariant axis, for example the geodesic $g_1$ centered at $q^{(1)}_1$ with radius $\sqrt{14}$. To see how this is mapped by $\gamma_1$, we calculate the intersection points of $g_1$ with the absolute and calculate their image under $\gamma_1$. We obtain 
	$\gamma_1(g_1)=g_2$, where $g_2$ is a geodesic centered at $q_2^{(1)}$ with radius $\sqrt{14}/3$. So $\gamma_1$ maps all points under $g_1$ to points in the strip between $g_1$ and $g_2$, moreover, it maps the points in this strip to points under the geodesic $g_2$. 
	Therefore the fundamental domain is the strip contained between the two green geodesics (dashed and solid).
	Observe that points that are on the left (right) of the invariant axis will remain on the left (right) under the action of $\gamma_1$.
\end{example}

\begin{example}\label{ex:gamma2}
	Consider the following hyperbolic element
	$$
	\gamma_3(z):=-\frac{ z+2}{3 z+4},
	$$
	then its fixed points are $q^{(3)}_1=-1,\, q^{(3)}_2=-\frac{2}{3}$.
	We pick a geodesic that is orthogonal to the invariant axis, for example the purple one in \Cref{fig-gamma3}. Then consider its image under $\gamma_3$, which is given by the dashed purple geodesic. 
	\begin{figure}[!htb]
		\centering
		\includegraphics[width=.9\textwidth]{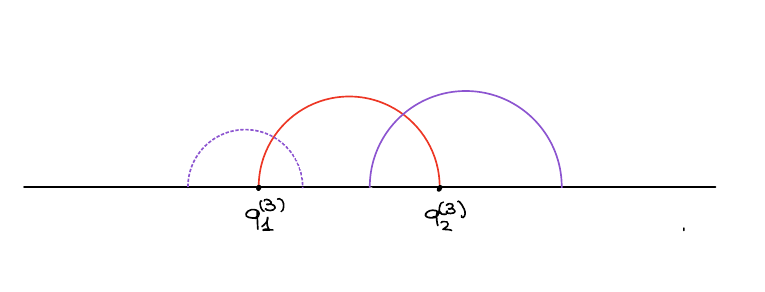}
		\caption{The action of $\gamma_3$.}\label{fig-gamma3}
	\end{figure}
	The region under the solid purple geodesic is mapped to the region above the dashed purple one. A fundamental domain for the hyperbolic element $\gamma_3(z)$ is the portion of the upper half plane above the solid and dashed purple geodesics in \Cref{fig-gamma3}.
\end{example}

\end{document}